\documentclass[aap]{imsart}

\RequirePackage{amsthm,amsmath,amsfonts,amssymb}
\RequirePackage[authoryear]{natbib}%% uncomment this for author-year citations

\usepackage{eucal, amssymb,amsmath,graphicx, amsfonts, latexsym, mathtools}

\startlocaldefs

\numberwithin{equation}{section}
\newtheorem{theorem}{theorem}%[section]
\newtheorem{corollary}[theorem]{Corollary}
\newtheorem{lemma}[theorem]{Lemma}

\newtheorem{definition}[theorem]{Definition}

\def\kn{\ceil{\log m_n}}

\numberwithin{equation}{section}
\numberwithin{theorem}{section}

\def\bwn{\beta_W^{(n)}}
\def\bgn{\beta_G^{(n)}}
\def\thetan{\theta^{(n)}}
\def\thetant{\tilde{\theta}^{(n)}}

\def\lw{\lambda_W}
\def\lg{\lambda_G}

\def\muyn{\mu_Y^{(n)}}

\def\e{{\rm e}}

\def\E{{\rm E}}
\def\P{{\rm P}}
\def\zinf{z_{\infty}}
\def\etan{\eta^{(n)}}
\def\tildetn{\tilde{T}^{(n)}}

\def\tildean{\tilde{A}^{(n)}}
\def\barzn{\bar{Z}^{(n)}}
\def\baran{\bar{A}^{(n)}}

\def\bartn{\bar{T}^{(n)}}
\def\tinfn{T^{(n)}_{\infty}}
\def\tildetinfn{\tilde{T}^{(n)}_{\infty}}
\def\tildetinfnep{\tilde{T}^{(n)}_{\infty,\epsilon}}
\def\tildetinfnL{\tilde{T}^{(n,L)}_{\infty}}
\def\tildetinfnU{\tilde{T}^{(n,U)}_{\infty}}
\def\tildetinfnLep{\tilde{T}^{(n,L)}_{\infty,\epsilon_1}}
\def\tildetinfnUep{\tilde{T}^{(n,U)}_{\infty,\epsilon_1}}

\def\checkzn{\check{Z}^{(n)}}

\def\barabn{\bar{A}^{(n)}_{\bullet}}

\def\rnlam{R^{(n)}_{\lambda}}
\def\rnlami{R^{(n)}_{\lambda i}}
\def\rnlamone{R^{(n)}_{\lambda 1}}
\def\rnlamtwo{R^{(n)}_{\lambda 2}}
\def\rnlammn{R^{(n)}_{\lambda m_n}}

\def\bRn{\boldsymbol{R}^{(n)}}
\def\bRbn{\boldsymbol{R}^{(n)}_{\bullet}}
\def\bRbnT{\boldsymbol{R}^{(n)}_{\bullet T}}

\def\brn{\boldsymbol{r}^{(n)}}
\def\brnT{\boldsymbol{r}^{(n)}_T}
\def\bybn{\boldsymbol{Y}^{(n)}_{\bullet}}
\def\br{\boldsymbol{r}}
\def\bb{\boldsymbol{b}}

\def\bW{\boldsymbol{W}}
\def\bY{\boldsymbol{Y}}

\def\bN{\boldsymbol{N}}

\def\png{p^{(n)}_{\gamma}}

\def\Tn{T^{(n)}}
\def\Zn{Z^{(n)}}
\def\zn{z^{(n)}}
\def\zndash{z^{(n)\prime}}
\def\Yn{Y^{(n)}}
\def\Xn{X^{(n)}}
\def\xn{x^{(n)}}
\def\xndash{x^{(n)\prime}}
\def\Wn{W^{(n)}}
\def\Vn{V^{(n)}}
\def\Un{U^{(n)}}
\def\Gn{G^{(n)}}
\def\An{A^{(n)}}
\def\an{a^{(n)}}
\def\andash{a^{(n)\prime}}
\def\Bn{B^{(n)}}
\def\BB{\mathcal{B}}
\def\BBn{\mathcal{B}^{(n)}}
\def\Sn{S^{(n)}}
\def\SSn{\mathcal{S}^{(n)}}
\def\Mn{M^{(n)}}
\def\taun{\tau^{(n)}}

\def\TnL{T^{(n,L)}}
\def\TnU{T^{(n,U)}}
\def\AnU{A^{(n,U)}}
\def\hatZnL{\hat{Z}^{(n,L)}}
\def\checkZnL{\check{Z}^{(n,L)}}

\def\znmt{Z^{(n)}_T}
\def\znma{Z^{(n)}_A}
\def\znmb{Z^{(n)}_B}
\def\znmc{Z^{(n)}_C}
\def\barznmc{\bar{Z}^{(n)}_C}
\def\tznmt{\tilde{Z}^{(n)}_T}
\def\tznma{\tilde{Z}^{(n)}_A}
\def\tznmb{\tilde{Z}^{(n)}_B}
\def\tanm{\tilde{A}^{(n)}}

\def\Enm{\mathcal{E}^{(n,m_n)}}
\def\En{\mathcal{E}^{(n)}}

\def\Po{{\rm Po}}
\def\Bin{{\rm Bin}}

\def\BPo{\BB_{\Po}}
\def\BBin{\BB_{\Bin}}

\def\pipo{\pi_{\Po}}

\def\BBcheck{\check{\BB}}
\def\BBnhat{\hat{\BB}^{(n)}}
\def\BBncheck{\check{\BB}^{(n)}}

\def\Bnhat{\hat{B}^{(n)}}
\def\Ynhat{\hat{Y}^{(n)}}

\def\Wncheck{\check{W}^{(n)}}
\def\Xncheck{\check{X}^{(n)}}
\def\Yncheck{\check{Y}^{(n)}}

\def\ngam{n^{\gamma}}
\def\ngamd{n^{\gamma'}}
\def\ngamdd{n^{\gamma''}}
\def\ngamddd{n^{\gamma'''}}
\def\ngamone{n^{\gamma_1}}
\def\ngamTwo{n^{\gamma_2}}
\def\ngamthree{n^{\gamma_3}}
\def\ngamfour{n^{\gamma_4}}

\def\hatzng{\hat{z}^{(n)}_{\gamma}}
\def\hatwng{\hat{w}^{(n)}_{\gamma}}

\def\qingamt{q^{(n)}(\gamma,t)}
\def\qLgam{q^{(n)}_L(\gamma)}
\def\qUgam{q^{(n)}_U(\gamma)}

\def\chin{\chi^{(n)}}

\def\zbarnc{\bar{Z}^{(n)}(\cdot)}
\def\abarnc{\bar{A}^{(n)}(\cdot)}

\def\zbarnci{\bar{Z}^{(n)}_i(\cdot)}
\def\abarnci{\bar{A}^{(n)}_i(\cdot)}

\def\zbn{Z^{(n)}_{\bullet}}
\def\abn{A^{(n)}_{\bullet}}
\def\xbn{X^{(n)}_{\bullet}}

\def\rbnone{R^{(n)}_{\bullet 1}}
\def\rbntwo{R^{(n)}_{\bullet 2}}
\def\rbnthree{R^{(n)}_{\bullet 3}}
\def\rbnthreeT{R^{(n)}_{\bullet 3T}}
\def\rbni{R^{(n)}_{\bullet i}}
\def\rbniT{R^{(n)}_{\bullet i T}}
\def\rbnj{R^{(n)}_{\bullet j}}
\def\Psib{\Psi_{\bullet}}

\def\barzbn{\bar{Z}^{(n)}_{\bullet}}
\def\barzan{\bar{A}^{(n)}_{\bullet}}
\def\barxbn{\bar{X}^{(n)}_{\bullet}}

\def\barybnL{\bar{Y}^{(n,L)}_{\bullet}}

\def\tildezbn{\tilde{Z}^{(n)}_{\bullet}}
\def\tildeabn{\tilde{A}^{(n)}_{\bullet}}

\def\hatznc{\hat{Z}^{(n)}_C}

\def\convp{\stackrel{{\rm p}}{\longrightarrow}}
\def\convD{\stackrel{{\rm D}}{\longrightarrow}}
\def\convw{\stackrel{{\rm w}}{\longrightarrow}}

\def\eqdist{\stackrel{{\rm D}}{=}}
\def\stlt{\stackrel{{\rm st}}{\le}}

\def\var{{\rm Var}}

\DeclarePairedDelimiter{\ceil}{\lceil}{\rceil}
\DeclarePairedDelimiter{\floor}{\lfloor}{\rfloor}

\newcommand\bzero{\boldsymbol{0}}

\def\kn{\ceil{\log m_n}}

\newcommand{\hfigwidth}{7.1cm}

\endlocaldefs

\begin{document}

\begin{frontmatter}
\title{SIR epidemics in populations with large sub-communities}
%\title{A sample article title with some additional note\thanksref{T1}}
\runtitle{Epidemics in populations with large communities}
%\thankstext{T1}{A sample of additional note to the title.}

\begin{aug}
%%%%%%%%%%%%%%%%%%%%%%%%%%%%%%%%%%%%%%%%%%%%%%%
%% Only one address is permitted per author. %%
%% Only division, organization and e-mail is %%
%% included in the address.                  %%
%% Additional information can be included in %%

\author[A]{\fnms{Frank} \snm{Ball}\ead[label=e1]{frank.ball@nottingham.ac.uk}},
\author[A]{\fnms{David} \snm{Sirl}\ead[label=e2,mark]{david.sirl@nottingham.ac.uk}}
\and
\author[B]{\fnms{Pieter} \snm{Trapman}\ead[label=e3,mark]{j.p.trapman@rug.nl}}
%%%%%%%%%%%%%%%%%%%%%%%%%%%%%%%%%%%%%%%%%%%%%%
%% Addresses                                %%
%%%%%%%%%%%%%%%%%%%%%%%%%%%%%%%%%%%%%%%%%%%%%%
\address[A]{School of Mathematical Sciences, University of Nottingham, \printead{e1,e2}}

\address[B]{Department of Mathematics, Stockholm University;  current affiliation: Bernoulli Institute,
University of Groningen \printead{e3}}
\end{aug}

\begin{abstract}
We investigate final outcome properties of an SIR (susceptible $\to$ infective $\to$ recovered) epidemic model defined on a population of large sub-communities in which there is stronger disease transmission within the communities than between them. Our analysis involves approximation of the epidemic process by a chain of within-community large outbreaks spreading between the communities. We derive law of large numbers and central limit type results for the number of individuals and the number of communities affected and the so-called severity of the outbreak. These results are valid as the size of communities tends to infinity, with the number of communities either fixed or also tending to infinity. The weaker between-community connections lead to randomness even in the law of large numbers type limit. As part of our proofs we also obtain a new result concerning the rate of convergence of the expected fraction infected in a standard SIR epidemic to its large-population limit.
\end{abstract}

\begin{keyword}[class=MSC]
\kwd[Primary ]{92D30}
%\kwd{???}
\kwd[; secondary ]{60F99}
\kwd{60K99}
\end{keyword}

\begin{keyword}
\kwd{Stochastic epidemic}
\kwd{Final size}
\kwd{Basic reproduction number}
\kwd{Branching process}
\kwd{Coupling}
\kwd{Embedded process}
\kwd{Weak convergence}
\kwd{Central limit theorem}
\end{keyword}

\end{frontmatter}
%%%%%%%%%%%%%%%%%%%%%%%%%%%%%%%%%%%%%%%%%%%%%%
%% Please use \tableofcontents for articles %%
%% with 50 pages and more                   %%
%%%%%%%%%%%%%%%%%%%%%%%%%%%%%%%%%%%%%%%%%%%%%%
%\tableofcontents

%%%%%%%%%%%%%%%%%%%%%%%%%%%%%%%%%%%%%%%%%%%%%%
%%%% Main text entry area:
\section{Introduction}
\label{intro}

Stochastic epidemic models are mathematically interesting as well as being important tools for understanding past outbreaks, and both predicting and mitigating future outbreaks of infectious disease. Of course any mathematical model contains simplifications and assumptions.  In epidemic models which admit any significant mathematical tractability some of the less realistic assumptions revolve around the way the population is structured. Such structure can be introduced into epidemic models in myriad ways, for example by including households (or similar small groups) within which individuals transmit infection more readily than within the population at large, multiple `types' (e.g.\ age classes or other risk groups) of individuals which mix with each other with strengths depending upon the types involved, network models which explicitly model the (social) network structure of the population under consideration, and spatial models, to name but some of the possible structures that may be included.

In this paper we consider a stochastic SIR (susceptible $\to$ infective $\to$ recovered) epidemic within a population consisting of several large communities of individuals. Individuals are relatively strongly connected within the communities and more weakly connected to individuals in different communities. It strongly resembles the `multi-type' SIR epidemic (e.g.\ \citet[Sections~6.1--6.2]{AndBri2000}), but assumes weaker transmission of infection between types (i.e.\ between `communities' in our language). In particular, we assume scalings for the infection parameters such that a `large outbreak' within one community `seeds' outbreaks in other communities with probability strictly between 0 and 1. These outbreaks may be minor and play no significant role in the population at large or they may be major and possibly seed outbreaks in further communities. We investigate final outcome properties of this epidemic model and present results that are in some sense analogous to law of large numbers (LLN) and central limit theorem (CLT) results for simpler epidemic models, but differ on account of the weaker between-community transmission leading to more randomness in the limit random quantities. The method we use to go from the within-community level to the between-community level resembles renormalization methods used in long-range percolation theory \cite[e.g.][]{AizNew1986,DawGor2013}.  %As discussed in more detail when we describe the infection dynamics of
%our model in Section~\ref{sec:modelAndAsymp}, although our model and analysis is in an SIR
%framework it readily extends to the SEIR framework which also includes
%an exposed/latent phase.
For ease of presentation, we present our results within the framework of an SIR model, though, as our focus is on the final outcome of an epidemic, the analysis applies also to SEIR (susceptible $\to$ exposed $\to$ infective $\to$ recovered) models; see Section~\ref{sec:modelAndAsymp}, where this and other extensions are discussed.

Our model may be seen as modelling another version of `clumping' as discussed by~\cite{BlaHouKeeRos2014} which, like the models they investigate, allows for more variability in outcomes than standard homogeneously mixing models.
The model is also closely related to the `epidemic among giants' of~\citet[Section~2.4]{Ball97} and the discussion at the end of section 4.3 of that paper; but that model considers only Reed-Frost epidemic dynamics (see Section \ref{sec:multiComm}, 3rd paragraph) and here we provide much more detailed and complete results. The stochastic multi-type model with weaker transmission between types than within types goes back at least to~\cite{Watson1972}, who cites related models in publications from the late 1950s. The idea of a population of communities with relatively strong within-community links and weaker between-community links is similar in spirit to some motivations for the Stochastic Block Model or planted partition model (see e.g.\ \citet[Example~4.3]{BolJanRio2007} in the probabilistic literature or \citet{Abbe2018} in the networks and community detection literature); though in that context the strength of between- and within-community connections are usually, but not always, assumed to scale with population size in the same way as each other (as is the case in the usual multi-type epidemic model).

There is a substantial literature on CLTs for the final outcome of stochastic SIR epidemics, going back to \cite{NagStartsev68,NagStartsev70}, which consider the final size of the so-called general stochastic epidemic (a single-population model with exponentially distributed infectious periods -- see \citet[Chapter 6]{Bailey75}) when the initial numbers of susceptibles $n$ and infectives $m$ both tend to infinity.  In the practically more relevant situation, when $m$ is held fixed and $n \to \infty$, a major outbreak which infects a strictly positive fraction of the population occurs with non-zero probability only if the basic reproduction number $R_0$ (see~\eqref{R0def} in Section~\ref{sec:singleComm}) is strictly greater than one.  The first CLT to cover this case, conditional on the occurrence of a major outbreak, was given in \cite{Bahr80}. A method of proving CLTs for the final outcome of an SIR epidemic based on an embedding representation was introduced in \cite{Scal85,Scal90}.  This method is very powerful and has been extended to, for example, multitype epidemics \citep{BallClancy93} and both single and multitype epidemics among a community of households (\cite{Ball97} and \cite{BallLyne01}, respectively).  We use a novel extension of this embedding technique to analyse the asymptotic behaviour of our model.  A key difference between our asymptotic analysis and previous CLTs for multitype epidemics is that under the scaling we consider, in the event of a major outbreak, the asymptotic number of communities that experience a major outbreak is \emph{random}, whereas in the previous CLTs it is non-random (in almost all cases \emph{every} type experiences a major outbreak, though \cite{Scal86} considers situations, such as reducible mixing between types, in which a \emph{fixed} number of types do not experience a major outbreak).

An assumption of our model is that every person in a community can infect any other person in that community, and indeed any other person in the population.  Although such an assumption does not seem realistic from the point of view of applications, it is made in the vast majority of epidemic models including most of those used to inform public health policy.
One class of models in which that assumption is not made consists of epidemics on random networks, in which individuals are represented by nodes in a random graph and a person can infect only their neighbours in the graph \cite[e.g.][]{Newman02}.  Rigorous proof of the asymptotic behaviour of epidemics on random graphs as the population size tends to infinity is more difficult than for standard epidemic models, though CLTs have been developed for the final outcome of SIR epidemics on a few random graph models, see e.g.~\cite{Neal06} and \cite{Ball21} for epidemics on multitype Bernoulli random graphs and configuration model random graphs, respectively.  These results are qualitatively similar to corresponding results for standard SIR epidemics and
%SIR epidemics on random graphs display qualitatively similar asymptotic behaviour to standard SIR epidemics under usual `equal-strength' scaling regimes.
we conjecture that similar results to those proved in this paper will hold for epidemics on suitably-defined multi-community random graphs.

The paper is organised as follows. In Section~\ref{model} we specify the underlying model, outline our method of analysis and state our main results with some heuristic arguments supporting their conclusions. In Section~\ref{numerics} we present some numerical examples demonstrating the applicability of our large-population limit theorems to approximating properties of epidemics in finite populations. In Section~\ref{rigorous} we present proofs of our results, and in Section~\ref{conclusions} we make some concluding comments and discuss directions for future work. Appendices~\ref{appA} and~\ref{appB} contain two proofs that we defer from Section~\ref{rigorous} on account of their length.

\subsection{Notation}
\label{notation}

 For a non-negative random variable $X$, we write $Y\sim\Po(X)$ when $\P(Y=k) = \E_X[X^k \e^{-X} / k!]$ for $k=0,1,2,\cdots$; i.e.\ when $Y$ has a Mixed-Poisson distribution with parameter distributed as $X$. For a scalar $\lambda\geq0$ we write $\Po(\lambda)$ for the standard Poisson distribution and for $n \in \mathbb{N}$ and $p \in [0,1]$ we write $\Bin(n,p)$ for the usual binomial distribution. Furthermore, ${\rm N}(\mu,\sigma^2)$ denotes a normal distribution with mean $\mu$ and variance $\sigma^2$, ${\rm N}(\bzero,\Sigma)$ denotes
a zero-mean multivariate normal distribution with variance-covariance matrix $\Sigma$, whose dimension is obvious from the context, and ${\rm Exp}(\theta)$ denotes an exponential distribution with rate $\theta>0$ (and hence mean $\theta^{-1}$).

Throughout the paper we use the standard notation $\stlt$ for the usual stochastic ordering, $\convp$ for convergence in probability, $\convD$ for convergence in distribution and $\eqdist$ for equality in distribution. We also let $\BB(X)$ denote a Galton-Watson branching process with one ancestor (initial individual) and offspring distribution that is $\Po(X)$; for $\lambda\geq0$, $\BB(\lambda)$ is such a process with a fixed rather than random mean for the Poisson offspring distribution.

\section{Model, main results and heuristics}
\label{model}

In this section we define the model that we study and present our main results with heuristic supporting arguments. In Section~\ref{sec:modelAndAsymp} we specify our model for the spread of an SIR epidemic in a population of weakly-connected communities and detail the two asymptotic regimes under which we shall study it (roughly speaking, when the size of each community diverges and the number of communities is either fixed or also diverges). Then in Section~\ref{sec:singleComm} we present results on single-community (or local) outbreaks and in Section~\ref{sec:multiComm} describe how we can view the multi-community epidemic as a sequence of single-community epidemics. In Section~\ref{sec:fixedm} we justify heuristically and present our main result (Theorem~\ref{fixedmasymmain}) in the regime where the number of communities is fixed; and in Section~\ref{sec:divergingm} we similarly motivate and present our main results (Theorems~\ref{manymearlymain}, ~\ref{multifinalmain} and~\ref{multifinalmain1}) in the case of a diverging number of communities.

\subsection{Model and asymptotic analysis regimes}
\label{sec:modelAndAsymp}
Consider a population partitioned into $m+1$ communities, labelled $0,1,\cdots,m$, each having size $n$.  The epidemic is initiated at time $t=0$ by one of the individuals in community $0$ becoming infected, with all other individuals in the population assumed to be susceptible.  Infected individuals have independent infectious periods, each distributed according to a non-negative random variable $I$ having Laplace transform
\begin{equation*}
\phi_I(\theta)=\E\left[\e^{-\theta I}\right] \in (0,\infty] \qquad (\theta \in \mathbb{R}).
\end{equation*}
We assume throughout the manuscript that $E[I^{2+\alpha}]<\infty$ for some $\alpha>0$ and denote the (necessarily finite) mean and variance of the distribution by $\mu_I$ and $\sigma^2_I$ respectively.  Throughout its infectious period, a given infective makes infectious contacts with any given susceptible in its own community at the points of a homogeneous Poisson process having rate $\bwn$ and, additionally, with any given susceptible in the population at the points of a homogeneous Poisson process having rate $\bgn$.  The former are called \emph{local} or \emph{within-community} contacts and the latter are called \emph{global} contacts.  All the Poisson processes describing infectious contacts, whether or not either or both of the individuals involved are the same, as well as the random variables describing infectious periods, are assumed to be mutually independent.  A susceptible individual becomes infective as soon as they are contacted by an infective and recovers at the end of their infectious period.  Recovered individuals are permanently immune to further infection, so they play no further role in the epidemic.  The epidemic ceases as soon as there is no infective present in the population.
We denote this epidemic model with $m+1$ communities, each of size $n$,  by $\mathcal{E}^{(n,m)}$.
%\annote{}{Changed with slight variation on FB's suggested wording.  FGB Further change, and some restructuring and slight modification below as equation (2.1) came too early. }

We note here that the model above does not include a latent period, however since the results in this paper relate to the final outcome of an SIR epidemic they are insensitive to quite general assumptions concerning a latent period \cite[see e.g.][]{Ball86}. Moreover, the results of the present paper (suitably modified) carry over to a model in which very general two-type point processes, representing the times they make global and local contacts, are assigned to infectious individuals; so long as each of the contacts is made with an individual chosen uniformly from the population or community. In particular, all results in this paper apply without change to corresponding SEIR epidemics.  Adjusting our results to allow for
%initial conditions different from those described above (multiple
a fixed number of initial infectives spread in some specified way through the communities is also straightforward, but it is rather lengthy to describe. %\annote{}{Parenthetical comment as FB suggests.} For the sake of brevity we do not give any details.

We are interested in properties of the final size (the total number of individuals infected during the epidemic process) in the communities and in the total population. We obtain asymptotic results for the model where the community size $n$ tends to infinity and the number of communities $m+1$ is either fixed or also tends to infinity.  More formally, we consider sequences of epidemics $\Enm$, indexed by $n$, in which either (i) $m_n=m$ for all $n$, or (ii) $m_n \to \infty$ as $n \to \infty$.  We usually use the notation $m$ or $m_n$ for the number of communities to indicate which of these cases is being considered.

Throughout we assume that as $n \to \infty$,
\begin{equation}
\label{infectionrates}
\bwn n \to \lw \in (0,\infty) \qquad \mbox{ and }  \qquad \bgn n^2 m_n \to \lg \in (0,\infty).
\end{equation}
This implies that as $n \to \infty$, the total rate of within-community contacts of a given individual converges to $\lambda_W$ and the total rate of global contacts made by a given community, if every individual in that community is infected, converges to $\lambda_G$.
A key consequence of this scaling is that a `small' group of infected individuals within a community will with high probability (i.e.\ with probability that tends to 1 as $n \to \infty$) not make any global contacts during their infectious period. Asymptotically, only large groups of infected individuals within a community (i.e.\ groups with final size proportional to the size of the community) have a strictly positive probability of infecting other communities.

We prove some laws of large numbers and central limit theorems for the total number of infected individuals (the \emph{final size}), the number of communities that experience a large outbreak and the sum of all infectious periods of individuals infected during the epidemic (the \emph{severity}).  Under asymptotic regime (i), there is randomness in the law of large numbers limit and the limiting random variable in the central limit theorem has a mixture distribution.
We prove two central limit theorems under asymptotic regime (ii): Theorem~\ref{multifinalmain}(b), in which the mean vector both depends on $n$ and fails to have a readily available useful expression, and Theorem~\ref{multifinalmain1}, in which the mean vector is independent of $n$ and admits a simple form. The latter requires further assumptions, see (C1)--(C5) near the end of Section~\ref{sec:divergingm}; in particular that $m_n$ tends to infinity slower than $n^2$.

\subsection{Single-community epidemics}
\label{sec:singleComm}
The basis for our analysis is the spread of an SIR epidemic in a large homogeneously mixing population, without community structure.
So suppose that $m=0$ and thus that the epidemic only involves community $0$.
The model then reduces to a standard single-population SIR epidemic, which we denote by $\En(\bwn, I)$.  Let $Z_n$ be the total number of individuals infected by the epidemic, including the initial infective, and let $A_n$ denote the severity of the epidemic, i.e.~the sum of the infectious periods of all infectives in the epidemic.  Let $\bar{Z}_n=n^{-1}Z_n$ and $\bar{A}_n=n^{-1}A_n$. The quantity $A_n$ is of interest in its own right, but for our purposes it is an essential ingredient in the \citet{Sell83} like construction which is a key tool in the proofs of all the main results in the paper.
To state the results needed later in the paper we first need to introduce some more terminology and notation.

Let
\begin{equation}
\label{R0def}
R_0=\mu_I\lw
\end{equation}
be the asymptotic expected number of within-community contacts an infected individual makes during their infectious period. So $R_0$ is the limiting basic reproduction number \cite[p.~4]{Diek13}  for the single-community epidemic.
We say that a within-community epidemic is \emph{large}, or that a \emph{major} within-community epidemic occurs, if at least $\log n$ of the individuals in the community are ultimately recovered (i.e.\ are infected at some point during the epidemic).
Now define
\begin{equation}
\pi_W =  \inf\{s >0 : s= \phi_I(\lw(1-s))\}
\label{piw}
\end{equation}
and
\begin{equation}
\zinf =  \sup\{z\geq 0 : z=1-\e^{-R_0 z}\},
\label{zinf}
\end{equation}
noting that $\zinf$ is a function of $(\lw, \mu_I)$ as $R_0=\mu_I \lw$.
We also have $\zinf = 1 + R_0^{-1} W_0(-R_0 {\e}^{-R_0})$, where $W_0(\cdot)$ is the principal branch of the Lambert W function \cite[see e.g.][p.~337]{CorlessEtal1996}.  Note that $\zinf>0$ if and only if $R_0>1$.

The following theorem tells us that, as $n\to\infty$, if $R_0 \leq 1$ then an epidemic stays small with probability $\pi_W=1$. If $R_0>1$, then with probability $1-\pi_W>0$ a major outbreak occurs and the epidemic infects a positive fraction $\zinf$ of the community, while with probability $\pi_W$ the epidemic stays small. Furthermore, a central limit theorem is given for the number of individuals infected during a major outbreak.%; one for the severity of a major outbreak is also available~\citep{BallClancy93} but is not required for this paper.
\begin{theorem}
\label{singleSIR}
\begin{enumerate}
\item Suppose that $R_0 \le 1$.  Then
\begin{equation}
\label{singleminor}
\lim_{n \to \infty} \P(Z_n \ge \log n)=\lim_{n \to \infty} \P(A_n \ge \mu_I\log n)=0.
\end{equation}
\item Suppose that $R_0 > 1$.  Then
\begin{enumerate}
\item with $\pi_W$ as in \eqref{piw},
\begin{equation}
\label{singlemajor}
\lim_{n \to \infty} \P(Z_n \ge \log n)=\lim_{n \to \infty} \P(A_n \ge \mu_I\log n)=1-\pi_W;
\end{equation}

\item furthermore, as $n \to \infty$,
\begin{equation}
\label{singleWLLN}
\bar{Z}_n \mid Z_n \ge \log n \convp z_\infty \qquad \bar{A}_n \mid Z_n \ge \log n \convp \zinf \mu_I;
\end{equation}

\item if we also have $\displaystyle \lim_{n \to \infty}\sqrt{n}(n\bwn-\lw)=0$
then
\begin{equation}
\label{singleCLT}
\sqrt{n}
\left(\begin{array}{c}
\bar{Z}_n-\zinf\\
\bar{A}_n-\zinf \mu_I\\
\end{array}
\right) \mid Z_n \ge \log n \convD \bY \quad \mbox{as }n \to \infty,
\end{equation}
where $\bY \sim {\rm N}(\bzero,\Sigma_{\bY})$ with
\begin{equation}
\label{singleCLTvar}
\Sigma_{\bY}=\frac{\zinf}{\left[1-R_0(1-\zinf)\right]^2}\left\{
(1-\zinf)\left[\begin{array}{cc}
1&\mu_I\\
\mu_I & \mu_I^2
\end{array}
\right]
+
\sigma^2_I \left[\begin{array}{cc}
\lw^2(1-\zinf)^2&\lw(1-\zinf)\\
\lw(1-\zinf)&1
\end{array}
\right]
\right\}.
\end{equation}
\end{enumerate}
\end{enumerate}
\end{theorem}

We do not provide a proof for Theorem~\ref{singleSIR} as the results are already known.
The results concerning $Z_n$ and $\bar{Z}_n$ can be derived from \citet[Theorem 2]{Bahr80} under the stronger assumption that $\phi_I(\theta)$ is finite in an open interval containing the origin, which we also assume for our central limit results in Theorem \ref{multifinalmain1}, but \emph{not} Theorems \ref{fixedmasymmain}--\ref{multifinalmain}, below.
Under the present condition that $I$ has a finite moment of order $2+\alpha$, the results of Theorem \ref{singleSIR} can be obtained using a slight generalisation of \cite{Scal90}.
We also note that in equations~\eqref{singleminor} and~\eqref{singlemajor} we could replace $\mu_I$ with any positive constant. We use the particular constant $\mu_I$ as an aide-memoire that the quantity relates to the severity of (or infectious pressure generated by) the within-community epidemic.

We also note, for future reference, that if $R_0>1$ then $(1-\zinf)R_0<1$. This is intuitively plausible since $(1-\zinf)R_0$ is the effective reproduction number for the within-community epidemic following a large outbreak; if that reproduction number were larger than one then the large outbreak would not have terminated.  A formal proof may be obtained by noting from~\eqref{zinf} that $1-z>\e^{-R_0 z}$ for $z \in (0, \zinf)$ and $1-z<\e^{-R_0 z}$ for $z \in (\zinf,1]$.  Taking $z=1-\frac{1}{R_0}$ in both $1-z$ and $\e^{-R_0 z}$ gives $\zinf > 1-\frac{1}{R_0}$, since $\frac{1}{R_0}>\e^{-(R_0-1)}$.

\subsection{Multi-community epidemics}
\label{sec:multiComm}
%\annote{}{Reformulated this paragraph as suggested. Please check wording/flow.}
We now turn our attention to a population with multiple communities (both for constant $m_n$ and for $m_n \to \infty$ as $n \to \infty$).
A key observation that pervades our whole analysis is that in order to analyse properties of the final outcome of an epidemic (e.g.\ the final size and the severity) we do not have to keep track of the exact times when infections take place or who infects whom.
It is sufficient to keep track of how much infection each individual is exposed to, in the sense of `exposure to infection' of~\cite{Sell83}. (See also e.g.~\cite{Ludwig1975} and~\cite{PelFerFra2008}.)
The final outcome of our model can therefore be analysed by considering, in turn, (i) local epidemics (or outbreaks), which are epidemics restricted to a community, ignoring global contacts and (ii) the global contacts made by those infected in the local epidemics.  The latter may trigger further local epidemics, so (i) and (ii) are iterated until no further local epidemics occur. %\annote{}{FGB has added extra sentence at end.}

Individuals infected in a local epidemic are those individuals infected through a chain of local contacts, starting at an initial infective, which in community 0 is the initial infected individual and in other infected communities is an individual infected through a global contact. From \eqref{infectionrates} and Theorem \ref{singleSIR}(b)(ii), we obtain that if a major local outbreak occurs in one of the communities, the total number of global contacts the infected individuals in that community make to other individuals in the population as a whole converges in distribution to a $\Po(\lambda_G z_{\infty} \mu_I)$ random variable. This gives rise to non-trivial behaviour of the epidemic in the population with multiple communities since $\lambda_G z_{\infty} \mu_I$ is bounded away from 0 and $\infty$. Intuitively we proceed by only considering major local outbreaks in the communities, ignoring the minor local outbreaks and the possible global contacts made by individuals infected in those minor outbreaks. We justify this approach in Lemma~\ref{initialsev} by showing that the contribution of minor outbreaks and any subsequent global infections to the asymptotic fraction infected is negligible. Furthermore, it follows (from the last paragraph of Section~\ref{sec:singleComm}) that, asymptotically, introduction of infection to a community which has already experienced a major outbreak cannot lead to a further major outbreak. Therefore the number of communities which experience major outbreaks can be analysed by assuming that local outbreaks occur instantaneously. The spread of major outbreaks amongst communities can then be thought of as a Reed-Frost epidemic, where individuals correspond to communities and infectious contacts correspond to between-community contacts which seed a major outbreak.

A \emph{Reed-Frost epidemic} \cite[p.~48]{Diek13} is a discrete time epidemic model in which every infected individual at time $t$ ($t \in \mathbb{N}$) is recovered at time $t+1$. Furthermore, given $I(t)$, a susceptible at time $t$ is infected at time $t+1$ with probability $1-(1-p)^{I(t)}$, independently of the states of other individuals, otherwise the susceptible remains so at time $t+1$. Here $I(t)$ is the number of infected individuals at time $t$ and $p$ is the pairwise infectious contact probability per time unit.
In Theorem~\ref{fixedmasymmain} below, we consider the final size $Z^{(m)}_{\rm RF}= Z^{(m)}_{\rm RF}(p_{\rm RF})$ (i.e.\ the number ultimately recovered individuals, including the initially infective one) of a Reed-Frost epidemic in a population with one initial infected individual, $m$ initial susceptible individuals and pairwise infection probability
\begin{equation}
\label{RFprob}
p_{\rm RF}=1-\exp\left(-\lg(1-\pi_W)\zinf \mu_I/m\right).
\end{equation}
Here $\lg\zinf \mu_I$ is the mean number of between-community contacts emanating from a community that experiences a large outbreak. Each such contact is equally likely to be with each of the $m$ other communities and, so long as it has not previously experienced a large outbreak, the infected community will experience a large outbreak with probability $1-\pi_W$. The final size of this Reed-Frost epidemic thus approximates the final number of within-community major outbreaks.
Note that $Z^{(m)}_{\rm RF}$ is distributed as the size of the cluster of a uniformly chosen vertex in $G(m+1,p_{\rm RF})$, the Erd\H{o}s-R{\'e}nyi graph with $m+1$ vertices and edge probability $p_{\rm RF}$ \citep{BarMol1990}.

For $n=1,2,\cdots,$ let $Z^{(n)}_T$ be the total number of individuals, including the initial infective, infected in $\Enm$, $A^{(n)}_T$ be the sum of the infectious periods of those $Z^{(n)}_T$ infectives (i.e.~the severity of $\Enm$) and $Z^{(n)}_C$ be the number of communities experiencing epidemics of size at least $\log n$ (i.e.\ the number of communities experiencing large epidemics). Each of $Z^{(n)}_T$, $A^{(n)}_T$ and $Z^{(n)}_C$ is counted over all $m+1$ communities. These three quantities (or appropriate functions of them) are ultimately the main objects of study in the paper. The methods and results necessarily take rather different forms depending on whether the number of communities $m_n$ is fixed or diverges as $n\to\infty$.

\subsection{Constant number of communities}
\label{sec:fixedm}
If the number of communities is fixed then the Reed-Frost epidemic of major within-community outbreaks described above can be applied directly to analyse the multi-community epidemic. If the initially infectious individual does not initiate a large outbreak within community 0 (which under our assumption of a single initial infective occurs with probability $\pi_W$), then the probability that there is at least one global infectious contact made by the individuals infected during that minor outbreak tends to 0 as $n \to \infty$. If, on the other hand, the local epidemic in community 0 is large, then in the large population limit $Z^{(m)}_{\rm RF}$ describes the number of communities that experience a major outbreak; and the size of each such outbreak can be approximated using Theorem~\ref{singleSIR}. The next theorem formally states the convergence in distribution of $Z^{(n)}_C$ to the final size of a Reed-Frost epidemic and presents an associated conditional law of large numbers and central limit theorem for $(Z^{(n)}_T, A^{(n)}_T)$, the overall final size and severity of the epidemic. Let $\barzn_T=n^{-1}Z^{(n)}_T$ and $\baran_T=n^{-1}A^{(n)}_T$, so $(m+1)^{-1}\barzn$ is the proportion of the population that is infected in $\Enm$.

\begin{theorem}
\label{fixedmasymmain}
Assume that $m$ is fixed and that $R_0>1$.
Suppose that the epidemic in community $0$ infects at least $\log n$ individuals. Then, as $n \to \infty$,
\begin{enumerate}
\item
\begin{equation}
\label{RFlimitmain}
Z^{(n)}_C \convD Z^{(m)}_{\rm RF};
\end{equation}
\item
\begin{equation}
\label{fixedmLLNmain}
\left(\begin{array}{c}
\barzn_T\\
\baran_T\\
\end{array}
\right)
\convD Z^{(m)}_{\rm RF} \left(\begin{array}{c}
\zinf\\
\zinf\mu_I\\
\end{array}
\right);
\end{equation}
\item
\begin{equation}
\label{fixedmCLTmain}
\sqrt{n}\left(\begin{array}{c}
\barzn_T -Z^{(n)}_C\zinf\\
\baran_T -Z^{(n)}_C\zinf \mu_I\\
\end{array}
\right)
\convD \sum_{i=1}^{Z^{(m)}_{\rm RF}} \bY_i,
\end{equation}
where
$\zinf$ is defined in \eqref{zinf},
$\bY_1,\bY_2,\cdots,\bY_{m+1}$ are independent and identically distributed (i.i.d.) ${\rm N}(\bzero,\Sigma_{\bY})$ random variables that are independent of $Z^{(m)}_{\rm RF}$ and $\Sigma_{\bY}$ is defined in~\eqref{singleCLTvar}.
\end{enumerate}

\end{theorem}

\subsection{Diverging number of communities}
\label{sec:divergingm}
We now consider the case when the number of communities $m_n+1 \to \infty$ as $n \to \infty$.
Before we formulate our results regarding this asymptotic regime, we need some further notation.

For the epidemic $\Enm$, let $\hatznc$ denote the total number of communities, including community $0$, that are infected during its course, i.e.\ $\hatznc \in \{1,2,\cdots, m_n +1\}$ is the number of communities that contain at least one non-susceptible individual at the end of the epidemic.
As above, $\Zn_C$ denotes the number of communities that experience an epidemic having size at least $\log n$, $\Zn_T$ denotes the total number of individuals infected, $\An_T$ denotes the severity of the epidemic, $\barzn_T=n^{-1}\Zn_T$ and $\baran_T=n^{-1}\An_T$.  (Note that $\Zn_T$ and $\An_T$ are counted over all $m_n+1$ communities.)
We say that a global epidemic occurs if at least $\log m_n$ communities are affected by the outbreak and let $\Gn = \{\hatznc \ge \log m_n\}$ denote this event.

During the early stages of the  epidemic $\Enm$, the process of infected communities can be approximated by a branching process which assumes that all global contacts are with individuals in previously uninfected communities. We show that for large $n$ this branching process is close to the  branching process $\BBcheck \sim \BB(\lg \bar{A})$ (recall the notation defined in Section~\ref{notation}), where the random variable $\bar{A}$ satisfies
$$\P(\bar{A}=0)=\pi_W=1-\P(\bar{A}=\zinf\mu_I).$$
(The offspring distribution in the initial generation is the same as in subsequent generations since we assume a single initial infective for the epidemic.)
The offspring mean of this branching process is given by
\begin{equation}
\label{equ:Rstar}
R_*=\lg (1-\pi_W) \zinf\mu_I,
\end{equation}
which is the expected number of large within-community outbreaks initiated by a typical large within-community outbreak.
Let $\check{Z}_C$ and
$\check{\pi}_G$ denote respectively the total size (including the initial individual) and extinction probability of $\BBcheck$, so
\begin{equation*}
\check{\pi}_G = \inf\{s\geq 0 : \pi_W+(1-\pi_W)\e^{-\mu_I\lg\zinf(1-s)}=s\}.
\end{equation*}
Hence $\check{\pi}_G = \min(1,\pi_W
-(1-\pi_W) W_0(-R_*e^{-R_*})/R_*)$, where $W_0(\cdot)$ is the principal branch of the Lambert W function.
Standard branching process theory implies that $\check{\pi}_G <1$ if and only if $R_*>1$.
We write $G$ for the event that $\BBcheck$ does not go extinct.

We note from the definition of $R_*$ in~\eqref{equ:Rstar} that $R_0>1$ is a necessary but not sufficient condition for $R_*>1$ (since $R_0 \leq1$ implies $\pi_W=1$).  This is because  in the limit as $n \to \infty$ under the asymptotic regime \eqref{infectionrates}, an epidemic in one community must infect a strictly positive fraction of that community in order to have non-zero probability of spreading to other communities.

Next, to approximate the process of communities that experience large outbreaks,  let $\check{Z}_{C*}$ be the total number of individuals (including the ancestor) in $\BBcheck$ whose value of $\bar{A}$ is $\zinf \mu_I$.  Then $\P\left(\check{Z}_{C*}=0\right)=\pi_W$ and $\check{Z}_{C*} \mid  \check{Z}_{C*}>0$ is distributed as the total size (including the ancestor) of the branching process $\BB_* = \BB(\lg(1-\pi_W)\zinf\mu_I)$, i.e.\ of a Galton-Watson process having offspring distribution that is Poisson with mean $\lg(1-\pi_W)\zinf\mu_I$.

Now we are ready to formulate the first main result regarding $\Enm$ for $m_n \to \infty$. Theorem \ref{manymearlymain} describes the epidemic if only few communities are infected during the epidemic (i.e.\ if a global epidemic does not occur).
\begin{theorem}
\label{manymearlymain}
Suppose that $R_0>1$. Then, as $n \to \infty$,
\begin{description}
\item(a)\ \
\begin{equation}
\label{pglobalmain}
\P(\Gn) \to 1-\check{\pi}_G;
\end{equation}
\item(b)\ \
\begin{equation}
\label{BTlimitmain}
\hatznc \mid \left(\Gn\right)^C \convD \check{Z}_{C} \mid G^C \quad\mbox{and}\quad
\Zn_C \mid \left(\Gn\right)^C \convD \check{Z}_{C*} \mid G^C;
\end{equation}
\item(c)\ \
\begin{equation}
\label{BTLLNmain}
\left(\begin{array}{c}
\barzn_T\\
\baran_T\\
\end{array}
\right) \mid \left(\Gn\right)^C \convD
\left(\begin{array}{c}
\zinf\\
\zinf \mu_I\\
\end{array}
\right) \check{Z}_{C*} \mid G^C
\end{equation}
and
\begin{equation}
\label{BTCLTmain}
\sqrt{n}\left(\begin{array}{c}
\barzn_T -Z^{(n)}_C\zinf\\
\baran_T -Z^{(n)}_C\zinf \mu_I\\
\end{array}
\right) \mid \left(\Gn\right)^C \convD \left(\sum_{i=1}^{\check{Z}_{C*}} \bY_i\right) \mid G^C,
\end{equation}
where the sum is zero if vacuous, $\bY_1,\bY_2,\cdots$ are i.i.d. ${\rm N}(\bzero,\Sigma_{\bY})$ random variables that are independent of $\check{Z}_{C*}$ and $\Sigma_{\bY}$ is given by~\eqref{singleCLTvar}.

\end{description}

\end{theorem}

The probability mass function of $\check{Z}_{C*}$, the total size of $\BB_*$, is easily obtained using Theorem 2.11.2 of \citet{Jage75} and has a Borel distribution with parameter $R_*$.  Specifically,
\begin{equation*}
\P\left(\check{Z}_{C*}=k\right)=
\begin{cases}
   \pi_W&\quad \text{if }  k=0,\\
   (1-\pi_W)\frac{\left(kR_*\right)^{k-1}\e^{-kR_*}}{k!}& \quad \text{if } k=1,2,\cdots .
  \end{cases}
\end{equation*}
Theorem~\ref{manymearlymain} implies that, for large $n$,  if $R_* \le 1$ then $\check{Z}_{C}$ approximates the total number of infected communities $\hatznc$ in $\Enm$, and the total number of individuals infected is approximately distributed as a random sum of $\check{Z}_{C*}$ independent normal random variables, each having mean $n\zinf$ and variance $n\sigma^2_W$, where
\[
\sigma^2_W=\sigma^2_W(\lw,\mu_I,\sigma^2_I)=\frac{\zinf(1-\zinf)+\lw^2\sigma^2_I (1-\zinf)^2 \zinf}{(1-\lw\mu_I (1-\zinf))^2}.
\]
If $R_*>1$ and a global epidemic does not occur then these approximations hold with the distributions of $\check{Z}_{C}$ and $\check{Z}_{C*} \mid \check{Z}_{C*}>0$ being conditioned on extinction of $\BBcheck$ and $\BB_*$, respectively.

In order to describe the size of an epidemic in which many communities are infected during the course of the epidemic we need some further notation.
Let $\barznmc =m_n^{-1}\znmc$, $\tznmt=(nm_n)^{-1}\znmt$ and $\tanm=(nm_n)^{-1}A^{(n)}_T$. Thus $m_n(m_n+1)^{-1} \barznmc$ is the fraction of communities that experience large outbreaks (i.e.\ with size at least $\log n$), $m_n(m_n+1)^{-1} \tznmt=(n(m_n+1))^{-1}\znmt$ is the fraction of individuals in the population that are infected, while $m_n(m_n+1)^{-1} \tanm$ is the average severity per individual (infected or not) in the population.  (We scale by $m_n$, rather than perhaps the more natural $m_{n}+1$, as it fits better with our proofs, which involve analysing a slightly modified epidemic model defined on the $m_n$ communities not containing the initial infective.)

Define the functions $x(\cdot)$, $z(\cdot)$ and $a(\cdot)$ on $[0, \infty)$ by
\begin{equation}
\label{xzadefs}
x(t)=1-\e^{-\lg(1-\pi_W)t}, \quad
z(t)=\zinf x(t) \quad
\mbox{and}
\quad
a(t)=\mu_I \zinf x(t).
\end{equation}
Furthermore, define
\begin{equation}
\label{taudef}
\tau = \sup\{t \geq 0 : t=a(t)\}.
\end{equation}
Thus $\tau=\mu_I\zinf\left(1+W_0(-R_*e^{-R_*})/R_*\right)$,
where $W_0(\cdot)$ is the principal branch of the Lambert W function.  Note that $\tau>0$ if and only if $R_*>1$.

The following epidemic model is used repeatedly in our analysis.
\begin{definition}
\label{defEpsnt}
For $t \ge 0$, let the single-population epidemic $\En(t)$ be defined similarly to $\En(\bwn, I)$ in Section~\ref{sec:singleComm}, except initially there are $S_n \sim \Bin(n, \pi_n(t))$ susceptibles, where $\pi_n(t)=\e^{-\bgn n m_n t}$, and $n-S_n$ infectives.
\end{definition}
Note that in Definition~\ref{defEpsnt}, $t$ is not real time but relates to cumulative time units of external infectious pressure (see Section~\ref{embedding}).  Let $\Zn(t)$ and $\An(t)$ be, respectively, the size and severity of $\En(t)$.  (In Section~\ref{embedding}, an epidemic process $\left\{\left(Z^{(n)}(t),A^{(n)}(t)\right):t \ge 0\right\}$ is defined, with slight abuse of notation, which for each $t \ge 0$ has the same marginal joint distribution of $(\Zn(t), \An(t))$ as the size and severity of $\En(t)$.)
For $t \ge 0$, let
\begin{equation}
\label{xzandefs}
\xn(t) = \P(\Zn(t) \ge \log n), \quad
\zn(t) = n^{-1} \E[\Zn(t)], \quad
\an(t) = n^{-1} \E[\An(t)]
\end{equation}
and
\begin{equation*}
\taun = \inf \{ t>0 : t=\an(t) \},
\end{equation*}
with $\taun=0$ if the set $\{t>0:t=\an(t)\}$ is empty (cf.~equation~\eqref{taudef}).  It also holds that the mean size and severity satisfy $\an(t) = \mu_I \zn(t)$ (see the discussion just after\ \eqref{zn1}; cf.\ the final two equations in~\eqref{xzadefs}). The functions $\xn(\cdot), \zn(\cdot)$ and $\an(\cdot)$ do not have useful explicit forms but, under suitable conditions, $\lim_{n \to \infty}\xn(\cdot)=x(\cdot)$ , $\lim_{n \to \infty}\zn(\cdot)=z(\cdot)$,
$\lim_{n \to \infty}\an(\cdot)=a(\cdot)$ and $\lim_{n \to \infty}\taun = \tau$
(see Lemma~\ref{GC}(a) and Lemma~\ref{tauCgce}).

Finally, let  $\bN$ be a three-dimensional zero-mean normal random vector with variance-covariance matrix given by
\begin{equation}
\label{finalSigmaZ}
\Sigma_{\bN}
%=\frac{x(\tau)(1-x(\tau))}{\left[1-\zinf \mu_I \lg (1-\pi_W)(1-x(\tau))\right]^2}
%\, \bb \bb^\top
=\frac{x(\tau)(1-x(\tau))}{\left[1-R_*(1-x(\tau))\right]^2}
\begin{bmatrix}
1 & \zinf & mu_I \zinf \\
\zinf & \zinf^2 & \mu_I \zinf^2 \\
\mu_I \zinf & \mu_U \zinf^2 & \mu_I^2 \zinf^2
\end{bmatrix}
=\frac{x(\tau)(1-x(\tau))}{\left[1-R_*(1-x(\tau))\right]^2}
\, \bb \bb^\top,
\end{equation}
where
\begin{equation}
\label{bbDef}
\bb=(1,\zinf,\mu_I \zinf)^\top.
\end{equation}

It is easily checked that $\tau=0$ if $R_* \leq 1$, while $\tau \in (0,\infty)$ if $R_*>1$.
In Theorem~\ref{multifinalmain} below we see that $\tau$ is the asymptotic severity per individual in the case when many communities are infected during the epidemic. This theorem provides a law of large numbers and a central limit theorem for $\barznmc $, $\tznmt$ and $\tanm$, conditioned on $\Gn$, the event that at least $\log m_n$ communities are infected during the course of an epidemic.

\begin{theorem}
\label{multifinalmain}
Suppose that $R_* > 1$.
\begin{description}
\item(a)\ \ As $n \to \infty$,
\begin{equation*}
\label{finalWLeq1main}
\barznmc \mid \Gn \convp x(\tau),\quad \tznmt \mid \Gn \convp z(\tau) \quad \mbox{and } \quad \tanm \mid \Gn
\convp \tau.
\end{equation*}
\item(b)\ \ As $n \to \infty$,
\begin{equation*}
\sqrt{m_n}
\left(\begin{array}{c}
\barznmc -\xn(\taun)\\
\tznmt -\zn(\taun)\\
\tanm -\an(\taun)\\
\end{array}
\right) \mid \Gn
\convD \bN.
\end{equation*}
\end{description}
\end{theorem}

The central limit theorem in Theorem~\ref{multifinalmain}(b) is not immediately applicable since, as noted above, useful expressions for the mean vector
$(\xn(\taun), \zn(\taun), \an(\taun))^\top$ are unavailable.  We need to impose further conditions to obtain a central limit theorem in which
the mean vector is replaced by its limit $(x(\tau), z(\tau), a(\tau))^\top$.

Let $f(s)=\phi_I(\lw(1-s))$ $(s \ge 0)$.  Then $f$ is the offspring PGF (probability-generating function) for the Galton-Watson process $\BB=\BB(\lambda_W I)$ that approximates the initial phase of
the single-population epidemic $\En(\bwn, I)$, defined in Section~\ref{sec:singleComm}, and $\pi_W$ is the extinction probability of $\BB$. Note from~\eqref{equ:Rstar} that $R_*>1$ only if $\pi_W<1$, so a necessary condition for $R_*>1$ is that $\BB$ is supercritical.  Let $\hat{f}(s)=
\pi_W^{-1}f(s\pi_W)$ $(s \ge 0)$ be the offspring PGF of the subcritical Galton-Watson process which describes $\BB$ conditioned on extinction (see e.g.~\cite{Daly79}) and define
\begin{equation*}
s_0=\max\{s>1:h=s\hat{f}(h)\mbox{ for some } h \in (1,\infty)\}.
\end{equation*}
Such an $s_0$ exists as $\hat{f}$ is convex and $\hat{f}'(1)<1$; it is the value of $s$ such that the line $g(h)=s^{-1}h$ is tangential to $\hat{f}(h)$.

For our next result we need to impose the following further conditions:\label{conditionsC}
\begin{enumerate}
\item[(C1)] there exists $\delta \in (0,2)$ such that $\lim_{n \to \infty}m_n n^{\delta-2}=0$,
\item[(C2)] $\lim_{n \to \infty} \sqrt{m_n} (n^2 m_n \bgn-\lg)=0$,
\item[(C3)] $\lim_{n \to \infty} \sqrt{m_n}(n\bwn-\lw) = 0$,
\item[(C4)] there exists $\theta<0$ such that $\phi_I(\theta)<\infty$ and
\item[(C5)] there exists $\theta' \in (0,2\log s_0)$ such that $\lim_{n \to \infty}m_n n^{-\theta'}=0$.
\end{enumerate}

\begin{theorem}
\label{multifinalmain1}
Suppose that $R_* > 1$ and conditions (C1)--(C5) are satisfied.  Then
\begin{equation*}
\sqrt{m_n}
\left(\begin{array}{c}
\barznmc -x(\tau)\\
\tznmt -z(\tau)\\
\tanm -a(\tau)\\
\end{array}
\right)\left|\right. \Gn
\convD \bN \qquad \mbox{as } n \to \infty.
\end{equation*}
\end{theorem}

Note from Theorem~\ref{multifinalmain}(b) and Slutsky's theorem that a necessary and sufficient condition for Theorem~\ref{multifinalmain1} to hold is that
\begin{equation}
\label{mnxnznanconv}
 \sqrt{m_n}(\xn(\taun)-x(\tau)) \to 0\mbox{ and }
 \sqrt{m_n}(\zn(\taun)-z(\tau)) \to 0 \mbox{ as } n \to \infty.
\end{equation}
Conditions (C1)--(C4) are sufficient for the second convergence in~\eqref{mnxnznanconv} to hold, and hence for $(\tznmt, \tanm) \mid \Gn$ to obey the corresponding central limit theorem.  Condition (C1) implies condition (C5) if $\log s_0>1$.  Note that if the infectious period distribution $I$ is held fixed, $s_0$ decreases with $\lambda_W$, so the closer the within-community epidemic $\En(\bwn, I)$ is to criticality the more stringent is condition (C5).  Theorem~\ref{multifinalmain1} holds under conditions (C1)--(C4) unless the within-community epidemic $\En(\bwn, I)$ is sufficiently close to criticality. %\annote{}{Added as FB suggested, with slight rewording to include $R_0>1$}(i.e.\ unless $R_0>1$ is sufficiently close to 1, see \eqref{R0def}).

To further illuminate the role/influence of condition (C5) above, we also note that using a more (respectively, less) stringent requirement to define a large within-community outbreak means that (C5) is replaced by a less (respectively, more) stringent condition. For example, if we take $G^{(n)}_W=\{Z_n\geq a \log n\}$, where $a \in (0,\infty)$, then (C5) is replaced by (C5a): there exists $\theta'\in(0,2a \log s_0)$ such that $\lim_{n\to\infty} m_n n^{-\theta'} =0$. Then (C1) and $\log s_0 >a^{-1}$ imply (C5a).
Alternatively, if we take $G^{(n)}_W= \{Z_n \geq  n^\beta\}$ (for $\beta\in(0,1/2)$) then (C1) is sufficient for the argument where we use (C5) in Appendix \ref{appB1}. This is qualitatively consistent with the above, in that the threshold for a large within-community outbreak has been made much more stringent and this results in no (C5)-like condition or condition on $s_0$ being required for Theorem~\ref{multifinalmain1}.
%\annote{}{Added last couple of sentences here. Wording OK? And is 1/2 an appropriate largest value to use for the range of possible $\beta$? FGB Slight change of wording.  We don't \emph{need} to modify (C5) when having a more stringent condition.  The BP approx breaks down if $\beta\ge 1/2$.}

\section{Numerical Examples}
\label{numerics}
In this section we present some numerical examples briefly illustrating our results. Figure~\ref{fig:FSapprox} demonstrates the use of two of our main limit theorems to approximate the distribution of $Z^{(n)}_T$,  the total number of individuals infected by an epidemic. We consider one situation with few large communities ($m=20$, $n=500$) and one with many large communities ($m=n=2000$); in both cases with two choices of infectious period distribution (exponential and constant, both with unit mean). Other model parameters are held fixed: $n\beta^{(n)}_W=\lambda_W=2$, $n^2m\beta^{(n)}_G=\lambda_G=6$.

We approximate the distribution of $Z^{(n)}_T$ in the event of a major outbreak in two different ways: (i) a mixture-of-normals approximation based on Theorem~\ref{fixedmasymmain}(c) and (ii) a normal approximation based on the second component of the vectors in Theorem~\ref{multifinalmain1}.

Note that the two approximations for the final size distribution are valid under different asymptotic regimes and use different characterisations of what a major outbreak is. The details are given in the lead-up to the statements of those theorems but, loosely speaking, the former considers the number of communities to be fixed and characterises a major outbreak as one in which there is a major outbreak within the community which contains the initial infective, whereas the latter considers the number of communities to be large and characterises a major outbreak as one in which the chain of between-community infections takes off and infects a significant fraction of communities. %(one in which the Reed-Frost epidemic of large within-community outbreaks described in Sec xxx takes off and infects a significant fraction of communities.)

This distinction is apparent on the first row of the figure, where we see that the Reed-Frost (RF) mixture-of-normals approximation characterises a range of final sizes, apart from those very close to zero which result from epidemics where the infection dies out in the initially infected community. The RF mixture approximation does capture the part of the distribution with final sizes less than around 3000, where there was a large outbreak in community 0 and a few other communities, but the process of community-to-community infection did not take off. On the other hand, both the RF mixture and the normal approximations do capture the part of the distribution that corresponds to epidemics of sizes roughly 5000--9000, where many communities are infected and experience large within-community epidemics, though obviously the simpler normal approximation captures this at only a very crude level.

Figure~\ref{fig:FSapprox} also illustrates the impact of the shape of the infectious period distribution. The more variable exponential infectious period distribution results in the within-community major outbreak probability $(1-\pi_W)$ being 0.5, compared to 0.8 for the fixed infectious period case. This is evident in the upper plots of the figure through the quite different final outcome distributions for the Reed-Frost `epidemic' of large within-community outbreaks; there is an appreciably larger probability of very few communities being infected and also more variability in the final size within the large outbreaks when $I \sim {\rm Exp}(1)$. The different infectious period distributions also manifest in the lower plots of the figure, which indicate that the quality of approximation derived from the CLT, though good, is affected by the infectious period distribution.

%from which it seems that quality of approximation derived from the CLT can be significantly affected by the infectious period distribution.

%We briefly present numerical examples illustrating some of the above results. Figure~\ref{fig:FSapprox} demonstrates the use of two of our main limit theorems to approximate the distribution of $Z^{(n)}_T$, the total number of individuals infected in the event of a major outbreak, in two different ways: (i) a mixture-of-normals approximation based on Theorem~\ref{fixedmasymmain}(c) (valid in the limit of a fixed number of large communities) and (ii) a normal approximation based on the second component of the vectors in Theorem~\ref{multifinalmain1} (valid in the limit of a large number of large communities).
%The upper plots have few large communities ($m=20$, $n=500$) and the lower plots have many large communities ($m=n=2000$); and the left and right plots use different infectious period distributions (exponential and fixed) with the same (unit) mean. Other model parameters are the same in all plots: $n\beta^{(n)}_W=\lambda_W=2$, $n^2m\beta^{(n)}_G=\lambda_G=6$.

%We see that the relevant asymptotic distributions seem to provide reasonable working approximations for the final size distributions investigated.
%We also see that the infectious period distribution has a substantial effect on the distribution of the outcome of the epidemic; and the lower plots of the figure suggest that the quality of approximation derived from the CLT can be significantly affected by the infectious period distribution.

\begin{figure}
\begin{center}
\begin{tabular}{ccc}
 (a) $I\sim\mbox{Exp}(1)$ &
 (b) $I\equiv1$ \\
 \includegraphics[width=\hfigwidth]{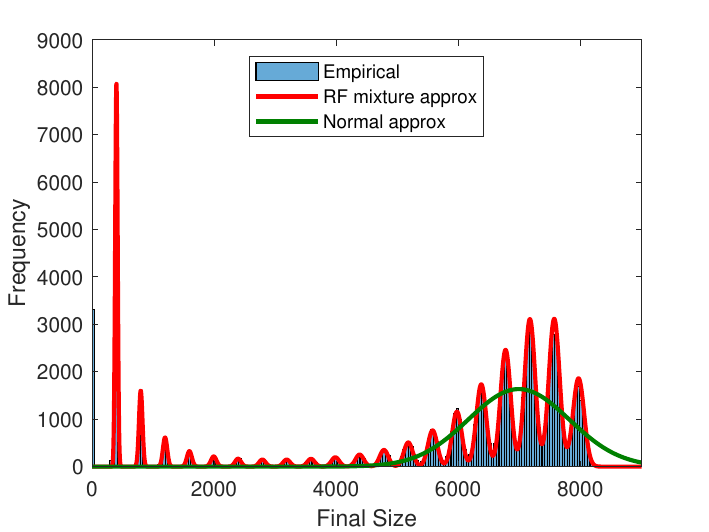} &
 \includegraphics[width=\hfigwidth]{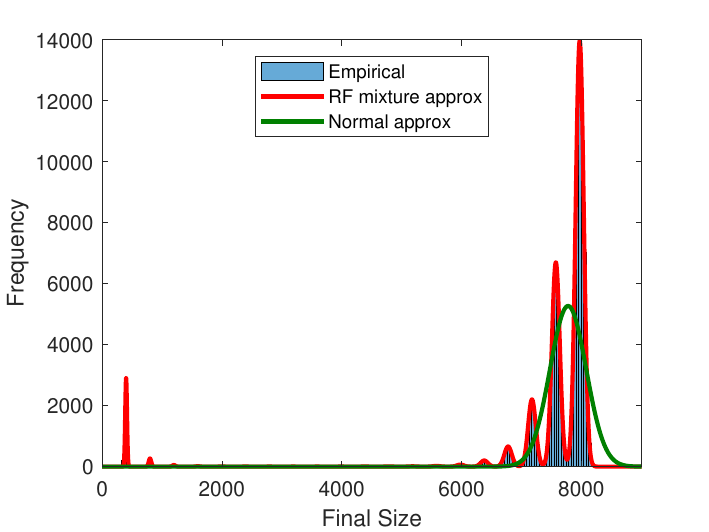} \\
 \includegraphics[width=\hfigwidth]{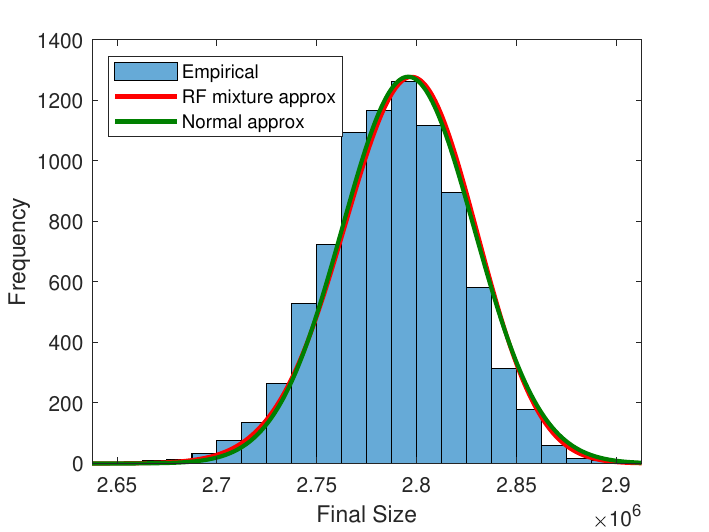} &
 \includegraphics[width=\hfigwidth]{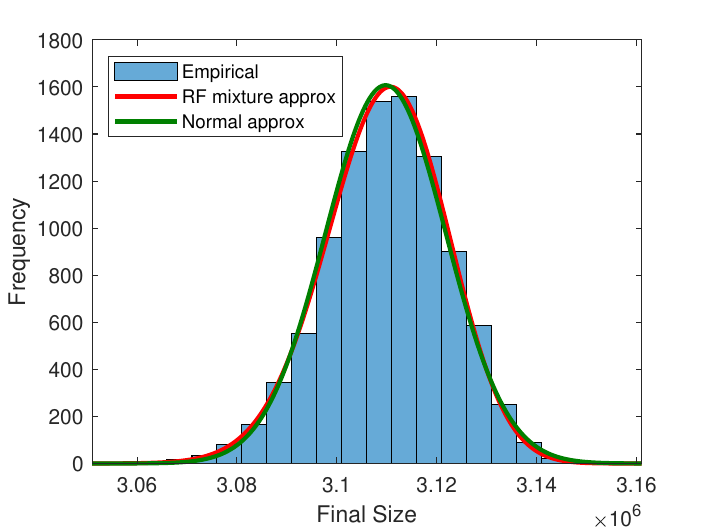}
\end{tabular}
\end{center}
\caption{Histograms of simulated final sizes for our model, with overlaid approximations based on Theorems~\ref{fixedmasymmain} and~\ref{multifinalmain1}. The upper plots have $m=20$, $n=500$ and the lower plots have $m=n=2000$. The left and right plots have infectious periods which are (a) exponentially distributed and (b) fixed, respectively, both with mean 1. Each plot is based on 10,000 simulations and in the lower plots we only display final sizes larger than the cutoff $0.15mn=0.6\times10^{6}$. Other parameters are $\lambda_W=n\beta^{(n)}_W=2$, $\lambda_G=n^2m\beta^{(n)}_G=6$. These parameters yield (a) $R_0=2$, $R_*\approx2.4$, $\pi_W=0.50$, $z_\infty\approx0.80$, $z(\tau)\approx0.70$ and (b)  $R_0=2$, $R_*\approx3.8$, $\pi_W=0.20$, $z_\infty\approx0.80$, $z(\tau)\approx0.78$.
}
\label{fig:FSapprox}
\end{figure}

%\pagebreak

\section{Asymptotic analysis of the model}
\label{rigorous}

Throughout this section we assume implicitly that the infectious period distribution satisfies $\E[I^{2+\alpha}]< \infty $ for some $\alpha>0$ and that the infection rate parameters satisfy~\eqref{infectionrates}, viz.\ $\bwn n \to \lw \in (0,\infty)$ and $\bgn n^2 m_n \to \lg \in (0,\infty)$. The further assumptions (C1)--(C5) are explicitly referred to as and when required.

\subsection{Embedding construction of final outcome}
\label{embedding}
To provide a rigorous asymptotic analysis of the final outcome of the epidemic in a population of communities, it is useful to extend the embedding arguments of \cite{Scal85,Scal90} and \cite{Ball97} to the present setting.

For fixed $n \in \mathbb{N}$, we define an epidemic process $\left\{\left(Z^{(n)}(t),A^{(n)}(t)\right):t \ge 0\right\}$ among a homogenously mixing population, $\mathcal{N}$ say, of $n$ individuals, all of whom are susceptible when $t=0$ and are also exposed to external infection.  In this construction the variable $t$ relates to cumulative time units of external infection experienced by the population and \emph{not} real time.
Let $\etan$ be a homogeneous Poisson process on $[0, \infty)$ with rate $\bgn n^2 m_n$ and recall from~\eqref{infectionrates} that $\bgn n^2 m_n \to \lg\in(0,\infty)$ as $n \to \infty$.  At each point of $\etan$, an individual is chosen uniformly at random from $\mathcal{N}$.  If the chosen individual is susceptible then it triggers an SIR epidemic among the susceptible individuals in $\mathcal{N}$, having size and severity distributed according to that of $\mathcal{E}^{(n_0)}(\bwn,I)$, where $n_0$ is the number of susceptibles in $\mathcal{N}$ immediately prior to the point of $\etan$.  (To be clear, the infection rate in $\mathcal{E}^{(n_0)}(\bwn,I)$ is $\bwn$, not $\beta^{(n_0)}_W$.)  The epidemic takes place instantaneously.  If the chosen individual is not susceptible then nothing happens.  The Poisson process $\etan$, uniform samplings and instantaneous epidemics are mutually independent.  For $t \ge 0$, let $Z^{(n)}(t)$ and $A^{(n)}(t)$ be respectively the sum of the sizes and severities of all epidemics that have occurred by `time' $t$.

%\annote{}{Is this enough re saying that these constructions are nearly the same thing? }
Note that, for fixed $t\ge0$, $(\Zn(t) , \An(t))$ as defined in the previous paragraph has the same distribution as the joint size and severity of $\En(t)$ defined in Definition~\ref{defEpsnt}. To see this, note that in $\En(t)$ each of the $n$ individuals is independently initially susceptible, with probability $\pi_n(t) = \exp(-\bgn n m_n t)$. In $(\Zn(t) , \An(t))$, a point in $\etan$ infects a given individual with probability $1/n$, so each individual is infected externally at rate $\bgn n^2 m_n / n$ and thus avoids external infection until time $t$ with probability $\P(\Po(\bgn n m_n t)=0) = \pi_n(t)$, independently of other individuals. The two different constructions %\annote{}{FB suggests adding ``$(Z^{(n)}(t), A^{(n)}(t))$ for fixed $t$'' here, but I wonder if leaving it out will be simpler and still clear?}
are useful in different parts of the rest of the manuscript: the process $\{(\Zn(t),\An(t)): t \ge 0\}$ provides a framework for proving limiting properties of the final outcome of the epidemic in a population of communities, whereas the epidemic $\En(t)$ is required for some calculations in the proofs.

We now describe the construction of the final outcome of a slightly modified model $\Enm_{\rm mod}$ for the epidemic amongst communities $1,2,\cdots,m_n$, which has (asymptotically) the same final outcome among those communities as $\Enm$ defined in Section~\ref{model} but is simpler to analyse.
For $n=1,2,\cdots$, we define the functions
$\left( \left(Z^{(n)}_i(\cdot), A_i^{(n)}(\cdot)\right), \, i=1,2,\cdots,m_n \right)$
to be i.i.d.\ copies of $\left(Z^{(n)}(\cdot),A^{(n)}(\cdot)\right)$.  For $t \ge 0$, let
\begin{equation*}
\zbn(t)=\sum_{i=1}^{m_n} Z^{(n)}_i(t) \qquad \mbox{and}\qquad \abn(t)=\sum_{i=1}^{m_n} A^{(n)}_i(t).
\end{equation*}
Suppose that each individual in the population of $m_n$ communities is initially exposed to $T_0^{(n)}$ time units of infection.  In $\Enm$, the total force of infection on a community if each of its members is exposed to $1$ time unit of infection from the population at large is $n\bgn$ $(=\bgn n^2 m_n/n m_n)$.  Thus, in view of the rate of the Poisson process $\etan$, in $\Enm_{\rm mod}$ we assume that if each individual in the population of $m_n$ communities is initially exposed to $T_0^{(n)}$ time units of infection, then the ensuing within-community epidemics give rise to a further $\abn(\tildetn_0)$ time units of infection, where $\tildetn_0=T_0^{(n)}/(n m_n)$, which may in turn give rise to further time units of infection.  For $k=0,1,\cdots$, let
\begin{equation}
\label{Tkiteration}
T_{k+1}^{(n)}=T_0^{(n)}+\abn(\tildetn_k),
\end{equation}
where $\tildetn_k=T_k^{(n)}/(n m_n)$.  Then $k^{(n)}_*=\min\{k:T_{k+1}^{(n)}=T_k^{(n)}\}$ is well defined since the population is finite.  Let $\tinfn={T_{k^{(n)}_*}^{(n)}}$ and $\tildetinfn=\tinfn/(nm_n)$.  Then  $\tildetinfn$ is given by
\begin{equation}
\label{Tinfinity}
\tildetinfn=\min\{t\ge 0:t=\tildetn_0+\tildeabn(t)\},
\end{equation}
where $\tildeabn(t)=(nm_n)^{-1}\abn(t)$, and the total size and severity of the epidemic in the $m_n$ communities are given by $\zbn(\tildetinfn)$ and
$\abn(\tildetinfn)$, respectively.

To connect more concretely with $\Enm$, we are concerned with the case when within-community epidemics are supercritical (i.e.\ $R_0=\lw \mu_I>1$) and the within-community epidemic in community $0$, triggered by the initial infective, takes off (i.e.\ has size at least $\log n$).  (If this epidemic does not take off then by Theorem~\ref{singleSIR}(b)(i), the expected total number of external contacts that emanate from community $0$ is bounded above by $\beta^{(n)}_G \mu_I m_n n \log n$, which tends to $0$ as $n \to \infty$, so the probability that the epidemic remains solely within community $0$ tends to 1 as $n \to \infty$.)
In the proof of Theorem~\ref{fixedmasymmain} we take $T_0^{(n)}$ to be the severity of the within-community epidemic in community 0 (i.e.\ ignoring contacts involving other communities), assuming that it takes off.
The modified model $\Enm_{\rm mod}$ differs from $\Enm$ in that it does not permit individuals in communities $1,2,\cdots,m_n$ to infect individuals in community $0$.  The following lemma shows that this difference has no material effect under the asymptotic schemes we consider.

For the epidemic model $\Enm$, let $Z^{(n)}_0$ and $T_0^{(n)}$ be the size and severity of the \emph{within-community} epidemic in community $0$ triggered by the initial infective. Further, let $Z^{(n)}_{0,+}$ be the number of individuals in community 0 that are infected in the course of $\Enm$ and let $A_0^{(n)}$ be the sum of their infectious periods, and write $\barzn_{0,+}=Z^{(n)}_{0,+}/n$, $\barzn_0=\Zn_0/n$, $\bartn_0=\Tn/n$, $\baran_0=\An_0/n$ and $\tildean_0=A_0^{(n)}/(nm_n)$.  As we are concerned with only the final outcome of the epidemic, we can assume that the within-community $0$ epidemic takes place first and that community 0 is exposed to external infection later.

\begin{lemma}
\label{initialsev}
Suppose that $R_0=\lw \mu_I > 1$.  Then, as $n\to\infty$,
\begin{equation}
\tildean_0 - \tildetn_0 \mid \Zn_0 \ge \log n \convp 0
\label{eq:AT0nDiff}
\end{equation}
and
\begin{equation}
\sqrt{n}(\barzn_{0,+} - \barzn_0) \mid \Zn_0 \ge \log n \convp 0 \quad \mbox{and} \quad \sqrt{n}(\baran_0 - \bartn_0) \mid \Zn_0 \ge \log n \convp 0 .
\label{eq:Z0nDiff}
\end{equation}
\end{lemma}

\begin{proof} Throughout the proof we assume implicitly that $\{\Zn_0 \ge \log n\}$ occurs and suppress the explicit conditioning on that event.  Thus, by Theorem~\ref{singleSIR}(b)(ii), $\barzn_0=\frac{1}{n}Z^{(n)}_0 \convp \zinf$ as $n \to \infty$.
Let $N^{(n)}_{C0}$ be the total number of contacts made by infectives in communities $1,2,\cdots,m_n$ to individuals in community $0$ during the course of $\Enm$.  Then $N^{(n)}_{C0}$ is bounded above by the number of such contacts if every member of communities $1,2,\cdots,m_n$ is infected, so $\E[N^{(n)}_{C0}] \le m_n n^2 \bgn \mu_I$.   Then, given $Z^{(n)}_0$, the process of infectives in the within-community epidemic initiated by any of the above $N^{(n)}_{C0}$ contacts is bounded above by the branching process $\BB\left(n\bwn(1-\barzn_0) I \right)$.

Now, recalling from Section~\ref{sec:singleComm} that $(1-\zinf)R_0<1$, fix the constant $\theta \in  ((1-\zinf)R_0,1)$.  If $n \bwn (1-\bar{Z}^{(n)}_0) \mu_I \le \theta$ then, given $Z^{(n)}_0$, the mean total progeny including the ancestor of $\BB\left(n\bwn(1-\barzn_0)I\right)$ is bounded above by $(1-\theta)^{-1}$.   Let $Z^{(n)}_{C0}$ and $A^{(n)}_{C0}$ denote respectively the total number of infectives and severity of the epidemics in community $0$ triggered by the $N^{(n)}_{C0}$ outside contacts.
Note that $\Zn_{C0}=\Zn_{0,+}-\Zn_0$ and $A^{(n)}_{C0}=A^{(n)}_0-\Tn_0$.
Fix $\epsilon>0$.  Then,
\begin{align}
\label{eq:pzncogteps}
\P\left(n^{-1/2}\Zn_{C0} > \epsilon\right)&=\P\left(n^{-1/2}\Zn_{C0} > \epsilon \mid n \bwn (1-\barzn_0) \mu_I > \theta\right)\P\left(n \bwn (1-\barzn_0) \mu_I > \theta\right)\\
&\,\,+\P\left(n^{-1/2}\Zn_{C0} > \epsilon \mid n \bwn (1-\barzn_0) \mu_I\le \theta\right)\P\left(n \bwn (1-\barzn_0) \mu_I\le \theta\right).\nonumber
\end{align}
Now, $\lim_{n \to \infty} \P(n \bwn (1-\barzn_0) \mu_I > \theta)=0$, by Theorem~\ref{singleSIR}(b)(ii), so the first term on the right-hand side of~\eqref{eq:pzncogteps} tends to $0$ as $n \to \infty$.  Further,
\begin{equation*}
\E \left[n^{-1/2}Z^{(n)}_{C0}\mid n \bwn (1-\barzn_0) \mu_I\le \theta \right] \le n^{-1/2}m_n n^2 \bgn \mu_I (1-\theta)^{-1} \to 0 \quad \mbox{as } n \to \infty,
%\label{eq:ZC0bound}
\end{equation*}
since $m_n n^2 \bgn \to \lg$ as $n \to \infty$.  Application of Markov's inequality then shows that the second term on the right-hand side of~\eqref{eq:pzncogteps} tends to $0$ as $n \to \infty$, and the first result in~\eqref{eq:Z0nDiff} follows as $\epsilon>0$ is arbitrary.  The second result in~\eqref{eq:Z0nDiff} follows by a similar argument, since by Wald's identity for epidemics \cite[Corollary~2.2]{Ball86},
\[
\E [A^{(n)}_{C0}\mid n \bwn (1-\barzn_0) \mu_I\le \theta]=\mu_I \E[Z^{(n)}_{C0}\mid n \bwn (1-\barzn_0) \mu_I\le \theta].
\]

Finally, \eqref{eq:AT0nDiff} follows from the second result in~\eqref{eq:Z0nDiff}, as $0 \le \tildean_0 - \tildetn_0 \le \baran_0 - \bartn_0$.
\end{proof}

Lemma~\ref{initialsev} justifies our approach in proving Theorem~\ref{fixedmasymmain} of studying the final outcome of the model $\Enm$ in the event of a large outbreak in community 0 through the modified model with $\tildetn_0$ units of initial external infectious pressure.  Effectively we initially expose communities $1,2,\cdots,m_n$ to $\tildetn_0$ units of external infectious pressure as a convenient `substitute' for conditioning on a large outbreak in community 0.
 Note that the modified epidemic $\Enm_{\rm mod}$ is a lower bound for the `true' model $\Enm$, and a branching process which provides an upper bound for $\Enm$ is given in the first paragraph of the proof of Lemma~\ref{initialsev}.

The following lemma is required in the sequel.  Let $\zbarnc=\{\bar{Z}^{(n)}(t):t \ge 0\}$ and $\abarnc=\{\bar{A}^{(n)}(t):t \ge 0\}$, where $\bar{Z}^{(n)}(t)=\frac{1}{n}Z^{(n)}(t)$ and $\bar{A}^{(n)}(t)=\frac{1}{n}A^{(n)}(t)$.  Also, for any $k \in \mathbb{N}$, let
$\Rightarrow$ denote weak convergence in the space $D_{\mathbb{R}^k}[0,\infty)$ of right-continuous functions $f:[0,\infty) \to \mathbb{R}^k$ having limits from the left (i.e.~c\`{a}dl\`{a}g functions), endowed with the Skorohod metric (e.g.~\cite[Chapter 3]{Ethi86}).

\begin{lemma}
\label{ZAnweak}
As $n \to \infty$,
\begin{equation}
\label{zanweak}
\left(\zbarnc,\abarnc\right) \Rightarrow \Psi(\cdot)(1,\mu_I)\zinf,
\end{equation}
where $\Psi(\cdot)=\{\Psi(t):t \ge 0\}$, with $\Psi(t)=1_{\{L \le t\}}$ and $L\sim {\rm Exp}(\lg(1-\pi_W))$.
\end{lemma}

\begin{proof} Let $(\Omega, \mathcal{F}, P)$ be a probability space on which is defined a homogeneous Poisson process $\eta$ on $[0,\infty)$ having rate $\lg$ and let $0<r_1<r_2<\cdots$ denote the times of the points in $\eta$.   For $n=1,2,\cdots$, let $\etan$ denote the point process with points at $0<r_1^{(n)}<r_2^{(n)}<\cdots$, where $r_k^{(n)}=\frac{\lg}{\bgn n^2 m_n }r_k$ $(k=1,2,\cdots)$, so
$\lim_{n \to \infty} r_k^{(n)}=r_k$ $(k=1,2,\cdots)$. (See the second paragraph of Section~\ref{embedding} for an interpretation of $\eta^{(n)}$; unadorned $\eta$ is the corresponding limiting process.)

For $n,k=1,2,\cdots$, let $\checkzn_k$  be the size of the epidemic initiated by the $k$-th point in $\etan$.  For $n=1,2,\cdots$, let $K^{(n)}=\min\{k:\checkzn_k \ge \log n\}$.  Then Theorem~\ref{singleSIR}(b)(i) implies that
\begin{equation}
\label{Knconvd1}
\lim_{n \to \infty}\P\left(K^{(n)}=k\right)=\pi_W^{k-1}(1-\pi_W)\qquad (k=1,2,\cdots),
\end{equation}
and Theorem~\ref{singleSIR}(b)(ii) implies that
\begin{equation*}
\barzn(r^{(n)}_{K^{(n)}}) \convp \zinf \qquad \mbox{as } n \to \infty.
\end{equation*}
Clearly, $\barzn(r^{(n)}_k) \convp 0$  for $k=1,2,\cdots, K^{(n)}-1$.

Let $\theta \in ((1-\zinf)R_0,1)$. Then arguing as in the proof of Lemma~\ref{initialsev}, for any $k> K^{(n)}$, the epidemic initiated by the $k$-th point in $\etan$ is bounded above by the (subcritical) branching process $\BB(\theta \mu_I^{-1} I)$ with probability tending to one as $n \to \infty$.  Thus, for any $k> K^{(n)}$,   by Markov's inequality,
\begin{equation*}
\lim_{n \to \infty} \P\left(\checkzn_k \ge \log n\right) \le \lim_{n \to \infty} \frac{1}{(1-\theta) \log n}=0.
\end{equation*}
Hence, as $n \to \infty$,
\begin{equation}
\label{zbarconvp}
\barzn(r^{(n)}_k)-\zinf 1_{\{K^{(n)} \le k\}} \convp 0.
\end{equation}
A similar argument shows that, as $n \to \infty$,
\begin{equation}
\label{abarconvp}
\baran(r^{(n)}_k)-\mu_I \zinf 1_{\{K^{(n)} \le k\}} \convp 0.
\end{equation}
Further,~\eqref{Knconvd1} implies that
\begin{equation}
\label{Knconvd2}
K^{(n)} \convD K \qquad \mbox{as } n \to \infty,
\end{equation}
where $\P(K=k)=\pi_W^{k-1}(1-\pi_W)$ $(k=1,2,\cdots)$.

The Skorohod representation theorem implies that there exists a probability space $(\Omega, \mathcal{F}, P)$,
on which is defined the Poisson process $\eta$ (and hence also the Poisson processes $\etan, n=1,2,\cdots$) and random variables $K, K^{(n)}, \barzn(r^{(n)}_k)$ and $\baran(r^{(n)}_k)$ $(n,k=1,2 \cdots)$, so that the convergence in~\eqref{zbarconvp}, \eqref{abarconvp} and~\eqref{Knconvd2} holds almost surely. Thus, there exists $F \in \mathcal{F}$, with $\P(F)=1$, such that for all $\omega \in F$,
$\displaystyle \lim_{k \to \infty} r_k(\omega) = \infty$,  $\displaystyle \lim_{n \to \infty}K^{(n)}(\omega) = K(\omega)$, $\displaystyle \lim_{n \to \infty} \barzn(r^{(n)}_k)(\omega)  = \zinf 1_{\{k \ge K^{(n)}(\omega)\}}$
and
$\displaystyle \lim_{n \to \infty} \baran(r^{(n)}_k)(\omega)=\zinf\mu_I 1_{\{k \ge K^{(n)}(\omega)\}}.$
It follows that for all $\omega \in F$,
\begin{equation}
\label{uniconv1}
\lim_{n \to \infty} \sup_{0 \le t <r_{K^{(n)}(\omega)}(\omega)}\max\left\{\left|\barzn(t,\omega)\right|,\left|\baran(t,\omega)\right|\right\} =0
\end{equation}
and, for any $T>0$,
\begin{equation}
\label{uniconv2}
\lim_{n \to \infty} \sup_{r_{K^{(n)}(\omega)}(\omega) \le t \le T}\max\left\{\left|\barzn(t,\omega)-\zinf \right|,
\left|\baran(t,\omega)-\zinf \mu_I \right|\right\}=0.
\end{equation}

Fix $\omega \in F$ and define the function $\Psi(\cdot,\omega):[0,\infty)\to \mathbb{R}$ by $\Psi(t,\omega)=1_{\{t\ge
r_{K(\omega)}\}}$.  Observe that $\Psi(\cdot)$ has the same distribution as in the statement of the lemma, as~\eqref{Knconvd2} implies that $r_K \sim {\rm Exp}(\lg(1-\pi_W))$.  For $g,h \in D_{\mathbb{R}^2}[0,\infty)$, let $d(g,h)$ denote the distance between $g$ and $h$ in the Skorohod metric \citep[see][Chapter 3.5]{Ethi86}.  By Proposition 5.3 on page 119 of~\citet{Ethi86},
\begin{equation*}
d\left(\left(\barzn(\cdot,\omega),\baran(\cdot,\omega) \right), \Psi(\cdot,\omega)(1,\mu_I)\zinf\right) \to 0 \qquad \mbox{as } n \to \infty
\end{equation*}
if, for each $T>0$, there exists a sequence $(\lambda_n)$ of strictly increasing functions mapping $[0, \infty)$ onto $[0, \infty)$ such that
\begin{equation}
\label{skor1}
\lim_{n \to \infty} \sup_{0 \le t \le T}|\lambda_n(t)-t|=0
\end{equation}
and
\begin{equation}
\label{skor2}
\lim_{n \to \infty} \sup_{0 \le t \le T}\left|\Psi(\lambda_n(t),\omega)(1,\mu_I)\zinf-\left(\barzn(t,\omega),\baran(t,\omega)\right)\right|=0,
\end{equation}
where, for $x,y \in \mathbb{R}^2$, $|x-y|$ denotes Euclidean distance.

If $T<r_{K(\omega)}(\omega)$ then~\eqref{skor1} and~\eqref{skor2} are satisfied when $\lambda_n$ is the identity map for each $n$, since $r_{K^{(n)}(\omega)}(\omega) \to r_{K(\omega)}(\omega)$, so~\eqref{uniconv1} implies that~\eqref{skor2} holds.  If $T\ge r_{K(\omega)}(\omega)$, choose $\delta \ge 0$ so that $T+\delta>r_{K(\omega)}(\omega)$ and let $\lambda_n$ be the piecewise-linear function joining $(0,0), (r_{K^{(n)}(\omega)}(\omega),r_{K(\omega)}(\omega))$ and $(T+\delta,T+\delta)$, with $\lambda_n(t)=t$ for $t > T+\delta$.  Then~\eqref{skor1} holds as $r_{K^{(n)}(\omega)}(\omega) \to r_{K(\omega)}(\omega)$ as $n \to \infty$, and~\eqref{uniconv1} and~\eqref{uniconv2} imply that~\eqref{skor2} holds, since $t \in  [0, r_{K^{(n)}(\omega)}(\omega)]$ if and only if $\lambda_n(t) \in [0,r_{K(\omega)}(\omega)]$.

Now $\P(F)=1$, so $\left(\barzn(\cdot), \baran(\cdot)\right)$ converges almost surely (and hence weakly) to $\Psi(\cdot)(1,\mu_I)\zinf$ as $n \to \infty$, which completes the proof.
\end{proof}

\subsection{Fixed number of communities; proof of Theorem \ref{fixedmasymmain}}
\label{fixedm}
Suppose that $m_n=m$ for all $n=1,2,\cdots$.  Recall that $\Enm$ denotes the multi-community epidemic defined in Section~\ref{model}.  Further, for $n=1,2,\cdots,$ $Z^{(n)}_T$ and $A^{(n)}_T$ are the total number of individuals infected in and severity of $\Enm$, respectively, and $Z^{(n)}_C$ is the number of communities that experience epidemics of size at least $\log n$; each over all $m+1$ communities.  We consider the case that the epidemic in community $0$ infects at least $\log n$ individuals. We prove that in that case $Z^{(n)}_C$ converges in distribution to the final size of a Reed-Frost epidemic (as defined in Section~\ref{sec:multiComm}), and derive an associated conditional law of large numbers and central limit theorem for $(\barzn_T, \baran_T)=n^{-1}(Z^{(n)}_T, A^{(n)}_T)$.

We first give a \citet{Sell83} construction  of the final size of the above-mentioned Reed-Frost epidemic.  Consider a single population epidemic, with initially one infective, labelled $0$, and $m$ susceptibles, labelled $1,2,\cdots,m$.  Infectives have constant infectious periods of length $\zinf \mu_I/m$, during which they contact any given susceptible at the points of a Poisson process having rate $\lg(1-\pi_W)$.  The final size of this epidemic has the same distribution as that of the Reed-Frost epidemic with initially one infective and $m$ susceptibles, having pairwise-infection probability
$p_{\rm RF}$ given by~\eqref{RFprob}.
Recall that $Z^{(m)}_{\rm RF}$ denotes the final size of this epidemic, including the initial infective.  A realisation of $Z^{(m)}_{\rm RF}$ can be constructed by letting $L_1',L_2',\cdots,L_m'$ be i.i.d.\ ${\rm Exp}(\lg(1-\pi_W))$ random variables and setting
\[
Z^{(m)}_{\rm RF}=\min\{k:k \zinf \mu_I/m < L_{(k)}'\},
\]
 where $L_{(1)}',L_{(2)}',\cdots, L_{(m)}'$ are the order statistics of $L_1',L_2',\cdots,L_m'$ and $L_{m+1}'=\infty$ (cf.\ \citet[Equation (2.1)]{Ball86}).  Note that, for $k=1,2,\cdots,m$,
 \begin{equation*}
 k\zinf \mu_I/m <L_{(k)}' \iff \sum_{i=1}^m 1_{\{L_i' \le k \zinf \mu_I/m\}}\le k-1,
 \end{equation*}
 so $Z^{(m)}_{\rm RF}$ admits the representation
 \begin{equation}
 \label{RFsize}
 Z^{(m)}_{\rm RF}=\min\left\{t \ge 1:t=1+\sum_{i=1}^m 1_{\{L_i' \le (\zinf \mu_I/m)t\}}\right\}.
 \end{equation}

\begin{proof}[Proof of Theorem~\ref{fixedmasymmain}]

In the epidemic process $\left\{ \left( Z^{(n)}(t), A^{(n)}(t) \right) : t \ge 0 \right\}$ defined at the start of Section~\ref{embedding}, for $t \ge 0$, let $$X^{(n)}(t)=1_{\{Z^{(n)}(t) \ge \log n\}} \quad \mbox{and} \quad \bY^{(n)}(t)=\sqrt{n}\left(\begin{array}{c}
\barzn_T -\zinf  X^{(n)}(t)\\
\baran_T -\zinf \mu_I  X^{(n)}(t)\\
\end{array}
\right).$$
Then arguing as in the proof of Lemma~\ref{ZAnweak} and also using Theorem~\ref{singleSIR}(b)(iii) shows that
\begin{equation}
\label{ZAXYweak}
\left(\zbarnc,\abarnc,X^{(n)}(\cdot),\bY^{(n)}(\cdot)\right) \Rightarrow \Psi(\cdot)\left(\zinf,\mu_I \zinf,1,\bY\right)
\end{equation}
as $n \to \infty$, where $\bY \sim {\rm N}(\bzero,\Sigma_{\bY})$ and is independent of the random variable $L$ used to define $\Psi(\cdot)$.  Let $$\left(\zbarnci,\abarnci,X^{(n)}_i(\cdot),\bY^{(n)}_i(\cdot)\right) \qquad (i=1,2,\cdots,m)$$ be i.i.d.\ copies of $\left(\zbarnc,\abarnc,X^{(n)}(\cdot),\bY^{(n)}(\cdot)\right)$.  Then, as $n \to \infty$, it follows immediately from~\eqref{ZAXYweak} that for every $i=1,2,\cdots,m$,
\begin{equation}
\label{ZAXYweak1}
\left(\zbarnci,\abarnci,X^{(n)}_i(\cdot),\bY^{(n)}_i(\cdot)\right) \Rightarrow \Psi_i(\cdot)\left(\zinf,\mu_I \zinf,1,\bY_i\right),
\end{equation}
where $\bY_i \sim {\rm N}(\bzero,\Sigma_{\bY})$ and $\Psi_i(t)=1_{\{t \ge L_i\}}$ with $L_i\sim {\rm Exp}(\lg(1-\pi_W))$.  Furthermore, $L_1,L_2,\cdots,L_m$ and $\bY_1,\bY_2,\cdots,\bY_m$ are all mutually independent.  Let $Z^{(n)}_0$ and $T^{(n)}_0$ denote respectively the size and severity of the initial within-community epidemic in community $0$ in $\Enm$,  and let
\begin{equation*}
\bar{Z}^{(n)}_0=n^{-1}Z^{(n)}_0, \quad \bar{T}^{(n)}_0=n^{-1}T^{(n)}_0 \quad \mbox{and} \quad \bY^{(n)}_0=\sqrt{n}\left(\begin{array}{c}
\barzn_0 -\zinf \\
\bartn_0 -\zinf \mu_I \\
\end{array}
\right).
\end{equation*}
 Then, since that epidemic infects at least $\log n$ individuals, Theorem~\ref{singleSIR} and Lemma~\ref{initialsev} imply that $\tildetn_0=\frac{1}{mn}T^{(n)}_0 \convp \frac{\zinf \mu_I}{m}$, $\bar{Z}^{(n)}_0 \convp \zinf$ and
$\bY^{(n)}_0 \convD {\rm N}(\bzero, \Sigma_{\bY})$ as $n \to \infty$.
Further, $\left(Z^{(n)}_0,T^{(n)}_0\right)$ is independent of $\left(\zbarnci,\abarnci,X^{(n)}_i(\cdot),\bY^{(n)}_i(\cdot)\right)$ $(i=1,2,\cdots,m)$.  Now let
\begin{align*}
\xbn(t) & = \sum_{i=1}^m X_i^{(n)}(t), & \bybn & = \sum_{i=1}^m \bY_i^{(n)}(t) ,\\
\barzbn(t) & = n^{-1}\sum_{i=1}^m Z^{(n)}_i(t) \quad \mbox{and} & \barabn(t) & = n^{-1}\sum_{i=1}^m A^{(n)}_i(t).
\end{align*}
 Then using \cite[Theorem~3.2]{Bill68} and the continuous mapping theorem \cite[e.g.][Theorem~5.1]{Bill68},
\begin{align}
\label{Multiconv}
  \left(\bar{Z}^{(n)}_0 ,\tildetn_0,\right. & \left. Y^{(n)}_0, \barzbn(\cdot),\barzan(\cdot),\xbn(\cdot),\bybn(\cdot)\right) \nonumber \\
  & \convD \left(\zinf, m^{-1}\zinf \mu_I,Y_0,\zinf \Psib(\cdot), \mu_I \zinf \Psib(\cdot), \Psib(\cdot),\Psi_{\bY}(\cdot)\right)
\end{align}
as $n \to \infty$, where $\Psib(t)=\sum_{i=1}^m \Psi_i(t)$ and $\Psi_{\bY}(t)=\sum_{i=1}^m \bY_i \Psi_i(t)$.

The probability that the Poisson process $\eta$ has a point at any integer multiple of $\zinf \mu_I$ is zero, so by the continuous mapping theorem it follows from~\eqref{Tinfinity} and~\eqref{Multiconv} that $\tildetinfn \convD \tilde{T}$ as $n \to \infty$, where
\begin{equation}
\label{sevlim}
\tilde{T}=\min\{t\ge 0:t=m^{-1}\zinf \mu_I(1+\Psib(t))\}.
\end{equation}
A further application of the continuous mapping theorem yields that
\begin{equation}
\label{sizelim}
Z^{(n)}_C= 1+\xbn(\tildetinfn) \convD Z_C= 1+\Psib(\tilde{T}) \qquad \mbox{as } n \to \infty.
\end{equation}
Equations~\eqref{sevlim} and~\eqref{sizelim} imply that $\tilde{T}=m^{-1}\zinf \mu_I Z_C$, so
$$Z_C=\min\{t\ge 0:t=1+\Psib(m^{-1}\zinf \mu_I t)\},$$
which on comparison with~\eqref{RFsize} and noting that $\Psib(t)=\sum_{i=1}^n 1_{\{L_i \le t\}}$, shows that $Z_C \eqdist Z^{(m)}_{\rm RF}$, hence proving part (a) of the theorem.

A similar argument shows that
\begin{equation*}
\left(\begin{array}{c}
\barzn_T \\
\baran_T \\
\end{array}
\right)
=\left(\begin{array}{c}
\barzn_0 \\
\bartn_0 \\
\end{array}
\right)
+\left(\begin{array}{c}
\barzbn(\tildetinfn) \\
\barabn(\tildetinfn) \\
\end{array}
\right)
\convD \zinf(1+\Psib(\tilde{T}))\left(\begin{array}{c}
1 \\
\mu_I \\
\end{array}
\right)
\qquad \mbox{as } n \to \infty,
\end{equation*}
and \eqref{fixedmLLNmain} follows, since $1+\Psib(\tilde{T})= Z_C$ from~\eqref{sizelim}.

Finally,
\begin{align*}
\sqrt{n}\left(\begin{array}{c}
\barzn_T-Z^{(n)}_C \zinf \\
\baran_T -Z^{(n)}_C \zinf \mu_I\\
\end{array}
\right)
&= \bY^{(n)}_0+\bybn(\tildetinfn)\\
&\convD \bY_0+\Psi_{\bY}(\tilde{T}) \qquad \mbox{as }n \to \infty\\
&\eqdist \sum_{i=0}^{Z_C-1} \bY_i
\end{align*}
and \eqref{fixedmCLTmain} follows.
\end{proof}

\subsection{Many communities}
\label{manycom}

We now consider the case when the number of communities $m_n$ diverges as $n \to \infty$.

\subsubsection{Early stages of the epidemic; proof of Theorem~\ref{manymearlymain}}
\label{early}
%Under the conditions of Theorem~\ref{manymearlymain}, during the early stages of the  epidemic $\Enm$, described in Section~\ref{model}, the process of infected communities can be approximated by a branching process which assumes that all global contacts are with individuals in previously uninfected communities.
Recall the  branching process $\BBcheck \sim \BB(\lg \bar{A})$, defined in Section~\ref{sec:divergingm}, which approximates the process of infected communities in the early stages of the epidemic $\Enm$.
Let $(Z_{n1},A_{n1}), (Z_{n2},A_{n2}),\cdots$ be i.i.d.\ copies of the bivariate random variable $(Z_n,A_n)$ defined in Section~\ref{sec:singleComm}.  Then $A_{n1}, A_{n2},\cdots$ can be used to construct a realisation of the branching process $\BBncheck= \BB(m_n n \bgn A_n)$ in the obvious fashion. Note that, as $n \to \infty$, $\BBncheck$ converges in distribution to $\BBcheck$, since $m_n n^2 \bgn \to \lg$.  Let $\checkzn_C$ denote the total size (i.e.\ total progeny including the ancestor) of $\BBncheck$. Recall also from Section~\ref{sec:divergingm} the notation associated with $\BBcheck$ and Theorem~\ref{manymearlymain}.

\begin{proof}[Proof of Theorem \ref{manymearlymain}]  To prove the theorem, we generalise the method of \cite{Ball95} to couple the epidemic $\Enm$ and the branching process $\BBncheck$.  Let $\chin_0=0$ and $\chin_1, \chin_2,\cdots$ be i.i.d.\ random variables that are uniformly distributed on the integers $0,1,\cdots,m_n$. Then, for $k=1,2,\cdots$, the individual contacted by the global contact corresponding to the the $k$-th birth in the branching process $\BBncheck$ resides in community $\chin_k$.  Let
$\Mn=\min\{k\ge 1:\chin_k \in \{\chin_0,\chin_1,\cdots,\chin_{k-1}\}\}$. Then the epidemic $\Enm$ and the branching process $\BBncheck$ coincide up until the time of the $\Mn$-th birth in the branching process, at which point they diverge if the individual contacted in community $\chin_{\Mn}$ is not susceptible.  For $k=1,2,\cdots$,
\begin{equation}
\label{birthday0}
\P\left(\Mn \le k\right)=1-\prod_{i=1}^{k-1} \left(1-\frac{i}{m_n+1}\right) \le \sum_{i=1}^{k-1} \frac{i}{m_n+1} = \frac{k(k-1)}{2(m_n+1)},
\end{equation}
so $\P\left(\Mn \le k\right) \to 0$ as $n \to \infty$.  Also, Theorem~\ref{singleSIR}(b) implies that $\bar{A}_n \convD \bar{A}$ as $n \to \infty$, so, recalling that $m_n n^2 \bgn \to \lg$ as $n \to \infty$, it follows that $m_n n \bgn A_n \convD \lg \bar{A}$, so the offspring distribution of
$\BBncheck$ converges to that of $\BBcheck$ as $n \to \infty$.

Let $\check{\pi}_n$ denote the extinction probability of $\BBncheck$.
Then, using \citet[Lemma 4.1]{Brit07}, $\check{\pi}_n \to \check{\pi}_G$ as $n \to \infty$.
Note that~\eqref{birthday0} implies that $\P\left(\Mn \le \log m_n\right) \to 0$
as $n \to \infty$.  Thus,
\begin{align}
\limsup_{n \to \infty} \P\left(\hatznc < \log m_n \right)&=\limsup_{n \to \infty} \P\left(\checkzn_C < \log m_n  \right) \nonumber\\
& \le \limsup_{n \to \infty} \P\left(\checkzn_C < \infty  \right) \nonumber\\
& = \check{\pi}_G \label{limsupng}
\end{align}
and, for $k=1,2,\cdots$,
\begin{align*}
\liminf_{n \to \infty} \P\left(\hatznc < \log m_n \right)&=\liminf_{n \to \infty} \P\left(\checkzn_C < \log m_n  \right) \\
&\ge\liminf_{n \to \infty} \P\left(\checkzn_C \le k  \right) \\
&=\P\left(\check{Z}_C \le k \right).
\end{align*}
Thus $\liminf_{n \to \infty} \P\left(\hatznc < \log m_n \right) \ge \check{\pi}_G$, since $\lim_{k \to \infty} P\left(\check{Z}_C \le k \right)= \check{\pi}_G$, whence, using~\eqref{limsupng}, $\lim_{n\to \infty}\P\left(\hatznc < \log m_n \right) = \check{\pi}_G$, and part (a) follows.

Turning to part (b), for any $k \in \mathbb{N}$,
\begin{align*}
\lim_{n \to \infty} \P\left(\hatznc \le k,\, \hatznc < \log m_n \right)&=
\lim_{n \to \infty} \P\left(\hatznc \le k \right)\\
&=\lim_{n \to \infty} \P\left(\checkzn_C \le k  \right)\\
&=\P\left(\check{Z}_C \le k\right)\\
&=\P\left(\check{Z}_{C} \le k, \, \check{Z}_{C}< \infty \right),
\end{align*}
where the second equality uses $\lim_{n \to \infty} \P\left(\Mn \le k\right)=0$ and the third equality follows as the offspring distribution of $\BBncheck$ converges to that of $\BBcheck$.  The first limit in~\eqref{BTlimitmain} now follows immediately, using part (a).  The second limit is proved similarly on noting that, using
Theorem~\ref{singleSIR}(b), $m_n n \bgn A_n 1_{\{Z_n \ge \log n\}} \convD \lg \mu_I \zinf V$ as $n \to \infty$, where $\P(V=0)=\pi_W=1-\P(V=1)$.  Part (c) of the theorem is also proved similarly, since Theorem~\ref{singleSIR}(b) implies that, as $n \to \infty$,
\begin{equation*}
\left(\bar{Z}_{ni},\bar{A}_{ni}, \sqrt{n}\left(\begin{array}{c}
\bar{Z}_{ni}- \zinf \\
\bar{A}_{ni}-\zinf \mu_I\\
\end{array}\right)\right) \convD \left(\zinf, \zinf \mu_I,\bY_i\right)V_i \quad (i=1,2,\cdots),
\end{equation*}
where $(V_1, \bY_1),(V_2,\bY_2),\cdots$ are independent copies of $(V,\bY)$, with $V$ and $\bY$ being independent, $V$ being distributed as above and $\bY \sim {\rm N}(\bzero,\Sigma_{\bY})$.
\end{proof}

In the next two subsections, we derive a law of large numbers and a central limit theorem for the final outcome of a global epidemic.  In Section~\ref{finmod}, we prove these results for the modified model $\Enm_{\rm mod}$ defined in Section~\ref{embedding}
and then, in Section~\ref{finaloutcome}, we show that corresponding results for the epidemic $\Enm$ follow.

\subsubsection{Final outcome of modified epidemic}
\label{finmod}
We return to the embedding construction, given in Section~\ref{embedding}, for the final outcome of a multi-community epidemic.  However, the amount of infection
$\Tn_0$ to which the population is initially exposed need not be the severity of the within-community epidemic in community $0$.  Indeed, in Section~\ref{finaloutcome}, $\Tn_0$ is the severity of the within-community epidemics in the first $\log m_n$ communities infected by a global epidemic.
Recall that $\left( \left(Z^{(n)}_i(\cdot), A_i^{(n)}(\cdot)\right), \, i=1,2,\cdots,m_n \right)$
are i.i.d.\ copies of $\left(Z^{(n)}(\cdot),A^{(n)}(\cdot)\right)$.

For $t \ge 0$, let
\begin{align}
\label{xbullet}
\barxbn(t)&=\frac{1}{m_n}\sum_{i=1}^{m_n} X^{(n)}_i(t)=\frac{1}{m_n}\sum_{i=1}^{m_n} 1_{\{ Z^{(n)}_i(t)\ge \log n\}},\\
\tildezbn(t)&=\frac{1}{nm_n}\sum_{i=1}^{m_n} Z^{(n)}_i(t)=\frac{1}{m_n}\sum_{i=1}^{m_n}\barzn_i(t), \nonumber\\
\tildeabn(t)&=\frac{1}{nm_n}\sum_{i=1}^{m_n} A^{(n)}_i(t)=\frac{1}{m_n}\sum_{i=1}^{m_n}\baran_i(t), \nonumber
\end{align}
and note that the functions defined at~\eqref{xzandefs} can be written as $x^{(n)}(t)=\E\left[\Xn(t)\right]$, $z^{(n)}(t)=\E\left[\barzn(t)\right]$ and $a^{(n)}(t)=\E\left[\baran(t)\right]$.
Recall the definitions of $x(\cdot)$, $z(\cdot)$ and $a(\cdot)$ in~\eqref{xzadefs}.

\begin{lemma}
\label{GC}
\begin{enumerate}
\item For all $t\in[0,\infty)$, $x^{(n)}(t) \to x(t)$, $z^{(n)}(t) \to z(t)$ and $a^{(n)}(t) \to a(t)$ as $n\to\infty$.

\item For all $t_0\in(0,\infty)$,
\begin{equation}
\label{XGC}
\sup_{0 \le t < t_0} \left|\barxbn(t)-x(t)\right| \convp 0,
\end{equation}
\begin{equation}
\label{ZGC}
\sup_{0 \le t < t_0} \left|\tildezbn(t)-z(t)\right| \convp 0
\end{equation}
and
\begin{equation}
\label{AGC}
\sup_{0 \le t < t_0} \left|\tildeabn(t)-a(t)\right| \convp 0
\end{equation}
as $n\to\infty$.
\end{enumerate}
\end{lemma}

\begin{proof}
We prove $a^{(n)}(t) \to a(t)$ and~\eqref{AGC}; the assertions relating to $x(\cdot)$ and $z(\cdot)$ are proved similarly.  Fix $t \in [0,\infty)$ and note that $\baran(t)$ is bounded above by $\baran(\infty) = n^{-1}\sum_{j=1}^n I_j$, so
\begin{equation*}
\var\left( \baran(t) \right) \le \E\left[\baran(t)^2\right] \le n^{-2}\E\left[\left(\sum_{j=1}^n I_j\right)^2\right]=\mu_I^2+n^{-1}\sigma_I^2.
\end{equation*}
Hence
\begin{equation*}
\var \left( \tildeabn(t) \right) \le m_n^{-1} (\mu_I^2 + n^{-1} \sigma^2_I) \to 0
\end{equation*}
as $n\to\infty$, so $\tildeabn(t)-a^{(n)}(t) \convp 0$ as $n \to \infty$ by the weak law of large numbers for triangular arrays \citep[e.g.][Theorem~2.2.4]{Durrett2010}.
The probability that the Poisson process $\eta$ has a point at $t$ is zero, so Lemma~\ref{ZAnweak} and the continuous mapping theorem imply that $\baran(t) \convD \mu_I \zinf \Psi(t)$ as $ n \to \infty$.  Further, $n^{-1}\sum_{j=1}^n I_j$ has finite mean $\mu_I$, so the dominated convergence theorem implies that $a^{(n)}(t) \to a(t)=\mu_I \zinf \E\left[\Psi(t)\right]$ as $n \to \infty$. Thus, for any $t \ge 0$, $\tildeabn(t) \convp a(t)$ as $n \to \infty$.  Using a similar argument to the proof of \citet[Lemma 1]{Ball05}, \eqref{AGC} follows since $a(t)$ and $\tildeabn(t)$ are both nondecreasing in $t$; cf.\ the second Dini theorem \cite[e.g.][item 127 on pp.\ 81, 270]{PolSze1978}, which states that if a sequence of monotone (continuous or discontinuous) functions converges pointwise on a closed interval to a continuous function then it converges uniformly.
\end{proof}

Recall from Section~\ref{embedding} that the total size and severity of the modified epidemic $\Enm_{\rm mod}$ are given by $\zbn(\tildetinfn)$ and $\abn(\tildetinfn)$, respectively, where from~\eqref{Tinfinity}, $\tildetinfn$ is given by the smallest solution of $t=\tildetn_0+\tildeabn(t)$.
Suppose that $\tildetn_0 \convp 0$ as $n \to \infty$. Then it follows using {\eqref{AGC}} that with high probability, i.e.\ with probability tending to 1 as $n\to\infty$, $\tildetinfn$ is close to a solution of $a(t)=t$.
 Note that $t=0$ is a solution of $a(t)=t$, the function $a(t) = \mu_I z_{\infty} (1-\e^{-\lambda_G(1-\pi_W)t})$ is concave, $\lim_{t \to \infty}a(t)=\mu_I \zinf$ and $a'(0)=R_*$.  Thus, if $R_* \le 1$, then $t=0$ is the only solution of $a(t)=t$ in $[0, \infty)$, and if $R_* > 1$, then there is a unique solution, $\tau$ say, in $(0,\infty)$.
It is easily verified that $a'(\tau)<1$, so $(\tau, a(\tau))$ is a proper crossing point of the function $a(t)$ and the straight line of gradient one through the origin.

For the remainder of this subsection, we assume that $R_* > 1$.  Then, cf.~the proof of Corollary~\ref{finalWL} below, it follows easily from~\eqref{AGC} that $$\min\left(\tildetinfn,\left|\tildetinfn-\tau\right|\right) \convp 0 \qquad \mbox{as $n \to \infty$.}$$  Thus $\tildetinfn$ is close to either $0$ or $\tau$ with high probability.  In Section~\ref{finaloutcome}, we show for the epidemic $\Enm$ that with probability close to 1, for large $n$, if a global epidemic occurs then $\tildetinfn$ is close to $\tau$, otherwise $\tildetinfn$ is close to $0$.  We wish to study the asymptotic distribution of the final outcome of a global epidemic, so
for $\epsilon \in (0, \tau)$, let $\tildetinfnep$ denote the smallest solution in
$[\epsilon, \infty)$ of $t=\tildetn_0+\tildeabn(t)$, if one exists, otherwise let $\tildetinfnep=\tildetinfn$.

\begin{corollary}
\label{finalWL}
%\annote{}{Reworded to include $\tildetinfnep \convp \tau$ in the corollary as per FB suggestion. Reworded proof too. }
Suppose that $R_*>1$ and $\tildetn_0 \convp 0$ as $n \to \infty$.  Then, for any $\epsilon \in (0, \tau)$, as $n \to \infty$ we have
\begin{equation*}
\tildetinfnep \convp \tau , \quad
\barxbn(\tildetinfnep) \convp x(\tau), \quad
\tildezbn(\tildetinfnep) \convp z(\tau) \quad \mbox{ and } \quad
\tildeabn(\tildetinfnep)
\convp \tau.
\end{equation*}

\end{corollary}

\begin{proof}
Letting $n \to \infty$ in~\eqref{Tinfinity} and recalling that $a'(\tau)<1$, the first assertion of the corollary follows using~\eqref{AGC}.  Thus, $\tildeabn(\tildetinfnep)=\tildetinfnep-\tildetn_0 \convp \tau$ as $n \to \infty$, proving the fourth assertion.  The second and third assertions follow similarly, using the first assertion, equations~\eqref{XGC} and~\eqref{ZGC}, respectively, and the fact that the functions $x(\cdot)$ and $z(\cdot)$, respectively, are continuous.
%Letting $n \to \infty$ in~\eqref{Tinfinity} and recalling that $a'(\tau)<1$, it follows using~\eqref{AGC} that $\tildetinfnep \convp \tau$ as $n \to \infty$.  Thus, $\tildeabn(\tildetinfnep)=\tildetinfnep-\tildetn_0 \convp \tau$ as $n \to \infty$, proving the third assertion of the corollary.  The first and second assertions follow using~\eqref{XGC} and~\eqref{ZGC}, respectively, since $\tildetinfnep \convp \tau$ as $n \to \infty$ and the functions $x(\cdot)$ and $z(\cdot)$ are continuous.
\end{proof}

We return now to the single-population epidemic process $\left\{\left(Z^{(n)}(t),A^{(n)}(t)\right):t \ge 0\right\}$, defined in Section~\ref{embedding}, and prove two lemmas related to it before continuing with our analysis of the multi-community model.
Recall that $\zn(t)=n^{-1}\E[\Zn(t)]$, where $\Zn(t)$ has the same distribution as the size of the single-population epidemic $\En(t)$ (see Definition~\ref{defEpsnt}) in which the initial number of
susceptibles $S_n$ follows the binomial distribution $\Bin(n, \pi_n(t))$  with $\pi_n(t)= \e^{-\bgn n m_n t}$.  For $k=0,1,\cdots,n$, let $\mu^{(n)}_k=\E[\Zn(t) \mid S_n=n-k]$ and note that $\mu^{(n)}_k$ is independent of $t$.
Taking expectation with respect to $S_n$ yields
\begin{equation}
\label{zn1}
\zn(t)=\frac{1}{n}\sum_{k=0}^n \binom{n}{k} \mu^{(n)}_k (1-\pi_n(t))^k \pi_n(t)^{n-k}.
\end{equation}
Similar expressions exist for $\xn(t)$ and $\an(t)$.  Thus $\xn(\cdot)$, $\zn(\cdot)$ and $\an(\cdot)$ are differentiable. Further, using Wald's identity for epidemics \cite[Corollary 2.2]{Ball86}, it follows immediately from the expressions for $\zn(t)$ and $\an(t)$
that $a^{(n)}(t)=\mu_I z^{(n)}(t)$ for any $n=1,2,\cdots$ and any $t \ge 0$.
Also, it follows from~\citet[Equation~(3.10)]{Ball97}
that $\zn(t)$ admits the representation
\begin{equation}
\label{znGont}
\zn(t)=1-\sum_{i=1}^{n} \frac{(n-1)!}{(n-i)!}\phi_I(i\bwn)^{n-i} \pi_n(t)^i \alpha_i^{(n)},
\end{equation}
where the $\alpha_i^{(n)}$ are independent of $\pi_n(t)$ and strictly positive; see the discussion following~\citet[Equation~(3.7)]{Ball97}.

Recall that $\tau = \sup \{ t\geq0 : t=a(t)\}$ and $\taun=\inf\{t>0:t=\an(t)\}$ (with $\taun=0$ if that set is empty); see Section~\ref{sec:divergingm}.

\begin{lemma}
\label{tauCgce}
We have $\taun \to \tau$ as $n \to \infty$.
\end{lemma}

\begin{proof}
We prove the lemma when $R_*>1$, which is all that we require for the sequel; the proof when $R_* \le 1$ is similar. By Lemma~\ref{GC}(a) we have $\lim_{n \to \infty} \an(t) =a(t)$ for all $t\ge 0$. Since $\pi_n'(t)<0$, it follows from~\eqref{znGont} that
the function $\an(\cdot) = \mu_I \zn(\cdot)$ is increasing for each $n$, so by the second Dini theorem $\an(\cdot)$ converges uniformly to $a(\cdot)$ on any finite interval as $n \to \infty$.  It follows that
$\taun \to \tau$ as $n \to \infty$, since $\tau$ is the unique solution in $(0,\infty)$ of $a(t)=t$ and $a'(\tau)<1$.
\end{proof}

\begin{lemma}
\label{GC1}
\item For all $t_0\in(0,\infty)$,
\begin{equation}
\label{XGCD}
\lim_{n \to \infty}\sup_{0 \le t < t_0} \left|\xndash(t)-x'(t)\right| =0
\end{equation}
and
\begin{equation}
\label{ZGCD}
\lim_{n \to \infty}\sup_{0 \le t < t_0} \left|\zndash(t)-z'(t)\right| =0.
\end{equation}
\end{lemma}

\begin{proof}
We prove~\eqref{ZGCD}.
The proof of~\eqref{XGCD} is similar and left to the reader.
To prove~\eqref{ZGCD}, we show first that $\lim_{n \to \infty} \zndash(t)=z'(t)$ for each fixed $t \ge 0$.  We then show that $\zndash(\cdot)$ is decreasing for each $n$. The uniform convergence in~\eqref{ZGCD} then follows by the second Dini theorem.

Differentiating~\eqref{zn1} yields
\begin{align}
\zndash(t)&=\frac{1}{n}\pi_n'(t)\sum_{k=0}^n \binom{n}{k} \mu^{(n)}_k (1-\pi_n(t))^{k-1} \pi_n(t)^{n-k-1}[n(1-\pi_n(t))-k]\nonumber\\
&=-\bgn n^2 m_n\left[\zn(t)-\frac{1}{n}\sum_{i=0}^{n-1} \binom{n-1}{i}\mu^{(n)}_{i+1}
(1-\pi_n(t))^i \pi_n(t)^{n-1-i}\right],\label{zndasht}
\end{align}
since $\pi_n'(t)=-\bgn n m_n \pi_n(t)$.

Recalling that $\pi_n(t)=\e^{-\bgn n m_n t}$ and $\bgn n^2 m_n \to \lambda_G$ as $n \to \infty$,
\begin{equation}
\label{poissonconv}
\lim_{n \to \infty} \binom{n-1}{i}
(1-\pi_n(t))^i \pi_n(t)^{n-1-i} = \frac{(\lambda_G t)^i  \e^{-\lambda_G t}}{i!}\qquad(i=0,1,\cdots).
\end{equation}
Also, since $0 \le n^{-1} Z_n\le 1$, it follows from an obvious extension of Theorem~\ref{singleSIR}(b) to $k>1$ initial infectives that
\begin{equation}
\label{munklim}
\lim_{n \to \infty} n^{-1} \mu^{(n)}_k=(1-\pi_W^k)z_{\infty} \qquad(k=0,1,\cdots).
\end{equation}

Fix $\epsilon>0$.  There exists $n_0 \in \mathbb{N}$ such that $\sum_{i=n_0+1}^{\infty} \frac{(\lambda_G t)^i  \e^{-\lambda_G t}}{i!}< \frac{\epsilon}{3}$.  The limits~\eqref{poissonconv} and~\eqref{munklim} imply that, for all sufficiently large $n$,
\[
\left|\frac{1}{n}\sum_{i=0}^{n_0} \binom{n-1}{i}\mu^{(n)}_{i+1}
(1-\pi_n(t))^i \pi_n(t)^{n-1-i}-\sum_{i=0}^{n_0} \frac{(\lambda_G t)^i  \e^{-\lambda_G t}}{i!}(1-\pi_W^{i+1})z_{\infty}\right| < \frac{\epsilon}{3}
\]
and
\[
\left|\frac{1}{n}\sum_{i=n_0+1}^{n} \binom{n-1}{i}\mu^{(n)}_{i+1}
(1-\pi_n(t))^i \pi_n(t)^{n-1-i}\right|<\frac{\epsilon}{3}.
\]
Further, $\sum_{i=n_0+1} ^{\infty} \frac{(\lambda_G t)^i  \e^{-\lambda_G t}}{i!} (1-\pi_W^{i+1})z_{\infty} < \frac{\epsilon}{3}$, since $(1-\pi_W^{i+1})z_{\infty}\in [0,1]$ for all $i$, so
\[
\left|\frac{1}{n}\sum_{i=0}^{n-1} \binom{n-1}{i}\mu^{(n)}_{i+1}
(1-\pi_n(t))^i \pi_n(t)^{n-1-i}-\sum_{i=0}^{\infty}\frac{(\lambda_G t)^i  \e^{-\lambda_G t}}{i!}(1-\pi_W^{i+1})z_{\infty}\right|<\epsilon\]
for all sufficiently large $n$.  Since $\epsilon>0$ is arbitrary it follows that
\[
\lim_{n \to \infty}\frac{1}{n}\sum_{i=0}^{n-1} \binom{n-1}{i}\mu^{(n)}_{i+1}
(1-\pi_n(t))^i \pi_n(t)^{n-1-i}= \sum_{i=0}^{\infty}\frac{(\lambda_G t)^i  \e^{-\lambda_G t}}{i!}(1-\pi_W^{i+1})z_{\infty}.
\]

Recall that $\bgn n^2 m_n \to \lambda_G$ and $\zn(t)\to z(t)=z_{\infty}(1- \e^{-\lambda_G(1-\pi_W)})$; by assumption and from Lemma~\ref{GC}(a), respectively.  It then follows from~\eqref{zndasht} that
\begin{align*}
\lim_{n \to \infty}\zndash(t)&=-\lambda_G\left[z(t)-z_{\infty}\left(1-\pi_W \e^{-\lambda_G(1-\pi_W)t}\right)\right]\\
&=\lambda_G z_{\infty}(1-\pi_W) \e^{-\lambda_G(1-\pi_W)t}\\
&=z'(t).
\end{align*}

Finally, that $\zndash(\cdot)$ is decreasing follows immediately from~\eqref{znGont} and the fact that $\pi_n''(t)>0$.
\end{proof}

Returning again to the embedding construction, it is convenient to introduce some vector notation.  For $i=1,2,\cdots,m_n$ and $t \ge 0$, let
\begin{equation*}
\bRn_i(t)=\left(R_{i1}^{(n)}(t),R_{i2}^{(n)}(t),R_{i3}^{(n)}(t)\right)^{\top}
=\left(X^{(n)}_i(t),\barzn_i(t),\baran_i(t)\right)^{\top}.
\end{equation*}
For $t \ge 0$, let
$$\bRbn(t)=\left(\rbnone(t),\rbntwo(t),\rbnthree(t) \right)^{\top}=\sum_{i=1}^{m_n}\bRn_i(t),$$ so
\begin{equation*}
\E\left[\bRbn(t)\right]=m_n\left(x^{(n)}(t), z^{(n)}(t), a^{(n)}(t)\right)^{\top}=m_n \brn(t),
\end{equation*}
say, and let $\br(t)=\lim_{n \to \infty}\brn(t)$, so $\br(t)=\left(x(t),z(t),a(t)\right)^{\top}$ (see Lemma~\ref{GC}(a)). Further, for $t,s\ge0$, let $\Sigma(t,s)=\left[\sigma_{ij}(t,s)\right]$, where for $i,j=1,2,3$ we define
\begin{equation*}
\sigma_{ij}(t,s)=\lim_{n \to \infty}m_n^{-1}{\rm cov}\left(\rbni(t),\rbnj(s) \right)=
\lim_{n \to \infty}{\rm cov}\left(R_{1i}^{(n)}(t),R_{1j}^{(n)}(s) \right).
\end{equation*}
Noting that the processes $\{\bRn_i(t) : t\geq0\}$ ($i=1,2,\cdots,m_n$) are independent, a simple calculation using
Lemma~\ref{ZAnweak}, together with the continuous mapping and dominated convergence theorems (cf.~the proof of Lemma~\ref{GC}(a)), gives
\begin{equation}
\label{sigmatt}
\Sigma(t,s)={\rm cov}\left(\Psi(t),\Psi(s)\right)\bb{\bb}^{\top}=x(t)(1-x(s))\bb{\bb}^{\top} \quad(0 \le t \le s),
\end{equation}
where $\bb=(1,\zinf,\mu_I \zinf)^\top$ as defined in equation~\eqref{bbDef}.

In the following theorem,  $\convw$ denotes weak convergence in the space of bounded functions from $[0,T]$ to $\mathbb{R}^3$ endowed with the supremum metric.
We use this toplogy, rather than the weaker Skorohod topology, as in our setting easily checkable conditions for tightness under the stronger topology exist
in the literature.  For $T>0$, let $\bRbnT=\{\bRbn(t):0 \le t \le T\}$, $\brnT=\{\brn(t):0 \le t \le T\}$ and $\bW_T=\{\bW(t): 0 \le t \le T\}$, where
\begin{equation}
 \label{Wdef}
 \{\bW(t): t \ge 0\} =
\{(W_1(t),W_2(t),W_3(t))^{\top}: t \ge 0\}
\end{equation}
is a zero-mean Gaussian process with covariance function $\Sigma(t,s)$ $(t,s \ge 0)$.

\begin{theorem}
\label{weakconvr}
For any $T>0$,
\begin{equation}
\label{weakconvreq}
\frac{1}{\sqrt{m_n}}\left(\bRbnT-m_n \brnT\right) \convw \bW_T \qquad \mbox{as } n \to \infty.
\end{equation}
\end{theorem}

\begin{proof}
We use the standard method of proving convergence of finite-dimensional distributions and asymptotic tightness; see e.g.~\citet[Theorem 1.5.4]{Vaar96}. The details are lengthy, so we defer them to Appendix~\ref{appA}.
\end{proof}

The next theorem gives a central limit theorem for the final outcome of the embedded multi-community epidemic and is used to prove Theorem~\ref{multifinalmain}.

\begin{theorem}
\label{finalCLT}
Suppose that $R_*>1$ and $\sqrt{m_n}\tildetn_0 \convp 0$  as $n \to \infty$.  Then,
for any $\epsilon \in (0, \tau)$,
\begin{equation}
\label{finalCLT1}
\sqrt{m_n}
\left(\begin{array}{c}
\barxbn(\tildetinfnep)-\xn(\taun)\\
\tildezbn(\tildetinfnep)-\zn(\taun)\\
\tildeabn(\tildetinfnep)-\an(\taun)\\
\end{array}
\right)
\convD \bN \qquad \mbox{as } n \to \infty,
\end{equation}
where $\bN$ is a three-dimensional zero-mean normal random vector with variance-covariance matrix given by~\eqref{finalSigmaZ}.
\end{theorem}

\begin{proof}
First note that, by definition,
\begin{equation}
\label{finalembed}
\sqrt{m_n}\left(\barxbn(\tildetinfnep),\tildezbn(\tildetinfnep),\tildeabn(\tildetinfnep)\right)^{\top}
=\frac{1}{\sqrt{m_n}}\bRbn(\tildetinfnep).
\end{equation}
Recall from Corollary \ref{finalWL} that $\tildetinfnep \convp \tau$ as $n \to \infty$.  Thus, setting $T>\tau$ in Theorem~\ref{weakconvr} and using Slutsky's lemma and the continuous mapping theorem \citep[e.g.][Example 1.4.7 and Theorem 1.3.6, respectively]{Vaar96} implies that
\begin{equation}
\label{CLT1}
\frac{1}{\sqrt{m_n}}\left(\bRbn(\tildetinfnep)-m_n \brn(\tildetinfnep) \right) \convD \bW(\tau) \qquad \mbox{as } n \to \infty,
\end{equation}
where $\{\bW(t) : t\geq0\}$ is defined in \eqref{Wdef}.

Now
\begin{align}
\label{MVT}
\sqrt{m_n}\left(\tildeabn(\tildetinfnep)-\an(\taun) \right)=&\sqrt{m_n}\left(\tildeabn(\tildetinfnep)-a^{(n)}(\tildetinfnep)\right)\\
&+
\sqrt{m_n}\left(a^{(n)}(\tildetinfnep)-\an(\taun)\right) . \nonumber
\end{align}
By the mean value theorem, there exists $\thetan$ lying between $\tildetinfnep$ and $\taun$, such that $\an(\tildetinfnep)-\an(\taun)=\left(\tildetinfnep-\taun\right) \andash(\thetan)$.
Recall from~\eqref{Tinfinity} that $\tildetinfnep=\tildetn_0+\tildeabn(\tildetinfnep)$ and from the definition of $\taun$ that $\an(\taun)=\taun$.  Thus,
\begin{align*}
\sqrt{m_n}\left(\an(\tildetinfnep)-\an(\taun)\right)&=\sqrt{m_n}\left(\tildeabn(\tildetinfnep)-\an(\taun) \right)\andash(\thetan)\\
&\qquad+\sqrt{m_n}\,\tildetn_0 \andash(\thetan),
\end{align*}
which on substituting into~\eqref{MVT} and rearranging yields
\begin{multline}
\label{MVT1}
  \sqrt{m_n} \left(1-\andash(\thetan)\right)\left(\tildeabn(\tildetinfnep)-\an(\taun)\right) \\
=  \sqrt{m_n}\left(\tildeabn(\tildetinfnep)-\an(\tildetinfnep)\right) +
\sqrt{m_n}\,\tildetn_0 \andash(\thetan).
\end{multline}

Now~\eqref{finalembed} and~\eqref{CLT1} imply that $\sqrt{m_n}\left(\tildeabn(\tildetinfnep)-a^{(n)}(\tildetinfnep)\right) \convD W_3(\tau)$ as $n \to \infty$.  By the sandwich principle, $\thetan \convp \tau$ as $n \to \infty$, since $\taun \to \tau$ (by Lemma~\ref{tauCgce}) and $\tildetinfnep \convp \tau$.  Hence, recalling $a^{(n)}(\cdot)=\mu_I z^{(n)}(\cdot)$ and using~\eqref{ZGCD} in Lemma~\ref{GC1},
$\andash(\thetan) \convp a'(\tau)$ as $n \to \infty$. Further, $\sqrt{m_n}\,\tildetn_0 \andash(\thetan) \convp 0$, since $\sqrt{m_n}\tildetn_0 \convp 0$.
Recall that $a'(\tau)<1$, so substituting these observations into~\eqref{MVT1} yields after application of Slutsky's theorem that
\begin{equation}
\label{Aconvd}
\sqrt{m_n}\left(\tildeabn(\tildetinfnep)-\an(\taun)\right) \convD (1-a'(\tau))^{-1}W_3(\tau)
\qquad \mbox{as } n \to \infty.
\end{equation}

Arguing in a similar fashion, using~\eqref{XGCD} instead of~\eqref{ZGCD}, shows that
\begin{align*}
\sqrt{m_n}&\left(\barxbn(\tildetinfnep)-\xn(\taun) \right)=\sqrt{m_n}\left(\barxbn(\tildetinfnep)-\xn(\tildetinfnep)\right)\\
&\qquad+
\xndash(\thetant)\left[\sqrt{m_n}\left(\tildeabn(\tildetinfnep)-\an(\taun)\right)
+\sqrt{m_n}\tildetn_0\right],
\end{align*}
for some $\thetant$ lying between $\tildetinfnep$ and $\taun$.  Using~\eqref{finalembed}, \eqref{CLT1}, \eqref{Aconvd} and similar observations to those above, it then follows that, as $n \to \infty$,
\begin{equation}
\label{Xconvd}
\sqrt{m_n}\left(\barxbn(\tildetinfnep)-\xn(\taun)\right) \convD W_1(\tau)+(1-a'(\tau))^{-1}x'(\tau)W_3(\tau).
\end{equation}
A similar argument  shows that
\begin{equation}
\label{Zconvd}
\sqrt{m_n}\left(\tildezbn(\tildetinfnep)-\zn(\tau)\right) \convD W_2(\tau)+(1-a'(\tau))^{-1}z'(\tau)W_3(\tau).
\end{equation}
Expressing~\eqref{Aconvd}-\eqref{Zconvd} in matrix form, and recalling~\eqref{finalembed},~\eqref{CLT1},~\eqref{sigmatt} and~\eqref{weakconvreq}, yields~\eqref{finalCLT1}, with $\Sigma_{\bN}=x(\tau)(1-x(\tau))B\bb \bb^{\top}B^{\top}$, where
%\begin{equation*}
%B=\left[\begin{array}{ccc}
%1&0& (1-a'(\tau))^{-1}x'(\tau)\\
%0 & 1 & (1-a'(\tau))^{-1}z'(\tau)\\
%0 & 0 & (1-a'(\tau))^{-1}
%\end{array}
%\right].
%\end{equation*}
\begin{equation*}
\setlength{\arraycolsep}{.3cm}
B=
\begin{bmatrix}
1&0& (1-a'(\tau))^{-1}x'(\tau)\\
0 & 1 & (1-a'(\tau))^{-1}z'(\tau)\\
0 & 0 & (1-a'(\tau))^{-1}
\end{bmatrix}.
\end{equation*}

Recall from~\eqref{xzadefs} that $x(t)=1- \e^{-\lg(1-\pi_W)t}, z(t)=\zinf x(t)$ and $a(t)=\mu_I \zinf x(t)$, so $x'(t)=\lg(1-\pi_W)(1-x(t))$. Simple algebra yields $$B \bb=\left[1-\zinf \mu_I \lg (1-\pi_W)(1-x(\tau))\right]^{-1}\bb$$ and~\eqref{finalSigmaZ} follows on recalling the expression~\eqref{equ:Rstar} for $R_*$.
\end{proof}

\subsubsection{Final outcome of epidemic $\Enm$; proofs of Theorems \ref{multifinalmain} and \ref{multifinalmain1}}
\label{finaloutcome}

\begin{proof}[Proof of Theorem \ref{multifinalmain}]
The proof proceeds in two stages.  We show first that given any $\epsilon>0$, there exists $\epsilon_1>0$ such that
$\P\left(\tildetinfn \ge \epsilon_1 \mid  \Gn \right) \ge 1-\epsilon$ for all sufficiently large $n$.  Then we use Corollary~\ref{finalWL} and Theorem~\ref{finalCLT} for the epidemic starting from when $\Enm$ has first infected at least $\log m_n$ communities.

Recall from Section~\ref{early} the branching processes $\BBcheck$ and $\BBncheck$, which approximate the process of infected communities in the early stages of the epidemic.
For $\delta \in (0,1)$, while
the number of communities, excluding community $0$, infected in $\Enm$ is not more than $\delta m_n$, the probability that a global infectious contact is
with a previously uninfected community is at least $1-\delta$, so the
process of infected communities in $\Enm$
is stochastically bounded below by the branching process $\BBn(\delta)=\BB(m_n n \bgn (1-\delta) A_n)$.
Let $\checkzn_C(\delta)$ denote the total size of the branching process $\BBn(\delta)$. Then,
\begin{equation}
\label{lowerbp}
\P\left(\hatznc \ge \delta m_n\right) \ge \P\left(\checkzn_C(\delta) \ge \delta m_n\right) \ge 1-\pi_n(\delta),
\end{equation}
where $\pi_n(\delta)$ is the extinction probability of $\BBn(\delta)$. Now
$\pi_n(\delta) \to \pi_G(\delta)$ as $n \to \infty$, where $\pi_G(\delta)$ is the extinction probability of $\BB(\lg (1-\delta) \bar{A})$,
and $\pi_G(\delta) \downarrow \pi_G$ as $\delta \downarrow 0$.  Further, $\pi_G<1$, since $R_*>1$.  Fix $\epsilon>0$ and recall from Section~\ref{early} that
\[
\P\left(\Gn\right)=\P\left(\hatznc \ge \log m_n\right) \to 1-\pi_G \qquad \mbox{as } n \to \infty.
\]
Then it follows from~\eqref{lowerbp} that there exist $\delta>0$ and $n_0 \in \mathbb{N}$ such that
\begin{equation}
\label{conddeltamn}
\P\left(\hatznc \ge \delta m_n  \mid  \Gn\right) \ge 1-\frac{\epsilon}{2} \qquad \mbox{for all }n \ge n_0.
\end{equation}

Now construct a realisation of $\Enm \mid  \Gn$ by running  $\Enm$ until at least $\log m_n$ communities have been infected. If the realisation of $\Enm$ infects fewer than $\log m_n$ communities, repeat the construction independently until an epidemic that infects at least $\log m_n$ communities is obtained.
Now use the embedding construction of Section~\ref{embedding} to construct two coupled processes on a population of $m_n'=m_n+1-\ceil{\log m_n}$ communities (i.e.~the ones not yet infected by $\Enm$) using the same $\left(Z^{(n)}_i(t),A_i^{(n)}(t)\right)$ $(i=1,2,\cdots,m_n')$, one with $\Tn_0=\TnL_0$, where $\TnL_0$ is the severity of $\Enm$ when at least $\log m_n$ communities first become infected, and one with $\Tn_0=\TnU_0$, where $\TnU_0$ is given by $\TnL_0+\AnU_0$ and $\AnU_0$ is obtained as follows.
Let $s_1^{(n)},s_2^{(n)},\cdots,s_{\kn}^{(n)}$
be the number of susceptibles remaining in each of the first $\kn$ communities infected by the epidemic when the $\kn$-th community becomes infected.  Then $\AnU_0$ is distributed as the sum of $\sum_{i=1}^{\kn} s_i^{(n)}$ independent copies of $I$, independently of the rest of the construction.  The embedding construction can be modified, by allowing appropriate infection of community $0$, so that $\Enm \mid  \Gn$ is stochastically bounded between the processes with $\Tn_0=\TnL_0$ and $\Tn_0=\TnU_0$.

The weak law of large numbers implies that $\sqrt{m_n'} \frac{1}{n m_n'}\TnU_0 \convp 0$ as $n \to \infty$, hence also $\sqrt{m_n'} \frac{1}{n m_n'}\TnL_0 \convp 0$.  Consider the epidemic obtained by running $\Enm$ until $\ceil{\log m_n}$ communities are infected and then using the embedding construction with $\Tn_0=\TnL_0$.  Let $\hatZnL_C$ be the number of communities infected by this epidemic and $\checkZnL_C=\hatZnL_C-\ceil{\log m_n}$ be the number of communities infected whilst using the embedding construction.  Note that the branching process $\BBn(\delta)$ is also a lower bound for this epidemic and~\eqref{conddeltamn} holds with $\hatznc$ replaced by $\hatZnL_C$.  It then follows easily that, given any $\epsilon>0$, there exist $\delta'>0$ and $n_1 \in \mathbb{N}$ such that
\begin{equation}
\label{conddeltamn1}
\P\left(\hatZnL_C \ge \delta' m_n'\right) \ge 1-\frac{\epsilon}{2} \qquad \mbox{for all }n \ge n_1.
\end{equation}

In the embedding construction of this epidemic, for $t \ge 0$, let $\barybnL(t)=\frac{1}{m_n'}\sum_{i=1}^{m_n'} 1_{\{ Z^{(n)}_i(t) > 0\}}$ (cf.~\eqref{xbullet}).  Arguing as in the proof of Lemma~\ref{GC}(a) shows that, for all $t > 0$, as $n \to \infty$, $\barybnL(t) \convp y(t)=1-\e^{-\lg t}$. Let $\epsilon_1=-\frac{1}{\lambda_G} \log(1-\frac{\delta'}{2})$.  Define $\tildetinfnL$ analogously to $\tildetinfn$.  Then
\begin{equation}
\label{tinfLeps1}
\P\left(\tildetinfnL < \epsilon_1 \right)=\P\left(\tildetinfnL < \epsilon_1, \hatZnL_C \ge \delta' m_n'\right)+\P\left(\tildetinfnL < \epsilon_1, \hatZnL_C <\delta' m_n'\right).
\end{equation}
Now, $\P\left(\tildetinfnL < \epsilon_1, \hatZnL_C \ge \delta' m_n'\right) \le \P\left(\barybnL(\epsilon_1)\ge \delta'\right)$, since $\barybnL(t)$ is nondecreasing in $t$, and $\barybnL(\epsilon_1) \convp y(\epsilon_1)=\frac{\delta'}{2}$ as $n \to \infty$.  Thus, the first term on the right-hand side of~\eqref{tinfLeps1} is smaller than $\frac{\epsilon}{2}$ for all sufficiently large $n$.  The second term is smaller than $\frac{\epsilon}{2}$ for all $n \ge n_1$, using~\eqref{conddeltamn1}, so in an obvious notation that there exists $\epsilon_1>0$ and $n_2\in \mathbb{N}$ such that for all $n \ge n_2$,
\begin{equation}
\label{ptinfLU}
\P\left(\tildetinfnU \ge \epsilon_1\right)\ge \P\left(\tildetinfn \ge \epsilon_1 \mid  \Gn\right)\ge\P\left(\tildetinfnL \ge \epsilon_1\right)\ge 1-\epsilon.
\end{equation}
(The first two probability inequalities in~\eqref{ptinfLU} follow immediately from $\Enm \mid  \Gn$ being stochastically bounded between the embedding processes with $\Tn_0=\TnL_0$ and $\Tn_0=\TnU_0$.)

Recall that $\znmt$ is the total number of individuals infected in $\Enm$ and let  $\znmt=\znma+\znmb$, where $\znma$ is the total number of individuals ever infected in the first $\kn$ other communities infected by $\Enm$.  Then,
\begin{equation}
\label{ztildeab}
\tznmt=\tznma+\frac{m_n'}{m_n}\tznmb,
\end{equation}
where $\tznma=\frac{1}{n m_n} \znma $ and $\tznmb=\frac{1}{nm_n'}\znmb$.
Now $0 \le \tznma \le \frac{\kn}{m_n}$, so $\tznma \convp 0$ as $n \to \infty$.  Further, for $n \ge n_2$, in view of~\eqref{ptinfLU}, with probability at least $1-\epsilon$,
\begin{equation}
\label{ztildebbounds}
\tildezbn(\tildetinfnLep) \le \tznmb \le \tildezbn(\tildetinfnUep),
\end{equation}
where $\tildezbn(\tildetinfnLep)$ and $\tildezbn(\tildetinfnUep)$ are as in Corollary~\ref{finalWL} and Theorem~\ref{finalCLT} in Section~\ref{finmod}, but with $m_n$ replaced by $m_n'=m_n+1-\kn$.  By Corollary~\ref{finalWL}, $\tildezbn(\tildetinfnLep) \convp z(\tau)$ and $\tildezbn(\tildetinfnUep) \convp z(\tau)$ as
$n \to \infty$, so since $\epsilon>0$ can be arbitrarily small given that $\Gn$ occurs, and $\lim_{n \to \infty}\frac{m_n'}{m_n}=1$, it follows using~\eqref{ztildeab} that
$\tznmt \mid \Gn \convp z(\tau)$ as $n \to \infty$.  The other results in part (a) are proved similarly.

Turning to part (b),
\begin{align*}
\sqrt{m_n}&\left(\tznmt-\zn(\taun)\right) \\
 & =\sqrt{m_n}\tznma+\sqrt{\frac{m_n'}{m_n}}\left[\sqrt{m_n'}\left(\tznmb-\zn(\taun)\right)\right] -\frac{m_n-m_n'}{\sqrt{m_n}}\zn(\taun).
\end{align*}
Now $0 \le \sqrt{m_n}\tznma \le \frac{\kn}{\sqrt{m_n}}$, so $ \sqrt{m_n}\tznma \convp 0$ as $n \to \infty$, $\lim_{n \to \infty}\sqrt{\frac{m_n'}{m_n}}=1$,
$\lim_{n \to \infty}\frac{m_n-m_n'}{\sqrt{m_n}}=0$ and $\lim_{n \to \infty} \zn(\taun)=z(\tau)$, so $\sqrt{m_n}\left(\tznmt-\zn(\taun)\right)$ and
$\sqrt{m_n'}\left(\tznmb-\zn(\taun)\right)$ have the same limiting distribution by Slutsky's theorem.  Using similar decompositions and bounds to~\eqref{ztildeab} and~\eqref{ztildebbounds}  for $\barznmc$ and
$\tanm$, and the fact that $\epsilon>0$ can be arbitrarily small given that $\Gn$ occurs, part (b) follows using two applications of~Theorem~\ref{finalCLT} applied to the process with $m_n'$ communities, one with $\tildetinfnep=\tildetinfnLep$ and one with $\tildetinfnep=\tildetinfnUep$.
\end{proof}

The proof of Theorem~\ref{multifinalmain1} requires Lemmas~{\ref{pn_znhat_convlemma}} and~{\ref{znconvlem}}, which concern the single-community epidemic $\En(t)$ defined in Definition~\ref{defEpsnt}.
Lemma~\ref{znconvlem}, whose proof requires Lemma~\ref{pn_znhat_convlemma}(a), is a new result concerning the rate of convergence of the expected fraction of the population infected in a homogeneously mixing population SIR epidemic to its limiting value as the population size $n \to \infty$.

\begin{lemma}
\label{pn_znhat_convlemma}
Suppose that $R_0>1$ and conditions (C1)--(C4) hold, and define the functions $p$, $\png$ and $\hatzng$ on the domain $[0,\infty)$ by $p(t)=1-\e^{-\lambda_G(1-\pi_W)t}$, $$\png (t) = \P\left(\Zn (t) \ge \ngam \right) \quad \mbox{and}\quad \hatzng (t) = \E\left[\barzn(t) \mid \Zn (t) \ge \ngam\right].$$
\begin{enumerate}
\item For any $\gamma\in(0,\delta/4)$, with $\delta\in(0,2)$ as in (C1), and any $T>0$,
\begin{equation}
\label{pnconv}
\lim_{n \to \infty} \sqrt{m_n} \sup_{0 \le t \le T}\left|\png (t)-p(t)\right|=0.
\end{equation}
and
\begin{equation}
\label{znhatconv}
\lim_{n \to \infty} \sqrt{m_n} \sup_{0 \le t \le T}\left|\hatzng (t)- \zinf \right|=0.
\end{equation}

\item If condition (C5) also holds then, for any $T>0$,
\begin{equation}
\label{pnconvlog}
\lim_{n \to \infty} \sqrt{m_n} \sup_{0 \le t \le T}\left|\P\left(\Zn (t) \ge \log n \right)-p(t)\right|=0.
\end{equation}
\end{enumerate}
\end{lemma}

\begin{proof}
The proof is lengthy and thus deferred to Appendix~\ref{appB}. We note here that only (C1)--(C3) are required for~\eqref{pnconv}.
\end{proof}
Note that, since $\delta\in(0,2)$, we must have $\gamma<\frac12$ in part (a) of the Lemma. Note also that the quantity $p(t)$ defined in the lemma is the same as $x(t)$; however the $p$ notation is more natural in the context of these lemmas and their proofs.  Now recall that $\zn(t)=n^{-1}\E[\Zn (t)]$.
\begin{lemma}
\label{znconvlem}
Suppose $R_0>1$ and that (C1)--(C4) hold. Then for any $T>0$,
\begin{equation*}
\lim_{n \to \infty} \sqrt{m_n} \sup_{0 \le t \le T}\left| z^{(n)}(t)-z(t)\right|=0.
\end{equation*}
\end{lemma}

\begin{proof}
Suppose that $\delta$ satisfies condition (C1).  Fix $\gamma \in (0,\frac{\delta}{4})$, so then we have $m_n n^{4\gamma-2} \to 0$ as $n \to \infty$.
In addition to the notation $p(t)$, $\png (t)$ and $\hatzng (t)$ defined in Lemma~\ref{pn_znhat_convlemma} we define $\hatwng (t) = \E\left[\barzn(t) \mid \Zn (t) < \ngam\right]$.
Then
\[
z^{(n)}(t)=(1-\png (t))\hatwng (t)+\png (t)\hatzng (t),
\]
so recalling from~\eqref{xzadefs} that $z(t)=p(t)\zinf$,
\begin{equation}
\label{modzn-z}
\left| z^{(n)}(t)-z(t)\right| \le \left|\left(1-\png (t)\right)\hatwng (t) \right|+
\left| \png (t)\hatzng (t) -p(t) \zinf \right|.
\end{equation}

Now $0 \le \hatwng (t) \le n^{-1} n^\gamma$ for all $t \in [0,T]$, so
\begin{equation}
\label{limit00}
\lim_{n \to \infty} \sqrt{m_n} \sup_{0 \le t \le T}\left|\left(1-\png (t)\right)\hatwng (t) \right| =0,
\end{equation}
since $m_n n^{2(\gamma-1)} \to 0$ as $n \to \infty$. Also,
\begin{equation*}
\left| \png (t)\hatzng (t) -p(t) \zinf \right| \le \hatzng (t)\left|\png (t)-p(t)\right|+p(t)\left|\hatzng (t)-\zinf \right|.
\end{equation*}
Now $0 \le \hatzng (t) \le 1$ and $0 \le p(t) \le 1$ for all $t \in [0,T]$, so in view of~\eqref{modzn-z} and~\eqref{limit00}, the lemma follows using Lemma~\ref{pn_znhat_convlemma}.
\end{proof}

\begin{proof}[Proof of Theorem \ref{multifinalmain1}]
Recall that, for $t \ge 0$, $\xn(t)=\P\left(\Zn (t) \ge \log n \right)$ and $x(t)=p(t)$ , and from Lemma~\ref{tauCgce} that $\taun \to \tau$ as $n \to \infty$. Hence Lemma~\ref{pn_znhat_convlemma} implies that $\sqrt{m_n}(\xn(\taun)-x(\tau)) \to 0$ as $n \to \infty$.  A similar argument using Lemma~\ref{znconvlem} yields $\sqrt{m_n}(\zn(\taun)-z(\tau)) \to 0$ as $n \to \infty$.
Thus~\eqref{mnxnznanconv} holds and Theorem~\ref{multifinalmain1} follows.
\end{proof}

\section{Concluding comments}
\label{conclusions}
We have proved LLN and CLT type results for the final outcome of an SIR epidemic on a population of weakly connected communities. These limit theorems are valid for fixed or large numbers of large communities, and brief numerical results suggest that they provide good working approximations for the final outcome of the model in finite populations. The asymptotic form of the weak connectivity that we use leads, in the case of a fixed number of communities, to LLN type results with randomness in the limits and CLT type results with a mixture of normal random variables in the limits.
We note that in the case of a fixed number of communities a deterministic approximation to the model is not available: the scalings of the infection rates are such that the proportion ultimately infected will always have randomness from global (between-community) infections.
The additional randomness (compared to that observed in most models with similar structure) in the outcomes of our model contributes to the development of epidemic models that exhibit more variability in their outcomes than many models in the literature. The temporal properties of our model also exhibit such increased variability; this will be explored in future work.

We note further that in the case of a fixed infectious period our work gives results concerning a version of the planted partition or stochastic block model with a different scaling of edge probabilities than in most of the random graph literature. This model (see, e.g.~\cite{MosNeeSly2015}, Example 4.3 of \cite{BolJanRio2007}, Section 9.3.1 of \cite{vdH2021}) can be viewed as a multi-type version of the Gilbert / Erd\H{o}s-R\'{e}nyi graph $G(n,p)$, with self-type edge probabilities $a/n$ and between-type edge probabilities $b/n$, usually with $b<a$ to model community structure in social networks. Our model can be interpreted as a variant of this with between-type edge probabilities $b/n^2$.

There are several generalisations of our model that are likely to be of interest. Allowing for unequal community sizes would increase realism and be fairly straightforward to accommodate. In principle, allowing some variation in the asymptotic local and/or global infection rates should also be possible. For fixed $m_n=m$ we might allow $\lw$ to be different for each community and $\lg$ to be different (possibly asymmetrically) for each pair of communities. For $m_n$ increasing with $n$ each community could be one of various `types', each with characteristic within- and between-community infection rates (in the spirit of the aforementioned stochastic block model and multi-type epidemic models). Whilst analogues of the formulae in our main theorems to allow for these generalisations readily spring to mind, the additional notation and details of the proofs would doubtless be rather cumbersome.

The CLT in Theorem~\ref{multifinalmain} requires weaker conditions than that in Theorem~\ref{multifinalmain1} but is not useful in practical applications as the mean vector, which depends on $n$, is not available.  A natural further question is whether one can find a useful approximation, $z^{(n)}_A(t)$ say,  to  $z^{(n)}(t)$, so that $a^{(n)}_A(t) = \mu_I z^{(n)}_A(t)$ approximates $a^{(n)}(t)$, such that e.g.
\begin{equation*}
\lim_{n \to \infty} \sqrt{m_n} (z^{(n)}(\tau^{(n)})- z^{(n)}_A(\tau_A^{(n)}))=0,
\end{equation*}
where $\tau_A^{(n)}$ solves $t = a^{(n)}_A(t)$, under weaker conditions than (C1) (cf.\ equation~\eqref{mnxnznanconv}).

We expect that if community sizes, or even local infection dynamics, differ between communities then we can use ideas from Poissonian random graphs \citep{NorRei2006} to derive results similar to ours. To do this we still consider a sequence of epidemic models indexed by $n$ ($n= 1,2,\cdots$), with $m_n \to \infty$ as $n\to \infty$ and a per-pair global contact rate that still satisfies the second equation in~\eqref{infectionrates}. Let $n_i^{(n)}$ ($i=1,2,\cdots, m_n$) be the size of the $i$-th community in the model indexed by $n$ and assume that, for every $i=1,2,\cdots, m_n$, $n_i^{(n)}/n \to c_i$  for some strictly positive constant $c_i$ as $n \to \infty$. Furthermore, $\pi_i^{(n)}$, the probability that a single introduction in the $i$-th community ($i=1,2,\cdots$) leads to a large local outbreak converges to some $\pi_i$, and if $\pi_i>0$, then the number of individuals infected in such a large local outbreak divided by $n_i^{(n)}$ converges in probability to some strictly positive constant $z_i$. Let $A_i^{(n)}$ be the severity of this large outbreak, i.e.\ $A_i^{(n)}$ is the sum of the infectious periods of the individuals infected in the large outbreak.  By the above convergence assumptions $A_i^{(n)}/n \to z_i c_i \mathbb{E}[I]$ as $n \to \infty$. We assume that the joint empirical distribution of these large outbreak probabilities and normalised final sizes converge to some random vector with support contained in $[0,1]^2$. If we ignore all minor outbreaks (as we did in Section \ref{sec:multiComm}), then we can study the epidemic indexed by $n$ using a directed random graph representation \cite[e.g.][p.\ 64]{AndBri2000} in which the vertices represent communities and edges global contacts that would spark a large outbreak in the community at the `head' of the edge, if there has not been a large outbreak there before. For $i,j = 1,2,\cdots, m_n$ there is a directed edge from $i$ to $j$ with probability $1-\exp(-\beta_G^{(n)} A_i^{(n)}n_j^{(n)} \pi_j^{(n)})$, and for large $n$ this probability is well approximated by $1-\exp(-\lambda_G z_i c_i \mu_I c_j \pi_j)$. This random graph can be studied as a directed version of a Poissonian random graph \citep{NorRei2006} or a special case of an inhomogeneous random graph \citep{BolJanRio2007}. Because the theory developed in those papers is for undirected graphs, some work has to be done in checking which results can be translated to directed graphs, but we expect that this will cause no problems for LLN type results similar to Theorem 3.1 of \cite{BolJanRio2007}.

Lastly, we note that in the course of proving Theorem~\ref{multifinalmain1} we have established a new result concerning the rate of convergence of the expected fraction of individuals infected in a standard SIR epidemic to its asymptotic limit, Lemma~\ref{znconvlem}, which may be of independent interest. Lemma~\ref{znconvlem} is actually for a variant of the standard stochastic SIR epidemic in which the number of initial infectives follows a binomial distribution, though it is clear from the proof that a corresponding result holds if the number of initial infectives is held fixed independent of $n$.

%%%%%%%%%%%%%%%%%%%%%%%%%%%%%%%%%%%%%%%%%%%%%%
%% Single Appendix:                         %%
%%%%%%%%%%%%%%%%%%%%%%%%%%%%%%%%%%%%%%%%%%%%%%
\begin{appendix}
\section{Proof of Theorem~\lowercase{\ref{weakconvr}}}
\label{appA}

As stated in the main text, we need to show convergence of the finite-dimensional distributions and asymptotic tightness of the sequence
$$\frac{1}{\sqrt{m_n}}\left(\bRbnT-m_n \brnT\right) \qquad (n=1,2,\cdots),$$
where $\bRbnT$ and $\brnT$ are defined just before the statement of Theorem~\ref{weakconvr}.

\subsection{Finite-dimensional distributions}
We show that, as $n\to\infty$, the finite-dimensional distributions of $\frac{1}{\sqrt{m_n}}\left(\bRbnT-m_n \brnT\right)$ converge to those of $\bW_T$, using the Cram\'{e}r-Wold device \cite[e.g.][pp.~48--49]{Bill68}.
For $j=1,2,3$, let $k_j \in \mathbb{N}$ and, for $l=1,2,\cdots,k_j$, let $\lambda_{jl} \in \mathbb{R}$ and $t_{jl} \in [0,T]$.  For $n=1,2,\cdots$, let
\begin{equation*}
\rnlam =\sum_{j=1}^3 \sum_{l=1}^{k_j} \lambda_{jl} \rbnj(t_{jl})=\sum_{i=1}^{m_n}\rnlami,
\end{equation*}
where
\begin{equation}
\label{rnlambdai}
\rnlami =\sum_{j=1}^3 \sum_{l=1}^{k_j}\lambda_{jl}R_{ij}^{(n)}(t_{jl}).
\end{equation}

Note that $\rnlamone,\rnlamtwo,\cdots,\rnlammn$ are i.i.d.  Now choose $\alpha>0$ such that $\E[I^{2+\alpha}]<\infty$ and, for $n=1,2,\cdots$, let
\begin{equation}
\label{Andefn}
B_n=\sum_{i=1}^{m_n} \E\left[\left|\rnlami-\E\left[\rnlami\right]\right|^{2+\alpha}\right]=
m_n\E\left[\left|\rnlamone-\E\left[\rnlamone\right]\right|^{2+\alpha}\right]
\end{equation}
and
\begin{equation}
\label{Bnvar}
C_n^2=\sum_{i=1}^{m_n} \var\left(\rnlami\right)=
m_n\var\left(\rnlamone\right).
\end{equation}
Then, using \citet[Corollary 1.9.3]{Serf80},
\begin{equation}
\label{RnlamconvD}
\frac{\rnlam -\E\left[\rnlam\right]}{\sqrt{m_n \var\left(\rnlamone\right)}} \convD N(0,1) \qquad \mbox{as }n \to \infty
\end{equation}
if $B_n=o(C_n^{2+\alpha})$ as $n \to \infty$. We now establish this by showing that, as $n\to\infty$,  $m_n^{-1} B_n $ is bounded  and $m_n^{-1} C_n^{2}= \var(\rnlamone)$ converges to a strictly positive limit.

Firstly,
\begin{align*}
\frac{1}{m_n}B_n & = \E\left[\left|\rnlamone-\E\left[\rnlamone\right]\right|^{2+\alpha}\right]\\
& = \E\left[\left|\sum_{j=1}^3 \sum_{l=1}^{k_j}\lambda_{jl}\left(R_{1j}^{(n)}(t_{jl})-\E\left[R_{1j}^{(n)}(t_{jl})\right]\right)\right|^{2+\alpha}\right]\\
& \le \left[\sum_{j=1}^3 \sum_{l=1}^{k_j}  | \lambda_{jl} | \left\{\left(\E\left[\left(R_{1j}^{(n)}(t_{jl})\right)^{2+\alpha}\right]\right)^{\frac{1}{2+\alpha}}
+\E\left[R_{1j}^{(n)}(t_{jl})\right]\right\}\right]^{2+\alpha},
\end{align*}
by Minkowski's inequality.   Also, $0 \le R_{1j}^{(n)}(t_{jl}) \le 1$ for $j=1,2$ and
\begin{equation*}
%\label{R13bound}
0 \le R_{13}^{(n)}(t_{3l})\le \frac{1}{n}(I_1+I_2+\cdots+I_n),
\end{equation*}
so $0 \le \E\left[R_{13}^{(n)}(t_{3l})\right]\le \mu_I$  and a further application of Minkowski's inequality yields
\[
0 \le \E\left[\left(R_{13}^{(n)}(t_{jl})\right)^{2+\alpha}\right]\le
\E\left[I^{2+\alpha}\right].\]
Thus, recalling~\eqref{Andefn},
\begin{equation}
\label{Anbounds}
0 \le B_n \le m_n\left[\lambda_* k_*\left\{4+\mu_I+\left(\E\left[I^{2+\alpha}\right]\right)^{\frac{1}{2+\alpha}}\right\}\right]^{2+\alpha},
\end{equation}
where $k_*=\max(k_1,k_2,k_3)$ and $\lambda_*=\max(|\lambda_{jl}|:j=1,2,3; l=1,2,\cdots,k_j)$.

Turning to $C_n$, note that Lemma~\ref{ZAnweak} and the continuous mapping theorem imply that, writing $\bb=(b_1,b_2,b_3)^\top$ and recalling~\eqref{bbDef},
\begin{equation*}
\rnlamone \convD \sum_{j=1}^3 \sum_{l=1}^{k_1} b_j \lambda_{jl} \Psi(t_{jl}).
\end{equation*}
Application of the dominated convergence theorem and recalling~\eqref{Bnvar} then yields
\begin{equation*}
m_n^{-1} C_n^2 \to \var\left(\sum_{j=1}^3 \sum_{l=1}^{k_1} b_j \lambda_{jl} \Psi(t_{jl})\right) \in(0,\infty) \qquad \mbox{as }n \to \infty,
\end{equation*}
which, together with~\eqref{Anbounds}, shows that $B_n=o(C_n^{2+\alpha})$ as $n \to \infty$, proving~\eqref{RnlamconvD}. Hence, by the  Cram\'{e}r-Wold device, the
finite-dimensional distributions of $\frac{1}{\sqrt{m_n}}\left(\bRbnT-m_n \brnT\right)$ converge to those of $\bW_T$, since~\eqref{RnlamconvD} holds for any choice of
$k_1,k_2,k_3$ and $(\lambda_{jl},t_{jl})$ $(j=1,2,3;l=1,2,\cdots,k_j)$.

\subsection{Tightness}
We show that the sequence $\left(\bRbnT-m_n \brnT\right)/\sqrt{m_n}$ $(n=1,2,\cdots)$ is asymptotically tight.
For $T>0$ and $i=1,2,3$, let $\rbniT=\{\rbni(t):0 \le t \le T\}$ and $r^{(n)}_{iT}=\{r^{(n)}_{i}(t): 0 \le t \le T\}$.  Then, by \citet[Lemma 1.4.3]{Vaar96} it is sufficient to show that, for $i=1,2,3$, the random processes
$\left(\rbniT-m_n r^{(n)}_{iT}\right) /\sqrt{m_n}$ $(n=1,2,\cdots)$ are asymptotically tight, which we do using Theorem 2.11.9 in that book (the bracketing central limit theorem).  We give the proof for $i=3$ (i.e.\ the severity component); the proofs for $i=1,2$ are similar but simpler. We first give some more notation, then use the above-mentioned Theorem 2.11.9 to state sufficient conditions for $\left(\rbnthreeT-m_n r^{(n)}_{3T} \right)/\sqrt{m_n}$ $(n=1,2,\cdots)$ to be asymptotically tight and finally prove that those conditions are satisfied.

For $T>0$ and $f:[0,\infty) \to \mathbb{R}$, let $$\|f\|_T=\sup_{0 \le t \le T}|f(t)|.$$
For $i=1,2,\cdots,n$, let $R_{i3}^{(n)}=\{R_{i3}^{(n)}(t): t \ge 0\}$.
For $\epsilon,T>0$ and $n=1,2,\cdots$, define the bracketing number $N_{[\,]}^{(n)}(\epsilon,T)$ as the minimal number of sets $N_{\epsilon}$ in a partition
$[0,T]=\cup_{j=1}^{N_{\epsilon}} \mathcal{T}^{(n)}_{\epsilon j}$ of $[0,T]$ into sets $\mathcal{T}^{(n)}_{\epsilon j}$ such that, for every
partitioning set $\mathcal{T}^{(n)}_{\epsilon j}$,
\begin{equation}
\label{partcrit}
\sum_{i=1}^{m_n}\E\left[\max_{t,s \in \mathcal{T}^{(n)}_{\epsilon j}}\left|\frac{1}{\sqrt{m_n}}\left(R_{i3}^{(n)}(t)-R_{i3}^{(n)}(s)\right)\right|^2\right] \le \epsilon^2.
\end{equation}
The inequality~\eqref{partcrit} is expressed in that form to make the connection with \citet[Section~2.11.2]{Vaar96} clearer; but we note that, since $R_{13}^{(n)}, R_{23}^{(n)}, \cdots , R_{m_n3}^{(n)}$ are identically distributed,
\begin{equation*}
\sum_{i=1}^{m_n}\E\left[\max_{t,s\in \mathcal{T}^{(n)}_{\epsilon j}}\left|\frac{1}{\sqrt{m_n}}\left(R_{i3}^{(n)}(t)-R_{i3}^{(n)}(s)\right)\right|^2\right]
=
\E\left[ \max_{t,s\in \mathcal{T}^{(n)}_{\epsilon j}} \left(R_{13}^{(n)}(t)-R_{13}^{(n)}(s)\right)^2 \right].
\end{equation*}

By \citet[Theorem 2.11.9]{Vaar96}, the sequence $\left(\rbnthreeT-m_n r^{(n)}_{3T}\right) /\sqrt{m_n}$ $(n=1,2,\cdots)$ is asymptotically tight if

\begin{description}
\item(i)\ \ for any $d>0$,
\begin{equation*}
\label{vdVWcond1}
\sum_{i=1}^{m_n} \E \left[\left\|\frac{1}{\sqrt{m_n}}R_{i3}^{(n)}\right\|_T 1_{\left\{\left\|\frac{1}{\sqrt{m_n}}R_{i3}^{(n)}\right\|_T>d \right\}}\right]\to 0 \qquad\mbox{as } n \to \infty;
\end{equation*}
\item(ii)\ \ for any sequence $(\delta_n)$ of real numbers satisfying $\delta_n \downarrow 0$ as $n \to \infty$,
\begin{equation*}
\label{vdVWcond2}
\sup_{s,t \in [0,T]:|s-t|<\delta_n}\sum_{i=1}^{m_n} \E\left[\left\{\frac{1}{\sqrt{m_n}}\left(R_{i3}^{(n)}(t)-R_{i3}^{(n)}(s)\right)\right\}^2\right] \to 0
\end{equation*}
as $n\to\infty$; and
\item(iii)\ \ for any sequence $(\delta_n)$ as above,
\begin{equation*}
\label{vdVWcond3}
\int_0^{\delta_n}\sqrt{\log N_{[\,]}^{(n)}(\epsilon,T)}  \,{\rm d}\epsilon \to 0 \qquad \mbox{as } n \to \infty.
\end{equation*}
\end{description}

Considering condition (i),
recall that $$R_{i3}^{(n)}(t) \stlt n^{-1}\sum_{j=1}^n I_j \qquad (i=1,2,\cdots,n; 0 \le t \le T)$$ and let $X=
\frac{1}{\sqrt{m_n}}\frac{1}{n}\sum_{j=1}^n I_j$.  Then, for $d>0$,
\begin{equation}
\label{Cond1eq1}
\sum_{i=1}^{m_n} \E \left[\left\|\frac{1}{\sqrt{m_n}}R_{i3}^{(n)}\right\|_T 1_{\left\{\left\|\frac{1}{\sqrt{m_n}}R_{i3}^{(n)}\right\|_T>d \right\}}\right]
\le
m_n \E\left[X1_{\{X > d\}}\right].
\end{equation}
Recall that there exists $\alpha>0$ such that $\E[I^{2+\alpha}]<\infty$.  Using the same method as a proof of Markov's inequality,
%\annote{}{Wording and display changed as per PT/FB suggestion. Just say `It then follows that' or `We than have' ?}
\begin{align}
\E\left[X1_{\{X > d\}}\right]   & =  d^{-(1+\alpha)}\ \E\left[X\left(1_{\{X > d\}} d^{(1+\alpha)}\right)\right]
 \nonumber \\
&  \le d^{-(1+\alpha)}\E\left[X^{2+\alpha}\right] = d^{-(1+\alpha)}\frac{1}{m_n^{1+\frac{\alpha}{2}}}\frac{1}{n^{2+\alpha}}\E\left[\left(\sum_{j=1}^n I_j\right)^{2+\alpha}\right]. \label{Cond1eq2}
\end{align}
 %By Markov's inequality,
 %\begin{align}
%\E\left[X1_{\{X > d\}}\right] & \le d^{-(1+\alpha)}\E\left[X^{2+\alpha}\right]\nonumber \\
%&= d^{-(1+\alpha)}\frac{1}{m_n^{1+\frac{\alpha}{2}}}\frac{1}{n^{2+\alpha}}\E\left[\left(\sum_{j=1}^n I_j\right)^{2+\alpha}\right]. \label{Cond1eq2}
%\end{align}
Further, by Minkowski's inequality,
\begin{equation*}
\E\left[\left(\sum_{j=1}^n I_j\right)^{2+\alpha}\right]\le n^{2+\alpha} \E\left[I^{2+\alpha}\right],
\end{equation*}
which, together with~\eqref{Cond1eq1} and~\eqref{Cond1eq2}, yields
\begin{equation*}
\sum_{i=1}^{m_n} \E \left[\left\|\frac{1}{\sqrt{m_n}}R_{i3}^{(n)}\right\|_T 1_{\left\{\left\|\frac{1}{\sqrt{m_n}}R_{i3}^{(n)}\right\|_T>d \right\}}\right]
\le m_n^{-\frac{\alpha}{2}}d^{-(1+\alpha)}\E\left[I^{2+\alpha}\right]\to 0
\end{equation*}
as $n\to\infty$, so condition (i) is satisfied.

Turning to condition (ii), recall the Poisson process  $\eta^{(n)}$  used in the construction of $A^{(n)}(\cdot)$ and for any set $F \subset [0,\infty)$ let
$\eta^{(n)}(F)$ denote the number of points of $\eta^{(n)}$ in $F$.  Then,
for any $0 \le a < b < \infty$,
\begin{equation*}
\max_{a \le s < t < b}\left|R_{13}^{(n)}(t)-R_{13}^{(n)}(s)\right| \stlt  1_{\{\eta^{(n)}([a,b)) \ge 1\}} \frac{1}{n}\sum_{j=1}^n I_j.
\end{equation*}
Thus
\begin{align*}
\frac{1}{m_n}\E\left[\max_{a \le s < t < b}\left|R_{13}^{(n)}(t)-\right.\right.\left.\left.R_{13}^{(n)}(s)\right|^2\right] &\le
\left(1-\e^{-\bgn n^2m_n(b-a)}\right) \frac{1}{m_n}\left(\mu_I^2+\frac{\sigma_I^2}{n}\right)\\
%&\le\e^{-\bgn n^2m_n a}\left(1-\e^{-\bgn n^2m_n (b-a)}\right)\frac{1}{m_n}\left(\mu_I^2+\frac{\sigma_I^2}{n}\right)\\
&\le \bgn n^2 (b-a) (\mu_I^2+\sigma_I^2),
\end{align*}
so, since $R_{13}^{(n)}, R_{23}^{(n)}, \cdots , R_{m_n3}^{(n)}$ are identically distributed,
%\annote{}{changed wording as PT suggested (26/5) (from just `for all sufficiently large $n$')}
there exists  $n_0 \in \mathbb{N}$ independent of $a$ and $b$ such that for all $n>n_0$,
\begin{align}
\label{moment2bound}
\sum_{i=1}^{m_n}\E\left[\max_{a \le s < t < b}\left|\frac{1}{\sqrt{m_n}}\left(R_{i3}^{(n)}(t)-R_{i3}^{(n)}(s)\right)\right|^2\right]
&\le \bgn n^2 m_n (b-a) (\mu_I^2+\sigma_I^2)\nonumber\\
&\le 2\lg (b-a) (\mu_I^2+\sigma_I^2),
\end{align}
%for all sufficiently large $n$,
since $ \bgn n^2 m_n \to \lg$ as $n \to \infty$.  It follows that condition (ii) is satisfied.

Turning finally to condition (iii), it follows from~\eqref{moment2bound} that~\eqref{partcrit} is satisfied for all sufficiently large $n$ if $[0,T]$ is partitioned into intervals of length $\epsilon^2/[2 \lg(\mu_I^2 + \sigma_I^2)]$ (the final interval may have a shorter length),
so $N_{[\,]}^{(n)}(\epsilon,T) \le c/\epsilon^2$ for all sufficiently large $n$, where $c=4T\lg (\mu_I^2 + \sigma_I^2)$.  Thus
\begin{equation*}
\limsup_{n \to \infty} \int_0^{\delta_n}\sqrt{\log N_{[\,]}^{(n)}(\epsilon,T)}  \,{\rm d}\epsilon
\le \limsup_{n \to \infty} J_n,
\end{equation*}
where $J_n= \int_0^{\delta_n} \sqrt{\log(c/\epsilon^2)} \,{\rm d}\epsilon$.  Substituting $u=\log(c/\epsilon^2)$ yields after a little algebra that $J_n \to 0$ as $n \to \infty$, so condition (iii) is satisfied.

\section{Proof of Lemma~\lowercase{\ref{pn_znhat_convlemma}}}
\label{appB}
In this appendix we first prove~\eqref{pnconv} and~\eqref{pnconvlog} in Section~\ref{appB1}. Then in Section~\ref{sec:suscSets} we introduce the notion of susceptibility sets and some results about them which are subsequently used in Section~\ref{sec:B3} to prove~\eqref{znhatconv}. Recall that $R_0= \mu_I \lambda_W >1$.

\subsection{Proof of~\eqref{pnconv} and~\eqref{pnconvlog}}
\label{appB1}
Recall that the first assertion we set out to prove here is that, under (C1)--(C3), for any $\gamma\in(0,\delta/4)$ and $T>0$ we have $$\lim_{n \to \infty} \sqrt{m_n} \sup_{0 \le t \le T}\left|\png (t)-p(t)\right|=0,$$ where $\png (t) = \P\left(\Zn (t) \ge \ngam \right)$ and $p(t)=1-\e^{-\lambda_G(1-\pi_W)t}$.

We first introduce some more notation.
For $\lg$, $\lw>0$, let $\BB(\lg, \lw I)$ denote a Galton-Watson branching process having a random number of ancestors distributed as $\Po(\lg)$ and offspring distribution $\Po(\lw I)$.
Let $B(\lg, \lw I)$ denote the total size of $\BB(\lg, \lw I)$, including the ancestors.  For $n=1,2,\cdots$, let $\chin_1, \chin_2,\cdots$ be i.i.d.\ random variables that are uniformly distributed on the integers $1,2,\cdots,n$.   Then, cf.~\cite{Ball95} and Section~\ref{early}, for any fixed $t \ge 0$, a realisation of $\Zn(t)$ can be obtained from $\BB( n^2 m_n \bgn t, n \bwn I)$ as follows.

Let $\BBn(t)=\BB( n^2 m_n \bgn t, n \bwn  I)$, $\Bn(t)=B(n^2 m_n \bgn t, n \bwn I)$ and recall from Section \ref{embedding} that $\Zn(t)$ has the same distribution as the size of the epidemic $\En(t)$ defined in Definition~\ref{defEpsnt}.
Label the individuals in the population upon which $\En(t)$ is defined $1,2,\cdots,n$.  Births in $\BBn(t)$ correspond to infectious contacts in $\En(t)$. For $k=1,2,\cdots$, the individual contacted at the $k$-th birth in $\BBn(t)$ (including the initial infectives) has label given by $\chin_k$.  If that individual is still susceptible in $\En(t)$ then they become infected and make infectious contacts according to the offspring of the corresponding individual in $\BBn(t)$.  If that individual is no longer susceptible in $\En(t)$ then they and all of the progeny of the corresponding individual in $\BBn(t)$ are ignored in the construction of $\En(t)$.

Let $\Mn=\min\{k\ge 2:\chin_k \in \{\chin_1,\chin_2,\cdots,\chin_{k-1}\}\}$ be the position of the first duplicate in $\{\chin_1,\chin_2,\cdots\}$ and fix $\gamma\in(0,\delta/4)$.  Now,
\begin{equation*}
\P\left(\Zn(t) <\ngam \right)=\P\left(\Zn(t) <\ngam, \Mn <\ngam\right)+ \P\left(\Zn(t) <\ngam, \Mn \ge \ngam\right)
\end{equation*}
and
\begin{equation*}
\P\left(\Bn(t) <\ngam \right)=\P\left(\Bn(t) <\ngam, \Mn <\ngam\right)+ \P\left(\Bn(t) <\ngam, \Mn \ge \ngam\right).
\end{equation*}
By the coupling, $\P\left(\Zn(t) <\ngam, \Mn \ge \ngam\right)=\P\left(\Bn(t) <\ngam, \Mn \ge \ngam\right)$, so
\begin{align}
\label{birthday}
 & \left| \P\left(\Zn(t) \ge \ngam \right) - \P\left(\Bn(t) \ge \ngam \right) \right| \nonumber\\
&\quad=\left| \P\left(\Zn(t) <\ngam \right)-\P\left(\Bn(t) <\ngam \right) \right| \nonumber\\
&\quad=\left| \P\left(\Zn(t) <\ngam, \Mn <\ngam\right)-\P\left(\Bn(t) <\ngam, \Mn <\ngam\right)   \right| \nonumber\\
&\quad\le\P\left(\Mn <\ngam\right)
=1-\prod_{k=1}^{\ceil{\ngam}-1} \left(1-\frac{k}{n}\right) \le \sum_{k=1}^{\ceil{\ngam}-1} \frac{k}{n} =  \frac{\ceil{\ngam}(\ceil{\ngam}-1)}{2n} . %\le \frac{\ngam(\ngam+1)}{2n}.
\end{align}
This bound is independent of $t$, so
\begin{equation}
\label{limit0}
\lim_{n \to \infty} \sqrt{m_n}\sup_{0 \le t \le T}\left|\png(t)-\P\left(\Bn(t) \ge \ngam\right)\right|=0,
\end{equation}
since $m_n n^{4\gamma-2} \to 0$ as $n \to \infty$, by condition (C1) as $\gamma < \frac{\delta}{4}$.
%\annote{}{FB suggested addition of ``by condition (C1) as ...''}

Recall from~\eqref{infectionrates} that $\bwn n \to \lw$ and $\bgn n^2 m_n \to \lg$ as $n\to\infty$.
For $t \ge 0$, let $\BB(t)=\BB(\lg t, \lw I)$ and $B(t)=B(\lg t, \lw I)$, and note that $p(t)=\P\left(B(t)=\infty\right)$.  Then
\begin{align}
\left|\P\left(\Bn(t) \ge \ngam\right)-p(t)\right| \le & \left|\P\left(\Bn(t) \ge \ngam\right)-\P\left(\Bn(t) = \infty\right)\right|\nonumber\\
&\qquad+\left|\P\left(\Bn(t) = \infty\right)-p(t)\right|.
%& \left|\P\left(\Bn(t) \ge \ngam\right)-p(t)\right| \nonumber \\
% & \quad \leq \left|\P\left(\Bn(t) \ge \ngam\right)-\P\left(\Bn(t) = \infty\right)\right| +\left|\P\left(\Bn(t) = \infty\right)-p(t)\right|.
\label{bound2}
 \end{align}
We consider the two terms on the right-hand side of~\eqref{bound2} separately, starting with the second one.

Recall that $\BB(\lw I)$ denotes the Galton-Watson branching process with one ancestor and offspring distribution $\Po(\lw I)$, so $\pi_W$ is the extinction probability of $\BB(\lw I)$.  Let $\pi_n$ denote the extinction probability of $\BB(n\bwn I)$.  Then $\P\left(\Bn(t) = \infty\right)=1-\e^{-\bgn n^2 m_n(1-\pi_n)t}$, so
\begin{align}
\label{bound3}
\Big|\P\left(\Bn(t) = \infty\right)&-p(t)\Big| = \left|\e^{-\bgn n^2 m_n(1-\pi_n)t}-\e^{-\lg(1-\pi_W)t}\right| \nonumber \\
&=\e^{-\bgn n^2 m_n(1-\pi_n)t} \left|1- \e^{\left[\bgn n^2 m_n(1-\pi_n)-\lg(1-\pi_W)\right]t}\right| \nonumber \\
&\le \left|1- \e^{\left[\bgn n^2 m_n(1-\pi_n)-\lg(1-\pi_W)\right]t}\right| \nonumber \\
&\le\left|\bgn n^2 m_n(1-\pi_n)-\lg(1-\pi_W)\right| \nonumber \\
 & \qquad \times t\max\left\{1,\e^{\left[\bgn n^2 m_n(1-\pi_n)-\lg(1-\pi_W)\right]t}\right\},
\end{align}
since $\left|1-\e^x\right| \le |x|\max\left(1,\e^x\right)$ for all $x \in \mathbb{R}$.
Now $n\bwn \to \lw$ as $n \to \infty$, so $\pi_n \to \pi_W$ as $n \to \infty$.  Thus,
\begin{equation*}
\lim_{n \to \infty}\sup_{0 \le t \le T} t \, \max \left\{1,\e^{\left[\bgn n^2 m_n(1-\pi_n)-\lg(1-\pi_W)\right]t}\right\}=T,
\end{equation*}
since $n^2 m_n \bgn \to \lg$ as $n \to \infty$.
Further, since
\begin{align*}
&\left|\bgn n^2 m_n(1-\pi_n)-\lg(1-\pi_W)\right|\\
& \qquad \qquad \le \left|\bgn n^2 m_n(1-\pi_n)-\lg(1-\pi_n)\right|
+\left|\lg(1-\pi_n)-\lg(1-\pi_W)\right|\\
& \qquad \qquad \le\left|\bgn n^2 m_n-\lg\right|+\lg\left|\pi_n-\pi_W\right|,
\end{align*}
it follows from~\eqref{bound3} and condition (C2) that, for any $T>0$,
\begin{equation}
\label{limit1}
\lim_{n \to \infty}\sup_{0 \le t \le T}\sqrt{m_n}\left|\P\left(\Bn(t) = \infty\right)-p(t)\right|=0,
\end{equation}
provided $\sqrt{m_n}(\pi_n-\pi_W)\to 0$ as $n \to \infty$, which we show now.

Let $\pi(\lambda)$ denote the extinction probability of
$\BB(\lambda I)$, so $\pi_n=\pi(n \bwn)$ and $\pi_W=\pi(\lambda_W)$.  Further, $\pi(\lambda)$ is given by the smallest solution of $\phi_I(\lambda(1-s))=s$, where $\phi_I$ is the Laplace transform of $I$.  Suppose that $\lambda\mu_I>1$, so
$\pi(\lambda)\in (0,1)$.  Then implicit differentiation shows that the derivative of $\pi$ at $\lambda$ is given by
\begin{equation*}
\pi'(\lambda)=\frac{(1-\pi(\lambda))\phi_I'(\lambda(1-\pi(\lambda)))}{\lambda\phi_I'(\lambda(1-\pi(\lambda)))+1}.
\end{equation*}
Further, $-\lambda\phi_I'(\lambda(1-\pi(\lambda)))<1$ as $-\lambda\phi_I'(\lambda(1-s))$ is increasing with $s\in[0,1]$ and $\pi(\lambda)$ is the unique solution
of $\phi_I(\lambda(1-s))=s$ in $(0,1)$.  Now $n\bwn \to \lw$ as $n \to \infty$, so $n\bwn \mu_I>1$ for all sufficiently large $n$ as $R_0=\lambda_W \mu_I>1$. Also, $\phi_I'$ is continuous on $(0,\infty)$, so
application of the mean value theorem and using condition (C3) shows that $\sqrt{m_n}(\pi_n-\pi_W)\to 0$ as $n \to \infty$, thus proving~\eqref{limit1}.

Turning to the first term on the right-hand side of~\eqref{bound2}, note that
\begin{equation}
\label{equation1}
\P\left(\Bn(t) \ge \ngam\right)-\P\left(\Bn(t) = \infty\right)=\P\left(\Bn(t) \in [\ngam,\infty) \right)
\end{equation}
and
\begin{equation}
\label{equation1a}
\P\left(\Bn(t) \in [\ngam,\infty) \right)
=\P\left(\Bn(t)<\infty \right) \P\left(\Bn(t) \in [\ngam,\infty) \mid \Bn(t)<\infty \right).
\end{equation}
Let $\Yn_0(t)$ be the number of ancestors in $\BBn(t)$, so $\Yn_0(t) \sim \Po(n^2 m_n \bgn t)$ and $\left(\Yn_0(t)\mid \Bn(t)<\infty\right) \sim \Po(n^2 m_n \bgn \pi_n t)$.  To show the latter, let $Y_0$ be the number of ancestors in $\BB(\lg, \lw I)$.  Then, $\P\left(B(\lg, \lw I)<\infty \mid Y_0=k\right)= \pi_W^k$ $(k=0,1,\cdots)$.  Hence, for $k=0,1,\cdots$,
\[
\P\left(Y_0=k \mid B(\lg, \lw I)<\infty \right) \propto \P(Y_0=k) \P\left(B(\lg, \lw I)<\infty \mid Y_0=k \right)=\frac{\lg^k \e^{-\lg}\pi_W^k}{k!},
\]
whence $\left(Y_0 \mid B(\lg, \lw I)<\infty\right) \sim \Po(\lg\pi_W)$.
Let $\BBn=\BB(n\bwn I)$ and $\Bn$ be the total size of $\BBn$, including the ancestor.  Thus $\BBn$ describes the process evolving from a typical ancestor in $\BBn(t)$.

Recall that $R_0=\lw \mu_I>1$ so, as indicated above, $\BBn$ is supercritical for all sufficiently large $n$, say $n \ge n_0$.  Let $f_n(s)=\phi_I(n\bwn(1-s))$ be the offspring PGF of $\BBn$.  Then, for $n \ge n_0$, conditional upon extinction $\BBn$ is distributed as a subcritical Galton-Watson process, $\BBnhat$ say, having offspring PGF $\hat{f}_n(s)=f_n(s\pi_n)/\pi_n$; see e.g.~\cite{Daly79}.  Let $\Bnhat$ denote the total size of
$\BBnhat$ and $\hat{h}_n(s)$ be the PGF of $\Bnhat$.  Then $\hat{h}_n(s)$ is given by the smallest solution of $\hat{h}_n(s)=s\hat{f}_n(\hat{h}_n(s))$ in
$[0,\infty)$ (see \citet[p.~39]{Jage75}.  Now
\[
\lim_{n \to \infty} f_n(s) = f(s)=\phi_I(\lw(1-s)) \quad (0 \le s \le 1),
\]
whence
\[
\lim_{n \to \infty} \hat{f}_n(s) = \hat{f}(s)=f(s \pi_W)/\pi_W \quad (0 \le s \le \pi_W^{-1}).
\]
Recall that $s_0=\max\{s>1:h=s\hat{f}(h)\mbox{ for some } h \in (0,\infty)\}$; see just before the statement of Theorem~\ref{multifinalmain1}
in Section~\ref{sec:divergingm}.
For any $s \in [0, s_0)$, $\hat{h}_n(s) \to \hat{h}(s)$ as $n \to \infty$,
where $\hat{h}(s)$ is given by the smallest solution of $\hat{h}(s)=s\hat{f}(\hat{h}(s))$ in
$[0,\infty)$. Note that $\hat{h}$ is the PGF of the total size of $\BB(\lw I)$ conditioned on extinction.

Recalling that $\left(\Yn_0(t)\mid \Bn(t)<\infty\right) \sim \Po(n^2 m_n \bgn \pi_n t)$,
\begin{equation*}
\E\left[s^{\Bn(t)}\mid \Bn(t)<\infty \right] = \e^{-n^2 m_n \bgn \pi_n t(1-\hat{h}_n(s))},
\end{equation*}
so Markov's inequality yields
\begin{equation*}
\P \left( \Bn(t) \in [\ngam,\infty) \mid \Bn(t) < \infty\right) \le s^{-\ngam} \, \e^{-n^2 m_n \bgn \pi_n t (1-\hat{h}_n(s))}
\end{equation*}
%\annote{}{no longer `for all suitable $s$'. DS: Is in fact true for all $s\geq0$ as is just an application of Markov's inequality? That we are interested in $s>1$ is what makes the next argument a bit more involved than if $0<s<1$; but at this point we are ok for all $s$ I am pretty sure.}
for all $s>1$. Now choose $\hat{s}_0 \in (1, s_0)$ and let $\hat{\theta}_0=\log \hat{s}_0>0$.
Noting that $n^2 m_n \bgn \to \lambda_G$, $\pi_n\to\pi_W$ and $\hat{h}_n(\hat{s}_0) \to \hat{h}(\hat{s}_0)$ as $n \to \infty$, we have that for any $\epsilon>0$ and all $t\in[0,T]$
\begin{equation}
\label{Markovineq}
\P \left( \Bn(t) \in [\ngam,\infty) \mid \Bn(t) < \infty\right) \le \e^{-\hat{\theta}_0 \ngam} \, \e^{-\lambda_G \pi_W t (1-\hat{h}(\hat{s}_0))+\epsilon}
\end{equation}
for all sufficiently large $n$.
Hence,
\[
\lim_{n \to \infty} \sqrt{m_n} \sup_{0 \le t \le T} \left|\P\left(\Bn(t)\in[\ngam,\infty) \mid \Bn(t) < \infty\right)\right|=0,
\]
since $\hat{\theta}_0>0$ and $\gamma>0$.  Thus, using~\eqref{equation1} and~\eqref{equation1a},
\begin{equation}
\label{limit3}
\lim_{n \to \infty} \sqrt{m_n} \sup_{0 \le t \le T} \left|\P\left(\Bn(t) \ge \ngam\right)-\P\left(\Bn(t) = \infty\right)\right|=0
\end{equation}
for every $T>0$.
Equation~\eqref{pnconv} now follows using~\eqref{limit0}, \eqref{bound2}, \eqref{limit1} and~\eqref{limit3}.

The proof of~\eqref{pnconvlog} parallels that of~\eqref{pnconv}.  Arguing as in the derivation of~\eqref{limit0} shows that
\begin{equation}
\label{limit0a}
\lim_{n \to \infty} \sqrt{m_n}\sup_{0 \le t \le T}\left|\P\left(\Zn(t) \ge \log n\right)-\P\left(\Bn(t) \ge \log n\right)\right|=0.
\end{equation}
Now take $\theta'$ as in condition (C5) and let $\hat{s}_0=\exp(\theta'/2)$, so we have $\hat{\theta}_0 = \log \hat{s}_0 = \frac{\theta'}{2}$.
A similar application of Markov's inequality to~\eqref{Markovineq} yields that, for all $t \in [0,T]$ and any $\epsilon>0$,
\begin{equation*}
\P\left(\Bn(t)\in[\log n,\infty) \mid \Bn(t) < \infty\right) \le n^{-\hat{\theta}_0} \e^{-\lambda_G \pi t(1-\hat{h}(\hat{s}_0))+\epsilon}
\end{equation*}
for sufficiently large $n$. Now $\lim_{n \to \infty} \sqrt{m_n} n^{-\hat{\theta}_0}=0$, since $\lim_{n \to \infty} m_n n^{-\theta'}=0$, so arguing as in the derivation of~\eqref{limit3} yields
\begin{equation*}
\lim_{n \to \infty} \sqrt{m_n} \sup_{0 \le t \le T} \left|\P\left(\Bn(t) \ge \log n\right)-\P\left(\Bn(t) = \infty\right)\right|=0
\end{equation*}
for every $T>0$, which together with~\eqref{limit1} and~\eqref{limit0a} implies~\eqref{pnconvlog}.

\subsection{Susceptibility sets}
\label{sec:suscSets}
Before proving \eqref{znhatconv}, we describe the concept of a susceptibility set (see e.g.~\cite{Ball02,Ball09}) and derive some associated bounds that are used in the proof.  The notation is local to this description.  Consider the standard single-population SIR epidemic model of Section~\ref{sec:singleComm} but suppose that initially there are $a$ infectives and $N$ susceptibles, the infectious period is distributed according to the random variable $I$ and the individual-to-individual infection rate is $\beta$.  Label the initial infectives, $-(a-1),-(a-2),\cdots,0$ and the initial susceptibles $1,2,\cdots,N$.  Let $\mathcal{G}$ denote the random directed graph on $\mathcal{N}=\{-(a-1), -(a-2),\cdots,N\}$ in which for any distinct $i,j \in \mathcal{N}$ there is a directed edge from $i$ to $j$ if and only if $i$, if infected, makes infectious contact with $j$.  More specifically, the graph $\mathcal{G}$ is constructed by for each $i \in \mathcal{N}$ first sampling its
infectious period, $I_i$ say, and then for each $j \in \mathcal{N} \setminus \{i\}$ drawing a directed edge from $i$ to $j$ independently with probability $1-\e^{-\beta I_i}$.

For distinct $i, j \in \mathcal{N}$, write $i \rightsquigarrow j$ if and only if there is a chain of directed edges from $i$ to $j$ in $\mathcal{G}$.
Let $\mathcal{I}_0=\{-(a-1), -(a-2),\cdots,0\}$ denote the set of initial infectives.  For $i=1,2,\cdots, N$, define the susceptiblity set of individual $i$ by
$\mathcal{S}_i=\{j \in \mathcal{N} \setminus \{i\}:j \rightsquigarrow i\}$ and let $S_i=|\mathcal{S}_i|$ denote the cardinality of $\mathcal{S}_i$.  Note that
susceptible $i$ is infected by the epidemic if and only if $\mathcal{S}_i \cap \mathcal{I}_0 \neq \emptyset$ and, by exchangeability,
\begin{equation*}
\P\left(i \mbox{ avoids infection from epidemic} \mid  S_i=k  \right)=\theta_{N,a}(k)\qquad (k=0,1,\cdots,N),
\end{equation*}
where $\theta_{N,a}(0)=1$, $\theta_{N,a}(k)=0$ if $k=N,N+1,\cdots,N+a-1$, and
\begin{equation}
\label{thetaNak}
\theta_{N,a}(k)=\prod _{j=1}^k \left( \frac{N-j}{N+a-j} \right) \quad (k=1,2,\cdots,N-1).
\end{equation}
It follows that
\begin{equation*}
\P\left(i \mbox{ avoids infection from epidemic}\right)=\sum_{k=0}^{N-1} \P(S_i=k) \theta_{N,a}(k).
\end{equation*}
Note that $\theta_{N,a}(k)$ decreases with $k$ so, for $k=1,2,\cdots,N-1$,
\begin{align}
\theta_{N,a}(k)\P(S_i\le k) & \le \P\left(i \mbox{ avoids infection from epidemic}\right) \nonumber \\
 & \le \P(S_i\le k)+\theta_{N,a}(k+1). \label{pavoidbounds}
\end{align}

\subsection{Proof of~\eqref{znhatconv}}
\label{sec:B3}
Recall that we are setting out to prove that, under (C1)--(C4), for any $\gamma\in(0,\delta/4)$ and $T>0$ we have $$\lim_{n \to \infty} \sqrt{m_n} \sup_{0 \le t \le T}\left|\hatzng (t)- \zinf \right|=0,$$ where $\hatzng (t) = \E\left[\barzn(t) \mid \Zn (t) \ge \ngam\right]$.
The proof is structured as follows. First, in Section~\ref{sec:qingamt}, we reduce the problem to finding bounds $\qLgam$ and $\qUgam$ which satisfy respectively \eqref{lowerlim} and \eqref{upperlim} below, for a key quantity $\qingamt$ that can be loosely interpreted as the probability that an individual who is susceptible early in a large outbreak remains susceptible at the end of the outbreak. Then in each of sections~\ref{sec:qL} and~\ref{sec:qU} we bound first the number of infectives early in a large outbreak and second the size of the susceptibility set of an individual who is susceptible after the early part of a large outbreak, and thus bound $\qingamt$.

\subsubsection{The quantity $\qingamt$}
\label{sec:qingamt}
Fix $\gamma \in (0, \frac{\delta}{4})$ and $T>0$, and note that for any $t \in [0,T]$
\begin{align*}
\hatzng (t) & = \frac{1}{n}\left[\ceil{\ngam}+\left(n-\ceil{\ngam}\right)\left(1-\qingamt\right)\right]\\
& = 1-\frac{\left(n-\ceil{\ngam}\right)}{n}\qingamt,
\end{align*}
where $\qingamt$ is the conditional probability, given that the epidemic $\En(t)$ has infected at least $\ceil{\ngam}$ individuals, that any given individual who is susceptible when the cumulative number of infectives in $\En(t)$ first reaches $\ceil{\ngam}$ is still susceptible at the end of the epidemic. For $\lambda>0$, let $\pipo(\lambda)$ denote the extinction probability of the branching process $\BB(\lambda)$; so $\pipo(\lambda) = -\lambda^{-1} W_0(-\lambda \e^{-\lambda})$, where $W_0(\cdot)$ is the principal branch of the Lambert W function.  Note that $\zinf=1-\pipo(\lw \mu_I)$.  Then
\[
\hatzng (t)- \zinf =\pipo(\lw \mu_I)-\qingamt+\frac{\ceil{\ngam}}{n}\qingamt,
\]
so
\begin{equation*}
\sqrt{m_n}\sup_{0 \le t \le T}\left|\hatzng (t)- \zinf \right| \le \sqrt{m_n}\sup_{0 \le t \le T}\left|\pipo(\lw \mu_I)-\qingamt \right|
+\sqrt{m_n}\frac{\ceil{\ngam}}{n}.
\end{equation*}
Now $\sqrt{m_n}\frac{\ceil{\ngam}}{n} \to 0$ as $n \to \infty$, since $\gamma<\frac{\delta}{4}$ and $\delta < 2$, so to prove~\eqref{znhatconv} it is sufficient to obtain lower and upper bounds, $\qLgam$ and $\qUgam$, for $\qingamt$ $(0 \le t \le T)$ that satisfy
\begin{equation}
\label{lowerlim}
\lim_{n \to \infty}\sqrt{m_n}\left|\pipo(\lw \mu_I)-\qLgam\right|=0
\end{equation}
and
\begin{equation}
\label{upperlim}
\lim_{n \to \infty}\sqrt{m_n}\left|\pipo(\lw \mu_I)-\qUgam\right|=0.
\end{equation}

In the following two sections we establish bounds for $\Ynhat(\ngam,t)$, the number of infectives in $\En(t)$ when the cumulative number of infectives reaches $n^\gamma$, and $S^{(n)}_i$, the size of the susceptibility set of an individual who is suscepbible at the time the cumulative number of infections reaches $n^\gamma$, amongst the individuals who are susceptible and infectious at that time. Since $\qingamt$ is an infection \emph{avoidance} probability, upper bounds for $\Ynhat(\ngam,t)$ and $S^{(n)}_i$ imply a lower bound for $\qingamt$, and vice-versa.

\subsubsection{The lower bound $\qLgam$}
\label{sec:qL}

First note that $\ceil{n^\gamma}$ is a trivial upper bound for $\Ynhat(\ngam,t)$.

Suppose that we stop the construction of $\En(t)$ when the cumulative number of infectives first reaches $\ngam$.  (If the epidemic $\En(t)$ infects fewer than $\ngam$ individuals, the construction is repeated independently until an epidemic that infects at least $\ngam$ individuals is obtained.)  Choose an individual, $i$ say, that is still susceptible and consider its susceptibility set, $\SSn_i$  having size $\Sn_i$, among the remaining $n-\ceil{\ngam}-1$ susceptibles and the $\Ynhat(\ngam,t)$ current infectives, so in the notation of Section~\ref{sec:suscSets} used to describe susceptibility sets, $N=n-\ceil{\ngam}$ and $a=\Ynhat(\ngam,t)$.
Each of the other $N-1$ remaining susceptibles in the population, if infected, makes infectious contact with individual $i$ independently and with probability $1-\phi_I(\bgn)$.
It follows that the indicator function of the event that a given individual contacts individual $i$ is stochastically smaller than a $\Po\left(-\kappa_I(\bgn)\right)$ random variable, where $\kappa_I(\theta)=\log \phi_I(\theta)$. Thus the total number of individuals that contact individual $i$ (and hence join  $\SSn_i$) is stochastically smaller than a $\Po\left(-(N-1)\kappa_I(\bgn)\right)$ random variable.  (Note that $\kappa_I(-\theta)$ is the cumulant-generating function of~$I$.)

The susceptibility set $\SSn_i$ can be constructed on a generation basis in the obvious fashion, with individual $i$ comprising generation 0 and for any other individual, $j$ say, in $\SSn_i$, its generation is given by the number of edges in the shortest directed path from $j$ to $i$ in $\SSn_i$.
Note however that, unless the infectious period $I$ is constant, the probability an individual joins the susceptibility set is different for different generations.
Fix two individuals, $j_1$ and $j_2$, let $q_0^{(n)}$ be the unconditional probability that $j_1$ fails to infect $j_2$  and, for $k=1,2,\cdots$, let $q_k^{(n)}$ be the conditional probability that $j_1$ fails to infect $j_2$ given that $j_1$ has failed to infect $k$ other individuals.  Then $q_k^{(n)}=\phi_I\left((k+1)\bwn\right)/\phi_I\left(k\bwn\right)$ $(k=0,1,\cdots)$ and since $\kappa_I$ is convex it is readily seen that $q_0^{(n)} \le q_1^{(n)} \le \cdots$, as one would expect on intuitive grounds.  Thus, as $N-1 \le n$, it follows that $\Sn_i \stlt \Vn$, where $\Vn$ is the total progeny of the Galton-Watson branching process $\BB(\lambda_n)$ and $\lambda_n=-n\kappa_I(\bwn)$. Note that $\Vn$ and $\Sn_i$ do not include the initial individual $i$.

Choose $\gamma_1>0$ such that $2(\gamma+\gamma_1)<\delta$, where $\delta \in (0, 2)$ is as in condition (C1).   Recall that, for any $t \in [0,T]$, the number of current infectives $\Ynhat(\ngam,t)$ is at most
$\ceil{\ngam}$.  Thus, from~\eqref{pavoidbounds}, $\qLgam=\theta_{N,a}(\ceil{\ngamone})\P\left(\Vn \le \ceil{\ngamone}\right)$ with $N=n-\ceil{\ngam}$ and $a=\ceil{\ngam}$ is a uniform lower bound for $\qingamt$. Now
\begin{align}
\Big|\qLgam- \pipo(\lw \mu_I)\Big| & = \Big|\P\left(\Vn \le \ceil{\ngamone}\right)\left(\theta_{N,a}\left(\ceil{\ngamone}\right)-1\right) \nonumber \\
  & \hspace*{3cm} +\P\left(\Vn \le \ceil{\ngamone}\right)-\pipo(\lw \mu_I)\Big| \nonumber \\
& \le \left|\theta_{N,a}(\ceil{\ngamone})-1\right|+\left|\P\left(\Vn \le \ceil{\ngamone}\right)-\pipo(\lw \mu_I)\right|. \label{qlntriangle}
\end{align}

Note that $\kappa_I$ is infinitely differentiable on $(0,\infty)$, so by Taylor's theorem,
\begin{equation}
\label{taylor}
\kappa_I(x)=-x \mu_I +x^2 O(1) \qquad \mbox{as }x \to 0.
\end{equation}
Hence,
\begin{equation*}
\lim_{n \to \infty}\sqrt{m_n}(\lambda_n-\lw \mu_I)=\lim_{n \to \infty}\sqrt{m_n}\left[\left(n\bwn-\lw\right)\mu_I+n\left(\bwn\right)^2 O(1)\right]=0,
\end{equation*}
using conditions (C1) and (C3).
Arguing as in the derivation of~\eqref{limit1} and~\eqref{limit3} in Section~\ref{appB1} now shows that
\begin{equation}
\label{mnvnpi}
\sqrt{m_n}\left|\P\left(\Vn \le \ceil{\ngamone}\right)-\pipo(\lw \mu_I)\right| \to 0 \qquad \mbox{as } n \to \infty.
\end{equation}
Further,
for $k=1,2,\cdots,N-1$,
\begin{equation*}
\theta_{N,a}(k)\ge\left(\frac{N-k}{N+a-k}\right)^k = \left(1-\frac{a}{N+a-k}\right)^k \ge 1-\frac{ka}{N+a-k},
\end{equation*}
so, substituting in $N=n-\ceil{\ngam}$ and $a=\ceil{\ngam}$,
\begin{equation}
\label{mntheta1}
\sqrt{m_n} \, \left|\theta_{N,a}(\ceil{\ngamone})-1\right| \le \sqrt{m_n} \frac{\ceil{\ngamone}\ceil{\ngam}}{n-\ceil{\ngamone}} \to 0 \qquad\mbox{as } n \to \infty,
\end{equation}
as $\gamma+\gamma_1 < \frac{\delta}{2}$ and $\delta<2$.  The limit~\eqref{lowerlim} follows using~\eqref{qlntriangle}, \eqref{mnvnpi} and~\eqref{mntheta1}.

\subsubsection{The upper bound $\qUgam$}
\label{sec:qU}

To find a lower bound for $\Ynhat(\ngam,t)$ (the number of infectives in $\En(t)$ when the cumulative number of infectives first reaches $\ceil{\ngam}$), consider the branching process $\BBn(t)$ introduced in Section~\ref{appB1}.  Its total size $\Bn(t)$ can be constructed in the well-known fashion from a random walk obtained by considering the offspring of individuals sequentially as follows.  Recall that $\Yn_0(t) \sim \Po\left(n^2 m_n \bgn t\right)$ is the number of ancestors in $\BBn(t)$ and let $\Yn_1,\Yn_2,\cdots$ be i.i.d.\ copies of the offspring distribution $\Po\left(n \bwn\right)$.  Let $\Xn_0(t)=\Yn_0(t)$ and, for $k=1,2,\cdots$, let $\Xn_k(t)=\Yn_0(t)+\sum_{i=1}^k \Yn_i$.  Suppressing the explicit dependence on $t$, let $\Wn=\inf\{k \ge 0: \Xn_{k}(t)-k=0\}$, where $\Wn=\infty$ if $\Xn_{k}(t)>k$ for all $k\ge 0$. Then $\Bn(t)=\Xn_{\Wn}(t)$.

Let $\muyn=\E\left[\Yn_1\right]=n\bwn \mu_I$ and for $\theta \in \mathbb{R}$, let
\begin{equation*}
M_n(\theta)=\E\left[\e^{\theta\left(\Yn_1-\muyn\right)}\right]=\e^{-n \bwn \mu_I \theta}\phi_I\left(n \bwn \left(1-\e^{\theta}\right)\right).
\end{equation*}
Recall that $\lim_{n\to\infty} n \bwn = \lw$ and (by condition (C4)) there exists $\theta_0<0$ such that $\phi_I(\theta_0)<\infty$.  Thus $M_n(\theta) \to M(\theta)=\e^{-\lw \mu_I \theta}\phi_I\left(\lw\left(1-\e^{\theta}\right)\right)$, as $n \to \infty$, in an interval containing the origin.  By a standard large-deviation argument (e.g.~\citet[p.~202]{Grim01}), for all $\zeta>0$, there exists $\theta_1>0$ such that $\e^{-\theta \zeta} M(\theta) \in (0,1)$ for all $\theta \in (0,\theta_1)$, so there exist $\theta_2>0$, $c_1 \in (0,1)$ and $n_0 \in \mathbb{N}$ such that $\e^{-\theta_2 \zeta} M_n(\theta_2) \le c_1$ for all $n \ge n_0$.  Hence, by Markov's inequality, for all $n\ge n_0$,
\begin{equation*}
\P\left(\Xn_{k}(t)-k\muyn \ge k \zeta\right) \le c_1^k\e^{n^2 m_n \bgn t\left(\e^{\theta_2}-1\right)}
\qquad (k=1,2,\cdots),
\end{equation*}
so, since $n^2 m_n \bgn \to \lg$ as $n \to \infty$, there exist $C_1>0$ and $n_1 \in \mathbb{N}$ such that, for all $n\ge n_1$,
\begin{equation}
\label{largedev1}
\P\left(\Xn_{k}(t)\ge k(\muyn+\zeta)\right) \le C_1 c_1^k \qquad (k=1,2,\cdots; 0 \le t \le T).
\end{equation}
Let $\mu_Y=\lim_{n \to \infty} \muyn = \lw \mu_I$, using~\eqref{infectionrates}.
Then, a similar argument shows that, for all $\zeta \in (0, \mu_Y)$, there exist $c_2 \in (0,1)$, $C_2>0$ and $n_2 \in \mathbb{N}$  such that, for all $n\ge n_2$,
\begin{equation}
\label{largedev2}
\P\left(\Xn_{k}(t)\le k(\muyn-\zeta)\right) \le C_2  c_2^k \qquad (k=1,2,\cdots; 0 \le t \le T).
\end{equation}

Suppose that $\gamma \in (0,\frac{\delta}{4})$.  Then, for any $\gamma' \in (0,\gamma)$, since $\ngam>\ngamd(\muyn+\zeta)$ for all sufficiently large $n$,~\eqref{largedev1} and condition (C1) imply that
\begin{equation}
\label{largedev3}
\lim_{n \to \infty} \sqrt{m_n} \sup_{0 \le t \le T}\P\left(X^{(n)}_{\ceil{\ngamd}}(t)>\ngam\right)=0.
\end{equation}
Similarly, for any $\gamma''>\gamma$, since for any $\zeta \in (0, \mu_Y)$, $\ngam \le \ngamdd(\muyn-\zeta)$ for all sufficiently large $n$,~\eqref{largedev2} and condition (C1) imply that
\begin{equation}
\label{largedev4}
\lim_{n \to \infty} \sqrt{m_n}\sup_{0 \le t \le T} \P\left(X^{(n)}_{\ceil{\ngamdd}}(t)< \ngam\right)=0.
\end{equation}
For $x >0$, let $\taun(x,t)=\min\{k \ge 0:\Xn_{k}(t) \ge x\}$.  Since $\Xn_{k}(t)$ is non-decreasing in $k$,
\begin{equation}
\label{largedev5}
\lim_{n \to \infty} \sqrt{m_n} \P\left(\taun(\ngam,t) \notin [\ngamd,\ngamdd] \right)=0.
\end{equation}
It follows immediately from~\eqref{largedev2} that, for all $n \ge n_2$,
\[
\P\left(\Xn_{k}(t)-k\le k(\muyn-\zeta-1)\right) \le C_2  c_2^k \qquad (k=1,2,\cdots; 0 \le t \le T).
\]
Note from~\eqref{R0def}, that $\mu_Y>1$, since $R_0>1$.  Thus, $\zeta$ can be chosen in $(0, \mu_Y-1)$.  Let $\gamma''' \in (0,\gamma')$.  Then, for such $\zeta$ and all sufficiently large $n$,  $\ngamddd<k(\muyn-\zeta-1)$ for all $k \ge \ngamd$, so there exists $n_3 \in \mathbb{N}$  such that, for all $n\ge n_3$,
\[
\P\left(\Xn_{k}(t)-k \le \ngamddd\right)\le C_2  c_2^{\ngamd} \qquad (k \ge \ngamd; 0 \le t \le T).
\]
Hence,
\[
\lim_{n \to \infty} \sqrt{m_n}\sup_{0 \le t \le T}\P\left(\Xn_{k}(t)-k< \ngamddd \mbox{ for some } k \in [\ngamd,\ngamdd]\right)=0,
\]
which, together with~\eqref{largedev5}, implies
\begin{equation}
\label{largedev6}
\lim_{n \to \infty} \sqrt{m_n}\P\left(\Xn_{k}(\taun(\ngam,t))-\taun(\ngam,t) < \ngamddd\right)=0.
\end{equation}

For $t \ge 0$, let $\Yn(\ngam,t)$ be the number of individuals alive in the above construction of $\Bn(t)$ when the total progeny first reaches $\ngam$, where $\Yn(\ngam,t)=0$ if $\Bn(t)<\ngam$.  Now
\[
\left\{\Yn(\ngam,t)< \ngamddd\right\}\cap \left\{\Bn(t) \ge \ngam\right\} \subseteq \left\{\Xn_{k}(\taun(\ngam,t))-\taun(\ngam,t) < \ngamddd\right\},
\]
so it follows from~\eqref{largedev6} that
\begin{equation}
\label{largedev7}
\lim_{n \to \infty} \sqrt{m_n}\sup_{0 \le t \le T}\P\left(\Yn(\ngam,t)< \ngamddd, \, \Bn(t) \ge \ngam \right)=0.
\end{equation}

A realisation of the total size $\Zn(t)$ of the epidemic $\En(t)$ can be constructed using the above random variables $\Yn_0(t),\Yn_1,\Yn_2,\cdots$ and the random variables $\chin_1, \chin_2,\cdots$ introduced near the start of Section~\ref{appB1}.  The construction proceeds in an analogous fashion to that used in Section~\ref{appB1}.  The $k$-th birth in the construction of $\Bn(t)$ (including the initial ancestors) is given the label is $\chin_k$ and corresponds to an infective in the construction of $\Zn(t)$ if and only if $\chin_k \notin \{\chin_1,\chin_2,\cdots,\chin_{k-1}\}$, otherwise that individual and its descendants are ignored in the construction of $\Zn(t)$.  As in the construction of $\Bn(t)$, the numbers of individuals infected by infectives are considered sequentially in the construction of $\Zn(t)$, yielding a process $\Xncheck_k(t)=\Yncheck_0(t)+\sum_{i=1}^k
\Yncheck_i$ $(k=0,1,\dots)$, where $\Yncheck_0(t)$ is the number of initial infectives in $\En(t)$ and for example, $\Yncheck_1$ is the number of people infected by the first infective considered in the construction of $\Zn(t)$.  Note that $\Zn(t)=\Xncheck_{\Wncheck}(t)$, where $\Wncheck=\inf\{k \ge 0: \Xncheck_k(t)-k=0\}$.  Let
\[
\Mn=\min\{k\ge 2:\chin_k \in \{\chin_1,\chin_2,\cdots,\chin_{k-1}\}\},
\]
as in Section~\ref{appB1}.  Then,
$\Xn_k(t)$ $(k=0,1,\cdots)$ and $\Xncheck_k(t)$ $(k=0,1,\cdots)$ coincide while $\Xn_k(t) < \Mn$.  Now $\lim_{n \to \infty}\sqrt{n}\P(\Mn \le \ngam)=0$ (cf.~\eqref{birthday}) and using~\eqref{largedev7}, a similar argument to the derivation of~\eqref{limit0} yields
\begin{equation}
\label{largedev8}
\lim_{n \to \infty} \sqrt{m_n}\sup_{0 \le t \le T}\P\left(\Ynhat(\ngam,t)< \ngamddd, \Zn(t) \ge \ngam \right)=0.
\end{equation}

Now $\png (t) = \P\left(\Zn (t) \ge \ngam \right)$ is increasing in $t$ and by Lemma~\ref{pn_znhat_convlemma}, $\png(t) \to p(t)$ as $n \to \infty$.  Hence, for any $T' \in(0, T)$, $\min_{T' \le t \le T}\png(t) >\frac{p(T')}{2}$ for all sufficiently large $n$.  In~\eqref{largedev8}, $\gamma'''<\gamma$ can be made arbitrarily close to $\gamma$.  Thus, it follows from~\eqref{largedev8} that, for any $\gamma'<\gamma<\frac{\delta}{4}$ and any $0<T'<T$,
\begin{equation}
\label{currentinf}
\lim_{n \to \infty}\sqrt{m_n}\sup_{T' \le t \le T}\P\left(\Ynhat(\ngam,t) < \ngamd \mid \Zn(t) \ge \ngam \right)=0.
\end{equation}

To obtain a uniform upper bound for $\qingamt$, first choose $\gamma_2 \in (0, \frac{\delta}{4})$, $\gamma' \in (0,\gamma)$ and then $\gamma_3 \in (0,1)$ such that $\gamma'+\gamma_3>1$. In view of~\eqref{currentinf}, we may assume that there are at least $\ngamd$ current infectives when the cumulative number of infectives
in $\En$ reaches $\ceil{\ngam}$.  It then follows from~\eqref{pavoidbounds} that
\begin{equation}
\label{qinbound}
\qingamt \le \P\left(\Sn \le \floor{\ngamthree}\right)+\theta_{N,a}(\ceil{\ngamthree}),
\end{equation}
where $N=n-\ceil{\ngam}$ and $a=\floor{\ngamd}$. Here, $\Sn=\left|\SSn\right|$, where $\SSn$ denotes the susceptibility set of a typical individual, $i$ say, in a population of total size $N+a=n-\ceil{\ngam}+\floor{\ngamd}$.  We get an upper bound for $\P\left(\Sn \le \floor{\ngamthree}\right)$ by first coupling the susceptibility set $\SSn$ to a lower bounding branching process, to obtain an
upper bound for $\P\left(\Sn \le \ngamTwo \right)$, and then showing that, as as $n \to \infty$,
$\P\left(\Sn \le \floor{\ngamthree} \mid  \Sn > \ngamTwo\right) \to 0$ sufficiently quickly.

In this final section we need the following additional notation to describe various branching processes and their properties. We write $\BBin(n,p)$ and $\BPo(X)=\BB(X)$ for Galton-Watson branching processes with a single ancestor and offspring distributions which are respectively $\Bin(n,p)$ (for $n \in \mathbb{N}$ and $p \in (0,1)$) and $\mbox{Po}(X)$. (We now add the `Po' subscript to the notation $\BB(X)$ for the latter in order to aid clarity.) Recall also the notation $\pipo(x)$ for the extinction probability of $\BPo(x)$.

Note that whilst the size of $\SSn$ is less than $\ngamTwo$, the number of individuals that could join the susceptibility set is at least
$n-\ceil{\ngam}-\ceil{\ngamTwo}$ (in fact at least $n-\ceil{\ngam}-\ceil{\ngamTwo}+\floor{\ngamd}-1$ but the coarser bound is sufficient for our purposes) and the probability that an individual joins the susceptibility set is at least $1-q_{\floor{\ngamTwo}}^{(n)}$, so
$\P\left(\Sn \le \ngamTwo \right) \le \P\left(\Un \le \ngamTwo \right)$, where $\Un$ is the total progeny of $\BBin(N',p_n)$ with
$N'=n-\ceil{\ngam}-\ceil{\ngamTwo}$ and $p_n=1-q_{\floor{\ngamTwo}}^{(n)}$. Let $\lambda_n=-N'\log(1-p_n)$. A realisation of $\BBin(N',p_n)$ can be obtained from one of $\BPo(\lambda_n)$ as follows.  Suppose that an individual in $\BPo(\lambda_n)$ has $k>0$ offspring. Thinking of those $k$ offspring as balls, place those $k$ balls independently and uniformly into $N'$ boxes; then in $\BBin(N',p_n)$ the number of offspring is the number of non-empty boxes.  Let $\Vn$ denote the total progeny of
$\BPo(\lambda_n)$.  Then (cf.~\eqref{birthday}), if $\Vn \le \ngamTwo$, the probability that $\BPo(\lambda_n)$ and $\BBin(N',p_n)$ differ is at most
$d_n=\ngamTwo(\ngamTwo-1)/(n-\ceil{\ngam}-\ceil{\ngamTwo})$ and $\lim_{n \to \infty} \sqrt{m_n} d_n =0$, since $\gamma_2<\frac{\delta}{4}$ and $\delta < 2$.  Thus we may use
$\P\left(\Vn \le  \ngamTwo\right)$ as an `upper bound' for $\P\left(\Sn \le \ngamTwo \right)$.

Observe that
\begin{equation}
\label{lamnbdaneq}
\lambda_n=-\left(n-\ceil{\ngam}-\ceil{\ngamTwo}\right)\left[\kappa_I\left(\bwn(1+\floor{\ngamTwo})\right)-\kappa_I\left(\bwn\floor{\ngamTwo}\right)\right].
\end{equation}
Omitting the details, using~\eqref{taylor}, we have $\lim_{n \to \infty}\sqrt{m_n}(\lambda_n-\mu_I \lw)=0$ if \newline (a) $\lim_{n \to \infty}\sqrt{m_n}n^{\gamma-1}=0$, (b) $\lim_{n \to \infty}\sqrt{m_n}n^{\gamma_2-1}=0$, (c)  $\lim_{n \to \infty}\sqrt{m_n}(n\bwn-\lw)=0$ and
(d) $\lim_{n \to \infty}\sqrt{m_n}n(\bwn)^2 n^{2\gamma_2}=0$.  Conditions (a) and (b) follow immediately from (C1) since $2\gamma<\delta$ and
$2\gamma_2<\delta$, and condition (c) is precisely (C3).  Condition (d) follows from $\lim_{n\to \infty}n\bwn=\lw$ (see~\eqref{infectionrates}) and $4\gamma_2<\delta$.  As at~\eqref{mnvnpi}, the same arguments as in the derivation of~\eqref{limit1} and~\eqref{limit3} in Section~\ref{appB1} now show that
\begin{equation}
\label{Vnngam2bound}
\sqrt{m_n}\left|\P\left(\Vn \le  \ngamTwo\right)-\pipo(\lw \mu_I)\right| \to 0 \qquad \mbox{as } n \to \infty.
\end{equation}

By bounding $\SSn$ between the processes $\BBin(N',p_n)$ and  $\BBin(N',p'_n)$, where $p'_n=1-q_0^{(n)}$, similar arguments to the derivation of~\eqref{currentinf} show that the susceptibility set $\SSn$ can be constructed so that if and when its size reaches $\ngamTwo$ then, for any $\gamma_4 \in (0,\gamma_2)$ there at least $\ngamfour$ members whose offspring have yet to be explored.  Until the size of $\SSn$ reaches $\ngamthree$, $\SSn$ is bounded below by the branching process
$\BBin(N(n),p(n))$, where $N(n)=n-\ceil{\ngam}-\ceil{\ngamthree}$ and $p(n)=1-q_{\floor{\ngamthree}}^{(n)}$.  It is easily shown using the mean value theorem that
$\lim_{n \to \infty}N(n)p(n)=\lw \mu_I$.  Thus, as $n \to \infty$, the offspring distribution of $\BBin(N(n),p(n))$ converges in distribution to
$\Po(\lw \mu_I)$, so the extinction probability, $\pi_n$ say, of $\BBin(N(n),p(n))$ converges to $\pipo(\lw \mu_I)$. Now $\pipo(\lw \mu_I)<1$, so there exists $\pi_0\in (\pipo(\lw \mu_I),1)$ and $n_0 \in \mathbb{N}$ such that $\pi_n<\pi_0$ for all $n \ge n_0$.  It follows that, for all $n \ge n_0$,
\begin{equation*}
P\left(\Sn \le \floor{\ngamthree} \mid  \Sn \ge \ngamTwo\right)\le (\pi_0)^{\ngamfour},
\end{equation*}
so, in view of~\eqref{qinbound} and the discussion just prior to~\eqref{lamnbdaneq}, we may take
\begin{equation}
\label{qugameq}
\qUgam= \P\left(\Vn \le  \ngamTwo\right)+(\pi_0)^{\ngamfour}+
\theta_{N,a}(\ceil{\ngamthree}).
\end{equation}
(Note that we cannot omit the term involving $\pi_0$ in~\eqref{qugameq} and replace $\gamma_2$ by $\gamma_3$ because conditions (b) and (d) following~\eqref{lamnbdaneq} may not hold with $\gamma_2$ replaced by $\gamma_3$.)

Now $\sqrt{m_n}(\pi_0)^{\ngamfour}\to 0 $ as $n \to \infty$, as $\pi_0 \in (0,1)$ and $m_n < n^2$ for all sufficiently large $n$.  Thus,
in view of~\eqref{Vnngam2bound}, to show~\eqref{upperlim} and thereby complete the proof of~\eqref{znhatconv} we just have to show that $\sqrt{m_n}\theta_{N,a}(\ceil{\ngamthree})\to 0 $ as $n \to \infty$.  Recall that $N=n-\ceil{\ngam}$ and $a=\floor{\ngamd}$.  Thus,
\begin{align*}
 \sqrt{m_n}\theta_{N,a}(\ceil{\ngamthree}) & \le \sqrt{m_n}\left(\frac{n}{n+\floor{\ngamd}}\right)^{\ceil{\ngamthree}}\\
 & =\sqrt{m_n}\left(1-\frac{\floor{\ngamd}}{n+\floor{\ngamd}}\right)^{\ceil{\ngamthree}}\\
 & \le \sqrt{m_n}\exp\left(-\frac{\floor{\ngamd}\ceil{\ngamthree}}{n+\floor{\ngamd}}\right) \to 0\qquad\mbox{as } n \to \infty,
\end{align*}
since $\gamma'+\gamma_3>1$ and $m_n < n^2$ for all sufficiently large $n$.
%\section*{???}%% if no title is needed, leave empty \section*{}.
%\end{appendix}
%%%%%%%%%%%%%%%%%%%%%%%%%%%%%%%%%%%%%%%%%%%%%%
%% Multiple Appendixes:                     %%
%%%%%%%%%%%%%%%%%%%%%%%%%%%%%%%%%%%%%%%%%%%%%%
%\begin{appendix}
%\section{???}
%
%\section{???}
%
\end{appendix}

%%%%%%%%%%%%%%%%%%%%%%%%%%%%%%%%%%%%%%%%%%%%%%
%% Support information, if any,             %%
%% should be provided in the                %%
%% Acknowledgements section.                %%
%%%%%%%%%%%%%%%%%%%%%%%%%%%%%%%%%%%%%%%%%%%%%%
\begin{acks}[Acknowledgments]
The authors thank an anonymous referee for their very careful reading of our paper and constructive comments, which have improved its presentation.
\end{acks}
%%%%%%%%%%%%%%%%%%%%%%%%%%%%%%%%%%%%%%%%%%%%%%
%% Funding information, if any,             %%
%% should be provided in the                %%
%% funding section.                         %%
%%%%%%%%%%%%%%%%%%%%%%%%%%%%%%%%%%%%%%%%%%%%%%
\begin{funding}
This work was partially supported by a grant from the Simons Foundation and was carried out as a result of the authors' visit to the Isaac Newton Institute
for Mathematical Sciences during the programme Theoretical Foundations for Statistical Network Analysis in 2016 (EPSRC Grant Number EP/K032208/1).
The third author was supported by Vetenskapsr{\aa}det (Swedish Research Council), grant 2016-04566.
This work was also supported by a grant from the Knut and Alice Wallenberg Foundation, which enabled the first author to be a guest professor at the Department of Mathematics, Stockholm University.

\end{funding}

%%%%%%%%%%%%%%%%%%%%%%%%%%%%%%%%%%%%%%%%%%%%%%
%% Supplementary Material, including data   %%
%% sets and code, should be provided in     %%
%% {supplement} environment with title      %%
%% and short description. It cannot be      %%
%% available exclusively as external link.  %%
%% All Supplementary Material must be       %%
%% available to the reader on Project       %%
%% Euclid with the published article.       %%
%%%%%%%%%%%%%%%%%%%%%%%%%%%%%%%%%%%%%%%%%%%%%%
%\begin{supplement}
%\stitle{???}
%\sdescription{???.}
%\end{supplement}

%%  MR numbers will be added by VTeX.                      %%
%%                                                         %%
%%  Use \cite{...} to cite references in text.             %%
%%                                                         %%
%%%%%%%%%%%%%%%%%%%%%%%%%%%%%%%%%%%%%%%%%%%%%%%%%%%%%%%%%%%%%

%% if your bibliography is in bibtex format, uncomment commands:
\bibliographystyle{imsart-nameyear}
\bibliography{publicationsA}

%% or include bibliography directly:
% \begin{thebibliography}{}
% \bibitem{b1}
% \end{thebibliography}

\end{document}